%% file: main.tex
\documentclass[11pt, a4paper]{article}
\usepackage[utf8]{inputenc}
\usepackage{xcolor}
\usepackage{array}
\usepackage{algorithm}
\usepackage{algorithmic}
\usepackage{fullpage}
\usepackage{enumitem}
\usepackage{mathtools}
\usepackage{amsmath,amsfonts,amssymb,amsthm}

\usepackage{enumitem}
\usepackage{mathtools}
\input{Macros.tex}
\usepackage{booktabs}
\usepackage{hyperref}
\usepackage{multirow}
\usepackage[caption=false]{subfig}
\usepackage{geometry}
\usepackage{enumitem}
\usepackage{booktabs}
\usepackage{url}
\usepackage{arydshln}
\usepackage[protrusion=true,expansion=true]{microtype}				
\theoremstyle{plain}
\newtheorem{theorem}{Theorem}
\newtheorem{proposition}[theorem]{Proposition}
\newtheorem{lemma}[theorem]{Lemma}
\newtheorem{corollary}[theorem]{Corollary}

\theoremstyle{definition}
\newtheorem{definition}[theorem]{Definition}

\newtheorem{remark}[theorem]{Remark}

\makeatletter
\long\def\REVERSEIT #1\\%
  {\@REVERSEIT
  #1&\@gobble&\@gobble&\@gobble&\@gobble&\@gobble&\@gobble&\@gobble&\@gobble\\}

\long\def\@REVERSEIT #1&#2&#3&#4&#5&#6&#7&#8\\%
  {#1&#8&#6&#3&#5&#2&#4&#7\\}
\makeatother

\newcommand{\intMetMod}[3]{\Delta_{#3}({#1}, {#2})}
\newcommand{\smean}[1]{\hat{\bbE}{#1}}
\newcommand{\lsmean}[2]{\hat{\bbE}_{#1}{#2}}
\newcommand{\sslen}{\SN{\CS}}
\newcommand{\ssler}{\SN{\CR}}
\newcommand{\lengamma}{\SN{\CI}}
\newcommand{\minSigma}{C_{\perp}}
\newcommand{\minLambda}{{\lambda}}
\newcommand{\suppmarg}{\textrm{supp}(\rho_X)}
\newcommand{\reach}{\tau_\gamma}
\newcommand{\highlight}[1]{\mathbf{#1}}
\newcommand{\curvtor}{{\kappa}}
\newcommand{\tempradius}{\eta}
\newcommand{\boundnormal}{B}
\newcommand{\Prpara}{P_{\CR}}
\newcommand{\Prperp}{Q_{\CR}}
\newcommand{\Ptan}{P}
\newcommand{\Pperp}{Q}
\newcommand{\sigmadiff}{C_W}
\newcommand{\summaryCrit}{\eta}
\newcommand{\summaryMax}{\theta}
\newcommand{\uniformC}{{c_V}}
\newcommand{\minSigmaY}{{\sigma_Y}}
\newcommand{\minSignal}{{\sigma_Y}} 

\newcommand{\minSigmaYcool}[1]{{\sigma_{{#1},Y}}}

\usepackage[affil-it]{authblk}

\date{}
\author[1]{\v Zeljko Kereta\thanks{Email: \texttt{zeljko@simula.no}} }
\author[1]{Timo Klock\thanks{Email: \texttt{timo@simula.no}} }
\author[2]{Valeriya Naumova\thanks{Email: \texttt{valeriya@simula.no}}}

\affil[1]{Simula Research Laboratory, Machine Intelligence Department, Oslo, Norway}
\affil[2]{SimulaMet, Machine Intelligence Department, Oslo, Norway}

\title{
		\usefont{OT1}{bch}{b}{n}
		\huge Nonlinear generalization of the monotone single index model\\
}

\begin{document}

\maketitle

\begin{abstract}
Single index model is a powerful yet simple model, widely used in statistics, machine learning, and other scientific fields.
It models the regression function as $g(\EUSP{a}{x})$, where $a$ is an unknown index vector and $x$ are the features.
This paper deals with a nonlinear generalization of this framework to allow for a regressor that uses multiple index vectors, adapting to local changes in the responses.
To do so we exploit the conditional distribution over function-driven partitions, and use linear regression to locally estimate the index vectors.
We then regress by applying a kNN type estimator that uses a localized proxy of the geodesic metric.
We present theoretical guarantees for estimation of local index vectors and out-of-sample prediction, and demonstrate the performance of our method with experiments on synthetic and real-world data sets, comparing it with state-of-the-art methods.
\end{abstract}

\section{Introduction}
\label{sec:introduction}

Many problems in data analysis can be formulated as learning a function from a given data set in a high-dimensional space.
Due to the curse of dimensionality, accurate regression on high-dimensional functions typically requires a number of samples that scales exponentially with the ambient dimension \cite{S98}.
A common approach to mitigating these effects is to impose structural assumptions on the data. 
Indeed, a number of recent advances in data analysis and numerical simulation are based on the observation that high-dimensional, real-world data is inherently structured, and that the relationship between the features and the responses is of a lower dimensional nature \cite{AC09}.

The most direct such model, which has become an important prior for many statistical and machine learning paradigms, considers a $1$-dimensional relationship of the form
\begin{equation}
\label{eq:sim}
Y = f(X)  + \varepsilon, \text{ for } f(X) = g(\EUSP{a}{X}),
\end{equation}
where $\varepsilon$ is a random noise term, and the features $X \in \bbR^D$ and responses  $Y \in \bbR$ are related
through an unknown index vector $a \in \bbR^D$ and an unknown monotonic function $g$.
Model \eqref{eq:sim} is called the \emph{single index model} (SIM), and it first appeared in economical and statistical communities in the early 90s \cite{h96,ichimura1993semiparametric}.
Moreover, SIM provides a basis for more complex models such as multi-index models \cite{samworth2016,li2005contour,HTF+05} and neural networks \cite{LBH15}.

An assumption shared by SIM and generalizations is that there is a single lower dimensional linear space that accounts for the complexity in relating $X$ and $Y$.
While simple, this assumption is only a first level approximation and is rarely observed in real-world regression problems.
The goal of this paper is to relax the assumption on global linearity in the model \eqref{eq:sim}, in order to locally adapt to changes in the relationship between $X$ and $Y$.
Specifically, we propose the \emph{nonlinear single index model} (NSIM), defined by
\begin{equation}
\label{eq:nsim}
Y=f(X) + \varepsilon, \text{ for } f(X) = g(\pi_{\gamma}(X)),
\end{equation}
where $\varepsilon$ is a random noise term, $g $ is a bi-Lipschitz function, $\gamma:\CI\rightarrow\bbR^D$ is a parametrization of a $\CC^2$ curve $\Im(\gamma)$, and $\pi_{\gamma}$ is the corresponding orthogonal projection, defined by
\begin{equation}
\label{eq:orthogonal_projection}
\pi_{\gamma}(x) \in \argmin_{z \in \Im(\gamma)}\N{x - z}.
\end{equation}
Function $g$ can be seen as a univariate scalar function, defined on the parametrization domain $\CI$, provided $\Im(\gamma)$ is a simple curve.
This identification is useful for defining examples of the setting, and reveals SIM as a special example of \eqref{eq:nsim}, where  $\gamma(t) = at$.

Before formally describing the assumptions and details of our approach, let us begin with a couple of comments.
Recall that smooth curves can be locally approximated by affine approximations, \emph{i.e.}, $\pi_{\gamma}(x) \approx \EUSP{a_j}{x}+c_j,$ where $a_j$ is a local tangent vector of $\Im(\gamma)$.
Problem  \eqref{eq:nsim} can therefore be approximated by a family of problems of the type
$f(x) \approx g_j\LRP{\EUSP{a_j}{x}}$, where $j$ corresponds to pieces of $\Im(\gamma)$ that are
approximately affine.
Notice now that due to the monotonicity of $g$, the proximity of $f(x)$ and  $f(x')$ implies the proximity of $\pi_{\gamma}(x)$ and $\pi_{\gamma}(x')$, and vice versa.
Therefore, instead of looking at approximately affine pieces of $\Im(\gamma)$, we can equivalently
consider a partition of $\Im(f)$, consisting of disjoint intervals $\mathcal{\CR}_j$, and split
\eqref{eq:nsim} into a family of localized SIM problems
\begin{equation}
\label{eq:local_sim_discrete}
\bbE[Y|X, f(X) \in \mathcal{\CR}_j] \approx g_j\left(\EUSP{a_j}{X}\right),\quad j=1,\ldots,J,
\end{equation}
where tangent vectors $a_j$ now play the role of index vectors in \eqref{eq:sim}.

\begin{figure}
\centering
\includegraphics[trim={0 1cm 0 3.25cm},clip, width=0.75\textwidth]{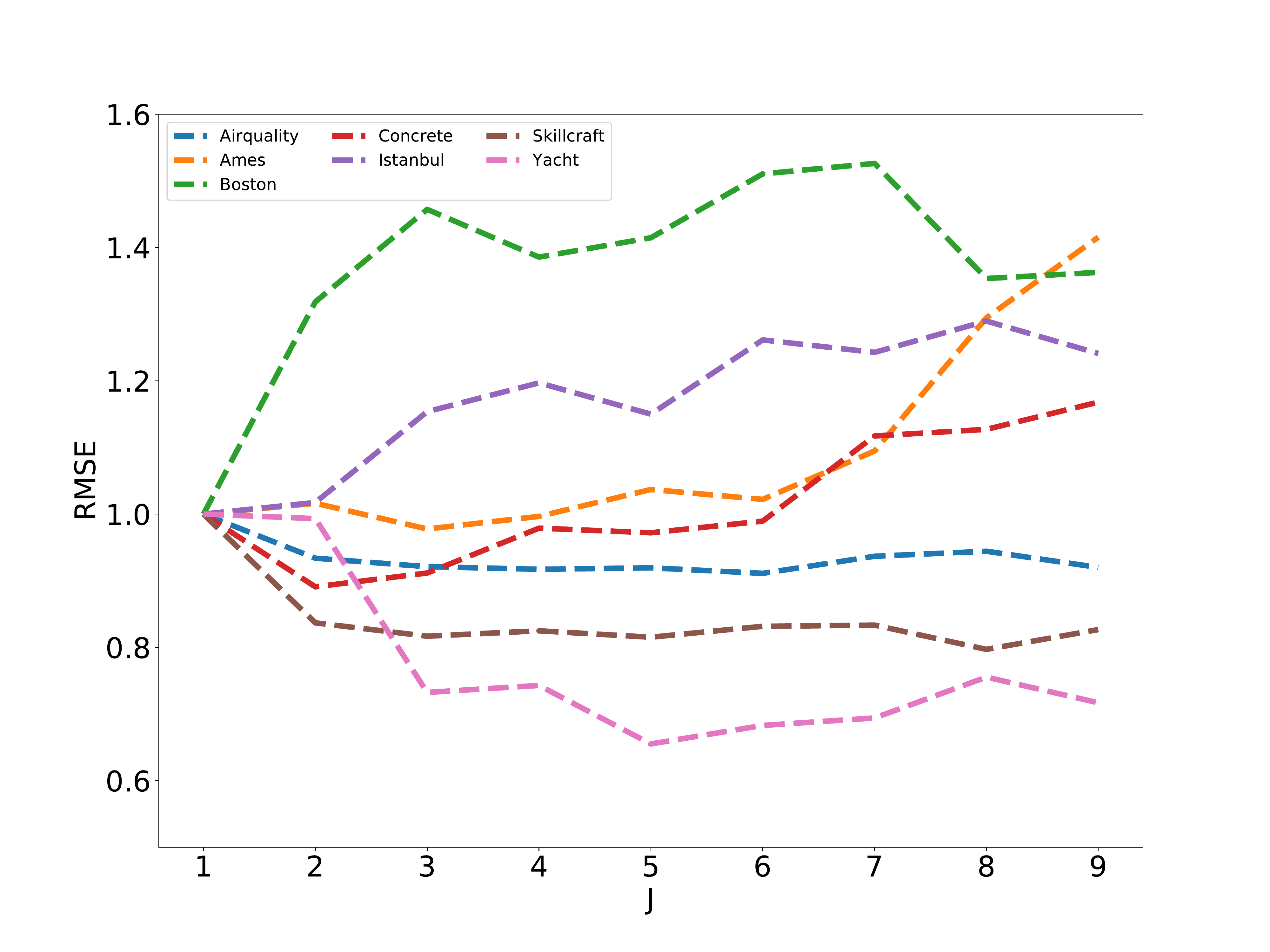}
\caption{The generalization error for  SIM ($J = 1$) and NSIM ($J>1$) on 7 UCI data sets.  We report the mean of
20 repetitions of the experiments for each $J$ and each data set.
To create a partition into level sets we construct statistically equivalent blocks based on the ordered $Y$-sequence.
The results for each data set are normalized so that the value at $J=1$ is $1.0$.
}
\label{fig:motivational_experiments}
\end{figure}

In Figure \ref{fig:motivational_experiments} we study the effects of such an approach on several UCI data sets\footnote{\url{https://archive.ics.uci.edu/ml/datasets.html}}.
Namely, for each data set we partition the data into $J$ sets, as detailed above,  learn a SIM estimator on each of the $J$ sets, and then plot the generalization error of the resulting estimator as a function of the hyperparameter $J$.
Given sufficient amount of data, we can observe that replacing \eqref{eq:sim} with \eqref{eq:local_sim_discrete}, its localized counterpart, often returns better estimation results.
For example, on the \texttt{Yacht} data set the generalization error improves by more than $30$ percent for $J=5$ compared to SIM.
Notice though that increasing the number of localized pieces does not always improve the performance.
This can mostly be attributed to the fact that splitting the original data set into disjoint subgroups reduces the number of samples within each group, which has a detrimental effect on the variance of the estimator.
In other words, we face a typical bias-variance trade-off, implying that hyperparameter $J$ needs to be carefully selected.
Furthermore, sometimes a SIM is indeed the best fit to the data (e.g. \texttt{Boston} data set). As shown in
the experiments in Section \ref{sec:experiments}, this will be detected
by our approach when combined with cross-validation to choose $J$.

\paragraph{Related work.}
\label{subsec:related_work}
To the best of our knowledge, relaxations of SIM have not yet been considered in this form.
However, three research areas are highly relevant:  linear single- and multi-index models, nonlinear sufficient dimension reduction, and manifold regression.
Below we provide a short overview of the most significant and relevant achievements in each of these fields.

\emph{Single- and multi index models} have been extensively researched, and we therefore, restrict ourselves to conceptually related work.
Most studies focus only on the estimation of index vector(s), which started with the early work on linear regression based methods \cite{brillinger2012generalized,han1987,li1989regression,sherman93}.
The most relevant work is \cite{hristache2001direct}, where the authors use iterative local linear regression to estimate the index vector $a$.
Locality is enforced by kernel weights, which are initially set to be spherical around the estimation point, and then iteratively
reshaped so that the isolines resemble level set boundaries of a strictly monotonous link function $g$.
This approach has been extended to the case of multiple index vectors \cite{dalalyan2008new}, estimating instead the corresponding index space.

Another relevant line of work are methods based on inverse regression that began
with the introduction of \emph{sliced inverse regression} (SIR) \cite{li1991sliced}.
This was followed by SAVE \cite{dennis2000save},  PHD \cite{li1992principal} , MAVE \cite{xia2002adaptive},
Contour regression \cite{li2005contour}, Directional regression \cite{li2007directional}, \emph{etc}.
The common thread shared by these methods is the use of inverse moments, such as $\bbE[X|Y]$ and $\Covv{X|Y}$, to estimate the index vector or the index space.

Several methods simultaneously learn the link function and the index vector.
We mention Isotron \cite{kalai2009isotron} and Slisotron \cite{kakade2011efficient},
which iteratively update the link function and the index vector; \cite{ganti2017learning} that  additionally
assumes sparsity of the index vector; \cite{cheng2017,kuchibhotla2016efficient} that use an iterative procedure and spline
estimates; \cite{radchenko2015high} that uses higher dimensional splines.

On the other hand, methods and theory for {\it nonlinear sufficient dimension reduction} are still in the early stages and there are many open questions.
Most of the existing studies consider kernelized versions of linear estimators (such as SIR or SAVE) to globally linearize the problem in feature space, and then apply well-known regression methods, see \cite{lee2013general,li2011principal, wu2008kernel,yeh2009nonlinear}.

Model \eqref{eq:nsim} can also be considered from the viewpoint of
{\it manifold regression,} where the goal is to estimate a function $f : \CM \rightarrow \bbR$ defined on the data.
Manifold regression methods, such as  \cite{bickel2007local, kpotufe2011k, mukherjee2010learning},
generally assume that the marginal distribution of $X$ is either supported on $\CM$ or in its close vicinity.
As a consequence, Euclidean distances can be used to locally approximate the geodesic metric.
This is a strong assumption which is implicitly or explicitly leveraged
by all manifold regression techniques, and presents a breaking point for their effective use.
In this work, we instead consider distributions that are spread in all directions of the ambient space around the curve $\gamma$.
Consequently, geodesic proximity cannot be inferred from Euclidean distances and we instead need to locally approximate the geodesic distance.

\paragraph{Main idea and estimation procedure for the NSIM model.}
\label{sec:algorithm}
Model \eqref{eq:nsim} increases the flexibility of the ordinary SIM
by allowing for varying index vectors, corresponding to different regimes of the response $f(X)$.
Consequently, a natural  approach would be to partition the data into several groups, based on $Y$,
and use a SIM-like estimator to approximate the index vector and the regression function.
In particular, our approach follows three steps.

In the \emph{first step} we partition the data set $\{(X_i,Y_i) : i \in [N]\}$ into $J$ sets, $\{\CX_j : j \in [J]\}$ and $\{\CY_j : j \in [J]\}$.
To do so we define a disjoint union of the responses, $\Im(Y) = \cup_{j=1}^{J} \CR_j$ for intervals $\CR_j$, and then set
\begin{equation}
\label{eq:level_set_requirement}
\CY_j := \CY \cap \CR_j,\quad \CX_j := \left\{X_i \in \CX : Y_i \in \CY_j\right\}.
\end{equation}
We refer to sets $\CX_j$ as \emph{level sets}, since they can be defined as $\CX_j = \CX \cap f^{-1}(\CR_j)$ in the noise-free case.
The optimal method for partitioning $\Im(Y)$ as $\cup_{j=1}^{J} \CR_j$ depends
on the marginal distribution of $Y$, and is best chosen after inspecting the empirical density.
For example, we suggest using dyadic cells of $[\min Y, \max Y]$ if the density of $Y$
is roughly uniform, and stochastically equivalent blocks
if the probability mass is unevenly distributed.

In the \emph{second step} we compute estimates  $\{\hat a_j : j \in [J]\}$ of local index vectors by using linear regression
on $\CX_j$ and $\CY_j$.
Namely, let $\hat{\Sigma}_j := \lsmean{\CX_j}{(X-\lsmean{\CX_j}{X})(X - \lsmean{\CX_j}{X})^\top}$
be the standard finite sample estimate for the conditional covariance $\Covv{X | Y \in \CR_j}$,
where $\hat\bbE$ denotes the empirical expectation.
Then, set $\hat a_j = \hat b_j/\Vert \hat b_j\Vert$, where $\hat b_j$ is the solution of linear regression,
\begin{equation}
\label{eq:tangent_argmin}
\hat{b}_j := \argmin\limits_{P_{\ker(\hat \Sigma_j)} \omega = 0} \lsmean{(\CX_j,\CY_j)}{\left(Y - \lsmean{\CY_j}{Y} - \left\langle \omega, X - \lsmean{\CX_j}{X}\right\rangle\right)^2},
\end{equation}
or equivalently,
\begin{equation}
\label{eq:tangent_direct}
\hat b_j := \hat \Sigma_j^{\dagger}\,\lsmean{(\CX_j,\CY_j)}{\LRP{(Y- \lsmean{\CY_j}{Y})(X - \lsmean{\CX_j}{X})}}.
\end{equation}
Intuitively, vectors $\hat a_j$ correspond to directions in which the function changes, and therefore approximates
local gradient directions of $f$. In the case of an ordinary SIM, it has been shown in \cite{balabdaoui2019score}
that the direction of the global linear regression vector, denoted by $\hat a$, is an unbiased estimator of
index vector $a$, if $X$ has elliptical distribution. Furthermore, $\sqrt{N}(\hat a - a)$
is asymptotically normal, hence $\hat a$ achieves $N^{-1/2}$-consistency.
As we will see in Sections \ref{sec:model} and \ref{sec:geometry}, in our case the analysis of $\hat a_j$ is more
challenging due to the underlying nonlinear geometry.

\begin{algorithm}[b!]
\caption{Summary of the NSIM Estimator}
\label{alg:main}
\vspace{0.2cm}
\textbf{Learning local index vectors}
\begin{algorithmic}
  \REQUIRE $\{(X_i, Y_i) : i \in [N]\}$, $J$
  \STATE Split data into $\{\CX_j : j \in [J]\}$ and $\{\CY_j : j \in [J]\}$ according to \eqref{eq:level_set_requirement}
  \FOR{$j=1,\ldots,J$}
    \STATE $\hat{b}_j = \hat \Sigma_j^{\dagger}\,\lsmean{(\CX_j, \CY_j)}{\LRP{(Y-\lsmean{\CY_j}{Y})(X - \lsmean{\CX_j}{X})}}$
    \STATE $\hat{a}_j = \hat{b}_j/\Vert \hat b_j\Vert$
  \ENDFOR
\ENSURE $\hat a_j$ for $j \in [J]$
\end{algorithmic}
\vspace{0.2cm}
\textbf{Out-of-sample prediction}
\begin{algorithmic}
  \REQUIRE sample $x$, sets $\{\CX_j : j \in [J]\}$, $\{\CY_j : j \in [J]\}$, index vectors $\hat{a}_j$, parameters $k$, $\tempradius$
  \STATE Compute nearest neighbor ordering $1(x),\ldots,k(x)$ based on $\intMetMod{x}{\cdot}{\tempradius}$
\ENSURE  $\hat{f}_{k}(x) = k^{-1}\sum_{i=1}^{k}Y_{i(x)}$
\end{algorithmic}
\end{algorithm}

In the \emph{final step} we use a kNN-type estimator to predict $f(x)$ for an out-of-sample $x$.
Since the make-or-break point of kNN-estimators regards how are distances between $x$ and training samples measured, the critical point of this step is about the selection of an appropriate distance function.
The issue is that the optimal choice (the geodesic metric on $\Im(\gamma)$) is not available since $\Im(\gamma)$ is not known, and the naive choice (the Euclidean metric) generally leads to estimation rates that depend on the ambient dimension, and thus the curse of dimensionality.

To develop a proxy metric, consider now the ordinary SIM.
Here the geodesic metric is equivalent to the Euclidean distance of projected samples if $\hat a$ approximates the true index vector $a$ with a sufficiently high rate, \emph{i.e.}, $\SN{\EUSP{\hat a}{(x-x')}}$ is a good proxy for the geodesic metric provided $\N{\hat a - a}$ is small.
Moreover, training the kNN estimator on
projected samples $(\EUSP{\hat a}{X_i},Y_i)$ achieves optimal univariate regression rates.
The NSIM case is more challenging because first, we have $J$ different index vectors to choose
from, and second, $x$ cannot be \emph{a priori} assigned to any level set since $f(x)$ is unknown.
Still, if we assign to each sample $X_i$ the index vector $\hat a(X_i) := \hat a_{j(X_i)}$,
where $j(X_i)$ is the unique level set with $X_i \in \CX_{j(X_i)}$,
we can show that
\begin{equation}
\label{eq:geodesic_approximation}
\intMetMod{x}{X_i}{\tempradius}:=\begin{cases}
  \SN{\hat a(X_i)^\top(x - X_i)}& \textrm{if } \N{x-X_i} \leq \tempradius,\\
  \infty &\textrm{else},
\end{cases},
\end{equation}
approximates the geodesic metric $d_{\gamma}(\pi_{\gamma}(x), \pi_{\gamma}(X_i))$ reasonably well,
under suitable choice of the \emph{restricting radius} $\tempradius$,
see Section \ref{subsec:function_estimation_general_curves}.
In the special case of a perturbed SIM, where $\gamma$ is not too far from an affine space,
this is also true for $\tempradius = \infty$, 
see Section \ref{subsec:function_estimation}.

This motivates the following estimator: let $(X_{i(x)}, Y_{i(x)})$ denote the $i$-th closest sample
to $x$ when measured in $\intMetMod{x}{\cdot}{\tempradius}$, and where ties can be broken arbitrarily.
Then set
\begin{equation}
\label{eq:prediction}
\hat{f}_{k}(x) := \frac{1}{k} \sum\limits_{i=1}^{k} Y_{i(x)}.
\end{equation}
As we will discuss in Section \ref{subsec:function_estimation_general_curves}, the radius $\tempradius$ plays a dual role.
It needs to be large enough so that there are enough samples to choose neighbors from, but small enough so that \eqref{eq:geodesic_approximation} is a good proxy for the geodesic metric.
The entire estimation approach is summarized in Algorithm \ref{alg:main}.

\paragraph{Computational complexity.} The first two
steps, partitioning and computing tangents, are dominated by $\CO(\min\{JD^3,JND^2\} + ND^2)$, which is mostly due to
forming covariance matrices and computing the generalized inverse. Out-of-sample
prediction requires $\CO(N + JD)$ operations per evaluation.

\paragraph{Contributions and organization of the paper.}
In this work we introduce a nonlinear generalization of the SIM and study estimation of the model from $N$ given
data points $\{(X_i,Y_i) : i \in [N]\}$ sampled iid. from an unknown distribution $\rho$.
The presented model synthesizes the fields of linear sufficient
dimension reduction and manifold regression, thereby attempting to extend both.
We first develop a rigorous mathematical framework, in Section \ref{sec:model}, through which NSIM can be theoretically analyzed.

We provide a simple and efficient estimator based on output-conditional linear regression
and kNN-regression. Theoretical guarantees of the approach
are subjects of Sections \ref{sec:geometry} (local index vectors) and \ref{sec:function_estimation} (function estimation).
In summary, we achieve optimal estimation rates \cite{gyorfi2006distribution,kohler2014optimal} in the noise-free scenario ($\varepsilon = 0$ almost surely), or if the data follows
the ordinary SIM. In the general case, the estimator remains biased.

The theoretical analysis on local index vector (or tangent field) estimation requires a careful study
of (conditional) ordinary linear regression \eqref{eq:tangent_argmin}. In particular,
two sources of error are present: a bias term, that decays when increasing the number $J$ of subsets in the level set partition,
and a variance term, that decays with the number of samples $N$. Our analysis
reveals a concentration bound of the form
\begin{align*}
\max_{i \in [N]}\N{\hat a(X_i) - \gamma'(\gamma^{-1} \circ \pi_\gamma(X_i))}\lesssim \frac{\kappa}{J} + \frac{\log(J)}{\sqrt{N J}},
\end{align*}
where $\kappa$ is a curvature bound for $\Im(\gamma)$. This is a surprising result
because both the bias and the variance decrease with $J$ (as long as the noise $\varepsilon$
is negligible compared to the $J^{-1}$). This observation is a key component for establishing
optimal regression rates in the noise-free case.

For the regression analysis, we show in Section \ref{sec:function_estimation}
that $\intMetMod{x}{\cdot}{\tempradius}$ is equivalent to the geodesic metric $d_{\gamma}(\pi_{\gamma}(x),\cdot)$,
up to an error made in tangent field estimation. This suffices to establish aforementioned kNN-regression guarantees.
These results are relevant from a more general perspective, because they can readily be used with other means of estimating the tangent field,
and can be extended to higher dimensional manifolds.

In Section \ref{sec:experiments} we conclude the paper with extensive numerical tests
on synthetic and real data sets, that have previously
been used as benchmarks for the SIM model. The results show that the extended flexibility
of NSIM is beneficial for both, out-of-sample prediction and model interpretability.

\paragraph{General notation.}
We use $[N] = \{1,\ldots,N\}$ for $N \in \bbN$. $\N{\cdot}$ denotes the Euclidean norm for vectors, and the spectral norm
for matrices. $d_{\gamma}$ denotes the geodesic metric on $\Im(\gamma)$. Provided that $\gamma$ is an arc-length parametrization, this means $
d_\gamma(\gamma(t_1),\gamma(t_2))=\SN{t_1-t_2}$. We extend the notation to $x,x'\in\bbR^D$ by setting $d_{\gamma}(x,x') := d_{\gamma}(\pi_{\gamma}(x), \pi_{\gamma}(x'))$ whenever projections
$\pi_{\gamma}(\cdot)$ are uniquely defined.
For a discrete set of points $A=\{x_1,\ldots,x_k\}\subset \bbR^D$ we use $\SN{A}$ to denote its number of elements.
On the other hand, if $A$ is a connected subsegment of $\Im(\gamma)$ or if $A\subset \CR$ is an interval, then $\SN{A}$ denotes its length.
By an \emph{interval}  $A\subset \bbR$ we always refer to a closed and connected subset of the real line.
We use $a\vee b = \max\{a,b\}$ and $a \wedge b =\min\{a, b\}$.
The Moore-Penrose inverse of a matrix $M$ is denoted by $M^{\dagger}$.

The abbreviation \emph{a.s.} is used as a shorthand for \emph{almost sure} events
(with respect to implicit random vectors), and \emph{iid.} refers to
\emph{independent and identically distributed} data sampling.
Table \ref{tab:Notation} contains an overview of notation and constants used in this paper.

\begin{table}[!htbp]
\begin{center}
\scriptsize
\begin{tabular}{@{}cc@{}}
      symbol& definition\\ \toprule
\textsf{geometry} & \\\midrule
$\gamma,\, \Im(\gamma)$ & $\gamma:I\subset\bbR\rightarrow\bbR^D$ is the parametrization of $\Im(\gamma)=\gamma(I)$ \\
$\pi_\gamma$ & the orthogonal projection onto $\Im(\gamma)$, see \eqref{eq:orthogonal_projection} \\
$\reach$ & $\sup_{r>0}\left\{\forall x\in\bbR^D\setminus\Im(\gamma) \text{ s.t. } \dist{x}{\Im(\gamma)}<r \, \exists!z\in\Im(\gamma) \text{ s.t. } \dist{x}{z} = \dist{x}{\Im(\gamma)}\right\}.$ \\
$d_\gamma(v,v')$ & geodesic distance for $v,v'\in\Im(\gamma)$, extended by $d_{\gamma}(x,x') := d_{\gamma}(\pi_{\gamma}(x),\pi_{\gamma}(x'))$\\
$\CB_m(x,R)$ & ball of radius $R$ around a point $x$, with respect to a metric $m$ \\
$\kappa$ & bound for the curvature of $\gamma$, \emph{i.e.} $\kappa=\N{\gamma''}_\infty$ \\\hdashline
\multirow{2}{*}{$\Prperp,\, \Prpara$} & projections onto the tangent/normal space at $\overline{t}_{\CR}=\bbE[t|Y\in\CR]$ \\
& here $\Prpara = \gamma'(\overline{t}_{\CR})\gamma'(\overline{t}_{\CR})^\top$ and $\Prperp = \Id - \Prpara$ \\
\midrule
      \textsf{probability} & \\\midrule
      $(X, Y)$ & random vector in $\bbR^D\times \bbR$ with a distribution $\rho$,  and the marginal of $X$ is $\rho_X$\\
      $V, W$ & random vectors such that $X=V+W$, where $V=\pi_{\gamma}(X)\in\Im(\gamma)$ \\
     $\bbE X,\, \Covv{X}$&  the expectation and the covariance of a random variable $X$ \\
$\smean{X},\,\hat \Sigma$ & empirical mean and sample
covariance over \textit{all} samples\\
$\bbE[V|\CR],\, \Covv{X|\CR}$ & shorthand for conditional mean $\bbE[V|Y \in \CR]$ and conditional covariance $\Covv{X|Y \in \CR}$\\
$\lsmean{\CU}{X},\, \hat\Sigma_\CU $ & mean, and covariance,
over samples that belong to $\CU$;
\( \hat\bbE_\CU X = \frac{1}{\SN{\CU}} \sum_{X\in\CU} X\)    \\[0.5ex] \midrule
\textsf{constants} & \\\midrule
$L_f$ & the bi-Lipschitz constant $L_f$ of the function $g$, see \eqref{eq:g_bilipschitz} \\
$J$ & number of level sets, \emph{i.e.} the size of the partitioning of the data; $\CX=\cup_{j=1}^J \{\CX_j\}$, see \eqref{eq:dyadic_partitioning} \\
$\sigma_\varepsilon$ & bound on the noise term $\varepsilon$, \emph{i.e.}, $\SN{\varepsilon}\leq\sigma_\varepsilon$,  where $Y=f(X)+\varepsilon$, see \ref{ass:A5}\\
$C_W$ & constant in bounding influence of cross-covariance, see \ref{ass:A42} \\
$\minSigma$ & lower-bound for non-zero eigenvalues in directions normal to $\gamma$, see \ref{ass:A41} \\
$B$ & bound for $\dist{X}{\Im(\gamma)}$, see \ref{ass:A2} \\
$\uniformC$ & uniformity constant for the distribution along $\Im(\gamma)$, see \ref{ass:A3}
\\\bottomrule
\end{tabular}
\end{center}
\caption{Summary of the notation used in the paper}
\label{tab:Notation}
\end{table}


\section{Theoretical framework for the NSIM model}
\label{sec:model}

Due to the broadness of its scope, it is relatively easy to construct examples of
NSIM that fit the model but for which estimation from finite samples is not possible.
The goal in this section is to define a framework that allows a rigorous analysis,
yet is broad enough to encompass both the SIM and its nonlinear generalization NSIM.
In the following we describe the assumptions on the function class, the
{underlying nonlinearity} $\Im(\gamma)$, and on the distribution of the data set.

\paragraph{Regularity assumptions for $f$ and $\Im(\gamma)$.}
Let  $\gamma:\CI\rightarrow \bbR^D$, for an interval $\CI\subset \bbR$, be an arc-length parametrization of a simple, connected, and $\CC^2$ smooth curve, denoted $\Im(\gamma)=\gamma(\CI)$, and set $\kappa={\N{\gamma''}}_\infty<\infty$. We consider Lipschitz functions $f:\Omega\subset \bbR^D \rightarrow \bbR$ that satisfy $f(x) = g(\pi_{\gamma}(x))$ for some ${L_f}$-bi-Lipschitz function $g : \Im(\gamma)\rightarrow \bbR$, that is
\begin{equation}\label{eq:g_bilipschitz} L_f^{-1}d_\gamma(v,v') \leq \SN{g(v)- g(v')} \leq L_f d_\gamma(v,v'), \text{ for all } v,v'\in\Im(\gamma).\end{equation}
Through rescaling we can always assume $\Im(f) = [0,1]$.
We can, without loss of generality, align $\gamma$ with $\nabla f$, \emph{i.e.}, choose an orientation such that $\left\langle \nabla f(\gamma(t)), \gamma'(t)\right\rangle > 0$, for almost every $t\in\CI$.
An important quantity is the reach $\reach$ of $\Im(\gamma)$ - the largest $r>0$ such that any point at distance less than $r$ from $\Im(\gamma)$ has a unique nearest point on $\Im(\gamma)$ \cite{federer1959curvature}. This ensures that $\pi_\gamma(x)$, and thus $f(x)$, is well defined for all $x$ within the reach, \emph{i.e.}, all $x$ such that $\min_{z \in \Im(\gamma)}\N{x - z} < \reach$.
\paragraph{Distributional assumptions.}
We consider distributions $\rho$ for which the distribution of $X|Y \in \CR$ is absolutely continuous with respect to
the Lebesgue measure on $\Im(\Covv{X|\CR})$ for any non-empty
interval $\CR \subset [0,1]$, and which satisfy assumptions \ref{ass:A5} - \ref{ass:A3} below.

Assumptions \ref{ass:A5} - \ref{ass:A41} are related to single- and multi-index model literature
(or more broadly sufficient dimension reduction literature, see \cite{ma2013review} for a review), whereas \ref{ass:A2} - \ref{ass:A3} are related to manifold regression.
We begin by describing the behavior of the noise $\varepsilon$.
\begin{enumerate}[label=({A\arabic*})]
\setcounter{enumi}{0}
\item\label{ass:A5} For $\varepsilon := Y - \bbE[Y|X] = Y - g(\pi_{\gamma}(X))$, we assume
$\varepsilon \independent X | \pi_{\gamma}(X)$, and $\SN{\varepsilon}\leq \sigma_{\varepsilon}$ \emph{a.s.}.
\end{enumerate}
In sufficient dimension reduction problems, $\varepsilon \independent X | \pi_{\gamma}(X)$ is
often more commonly written $Y \independent X|\pi_{\gamma}(X)$.

The next assumption states that $\Im(\gamma)$ is centered {in the middle of the distribution}.
\begin{enumerate}[label=({A\arabic*})]
  \setcounter{enumi}{1}
 \item\label{ass:A1} $\bbE[X|\pi_{\gamma}(X)] = \pi_{\gamma}(X)$ holds $\pi_{\gamma}(X)$-\emph{a.s.}
\end{enumerate}
{This is inspired by the \emph{linear condition mean} assumption from single- and multi-index model literature, and is an integral component of every method based on inverse regression \cite{ma2012semiparametric,ma2013review}.
It is needed to ensure the recovery of a subspace of the index space
in the population regime $(N\rightarrow \infty)$, see \emph{e.g.} \cite{dennis2000save,li1991sliced},
and is often ensured by a stronger condition: if $X$ is elliptically distributed \cite{ma2012semiparametric}.
\ref{ass:A1} also implies identifiability of $\Im(\gamma)$ by the distribution of $(X,Y)$.}

\begin{lemma}
\label{lem:uniquely_identifiable}
Let $\CD, \CD' \subset \bbR^D$, with orthogonal projections $\pi_{D}, \pi_{D'}$ defined according to \eqref{eq:orthogonal_projection}, and let
$X$ be a random vector such that $\pi_{D}(X)$ and
$\pi_{D'}(X)$ are \emph{a.s.} unique.
Let $g:\CD \rightarrow \bbR$ and $g' : \CD' \rightarrow \bbR$ be measurable and injective.
If $f = g\circ\pi_{\CD} = g'\circ \pi_{\CD'}$, and $\bbE[X - \pi_{\CD}(X)|\pi_{\CD}(X)] =
\bbE[X - \pi_{\CD'}(X)|\pi_{\CD'}(X)]$ \emph{a.s.}, then $\pi_{\CD}(X) = \pi_{\CD'}(X)$ \emph{a.s.}.
\end{lemma}
\begin{proof}
Due to the assumption we have
\begin{equation}
\label{eq:identification_lemma}
\pi_{\CD}(X) - \pi_{\CD'}(X) =  \bbE[X\lvert \pi_{\CD(X)}] - \bbE[X\lvert \pi_{\CD'(X)}], \text{ \emph{a.s.}}
\end{equation}
\noindent Since conditioning on an injective function of a random variable is {equivalent} with conditioning on the random variable itself, we get
\[
\bbE[X\lvert f(X)] = \bbE[X\lvert g(\pi_{\CD'}(X))] = \bbE[X\lvert \pi_{\CD'}(X)],
\]
and similarly for $\pi_{\CD}(X)$. Plugging into
\eqref{eq:identification_lemma} the claim follows.
\end{proof}

In the linear case \ref{ass:A5} and \ref{ass:A1}
imply $\Covv{PX, QX|\CR} = 0$ for any interval $\CR \subset [0,1]$, where $P$ is the orthoprojector onto the index space, and $Q = \Id - P$. For some
single- or multi-index model estimators this suffices to ensure the recovery of the index space in the population regime, see \emph{e.g.} \cite{li1991sliced}.
In the nonlinear case however, due to curvature we require an additional
assumption. Let $t := \gamma^{-1} \circ \pi_{\gamma}(X)\in \CI$ be the induced random variable and define the mean $\bar t_{\CR}:=\bbE[t|\CR]$,
the tangential projection $\Prpara := \gamma'(\bar t_{\CR})\gamma'(\bar t_{\CR})^\top$, and
the orthogonal projection $\Prperp := \Id - \Prpara$, see Figure \ref{fig:notation_graphic}.
Furthermore, let $\CS \subset \Im(\gamma)$ be the shortest connected segment with $\bbP(V \in \CS|\CR)=1$.
\begin{enumerate}[label=({A\arabic*})]
  \setcounter{enumi}{2}
\item\label{ass:A42} There exists an absolute constant $\sigmadiff>0$ such that $\N{\Covv{\Prperp X, \Prpara X\lvert \CR}}\leq \kappa \sigmadiff \SN{\CS}^2.$
\end{enumerate}
Due to other assumptions, \ref{ass:A42} trivially holds if $\SN{\CS}^2$ is replaced by $\SN{\CS}$, though we need more regularity.
Namely, our analysis shows that replacing $\SN{\CS}^2$ with $\SN{\CS}^{1+\alpha}$, for some $\alpha \geq 0$,
approximations of the tangent field are valid only if $\kappa \SN{\CS}^{\alpha}$ falls below a certain threshold.
Thus, for $\alpha = 0$ this restricts the analysis to only SIMs and curves with small curvature.
We select $\alpha=1$ for the sake of notational simplicity, though the results are valid for any $\alpha>0$.

\begin{figure}
\centering
\includegraphics[width=0.6\linewidth]{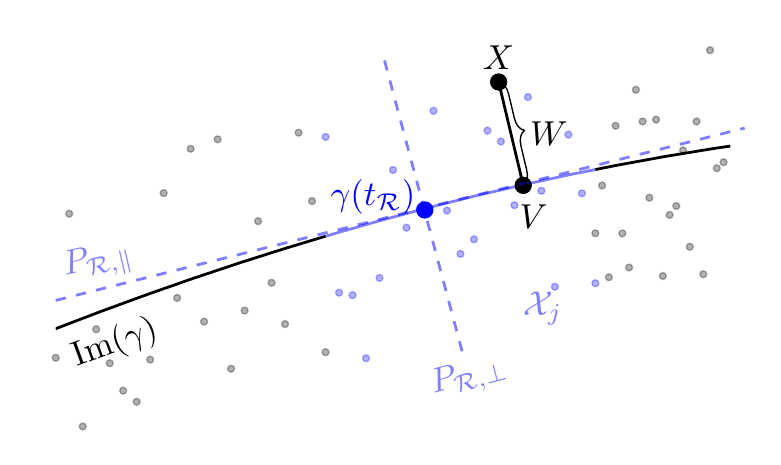}
\caption{A point $X$ can be represented by a component on $\Im(\gamma)$, given by $V:=\pi_{\gamma}(X)$, and a component orthogonal to $\Im(\gamma)$, given by $W:=X - \pi_{\gamma}(X)$. Each training sample belongs to a level set $\CX_j$ with \emph{midpoint} $\gamma(\bar t_{\CR})$, where $\bar t_{\CR} = \bbE[\gamma^{-1}\circ \pi_{\gamma}(X)|\CR]$, and each level set is associated with tangent projection $\Prpara = \gamma'(\bar t_{\CR})\gamma'(\bar t_{\CR})^\top$ and normal projection $\Prperp:=\Id - \Prpara.$}
\label{fig:notation_graphic}
\end{figure}

Our fourth assumption describes the behavior of $X$ orthogonal to the curve.
\begin{enumerate}[label=({A\arabic*})]
  \setcounter{enumi}{3}
  \item\label{ass:A41} For all $v \in \Im(\Covv{X|\CR}) \cap \Im(\Prperp)$, with $\N{v}=1$, we have \[v^\top \Covv{X|\CR} v > \minSigma > 0.\]
\end{enumerate}
In the nonlinear case an assumption of this form is necessary in order to ensure that
the solution of local linear regression aligns with the local tangent vector instead of the local curvature vector.
This is also observed numerically, where if the variance vanishes, as a function of $\CR$,
a vector close to a local curvature vector can minimize \eqref{eq:tangent_argmin}.
Such an assumption has also been used for multi index models, see \cite{dennis2000save,li2007directional,li2005contour}, though assuming \ref{ass:A5} and \ref{ass:A1}
would suffice in our case to ensure that the linear regression vector (for any conditioning on $\CR \subset [0,1]$) is contained in the index space.

The last two assumptions deal with properties of the distribution along the curve $\gamma$, denoted by $V$, and with components orthogonal to it, denoted by $W$.

\begin{enumerate}[label=({A\arabic*})]
\setcounter{enumi}{4}
\item\label{ass:A2}$W := X - \pi_{\gamma}(X)$, the component of $X$ orthogonal to $\Im(\gamma)$, satisfies $\N{W} \leq \boundnormal < \reach$, $W$-a.s.
\end{enumerate}
An assumption of this type is needed due to the fact that the projection $\pi_\gamma(X)$, and consequently the function $f$, is not always well defined for $\N{W}\geq\reach$.
In case of a straight line we have $\reach=\infty$, and thus there are no restrictions on $W$ (which reflects standard SIM assumptions).
On the other hand, \ref{ass:A2} is a relaxation of standard assumptions in manifold regression, which require samples $X$ to lie on, or very near the manifold, \emph{i.e.}, $\N{W}=0$ or $\N{W}\ll \reach$.

Lastly, we assume that the data distribution along the curve does not deviate too much from a uniform distribution.
This is used in manifold regression approaches that approximate the manifold by localization
and linearization, as it ensures that local pieces are sufficiently well covered, see e.g. \cite{liao2016learning}.
\begin{enumerate}[label=({A\arabic*})]
\setcounter{enumi}{5}
\item\label{ass:A3} For random vectors  $V := \pi_{\gamma}(X)\in\Im(\gamma)$ there exists $\uniformC\!>\!0$ such that $c_V^{-1}\SN{\CS}\SN{\CI}^{-1}\! <\! \bbP(V\! \in\! \CS)\! <\! \uniformC\SN{\CS}\SN{\CI}^{-1}$ holds for any $\CS\!\subset\!\Im(\gamma)$.
\end{enumerate}

A comparison of assumptions \ref{ass:A5}-\ref{ass:A3} with standard assumptions in the
literature, and their implication in case of the SIM, is provided in Table \ref{tab:NSIMvsAssumptions}.

\begin{table}[!htbp]
\begin{center}
\scriptsize
\begin{tabular}{@{}m{0.1cm}m{4.5cm}m{7.5cm}@{}}
    NSIM& implication on SIM  & comparable assumption in the literature \\ \midrule
    \ref{ass:A5} & $Y=f(a^\top x) + \varepsilon,\ \varepsilon\independent X\lvert a^\top X$ & the setting is often studied in SIM literature, e.g. in \cite{hristache2001direct,radchenko2015high}\\[1ex]\midrule
    \ref{ass:A1} & $\bbE[X| P X]=PX$ for $P=aa^\top$ & integral part for inverse regression based techniques, usually implied by ellipticity, e.g. \cite{li1991sliced,ma2012semiparametric,ma2013review}\\[1ex]\midrule
    \ref{ass:A42} & implied by \ref{ass:A5} and \ref{ass:A1}& - \\[1ex]\midrule
    \ref{ass:A41} & $v^\top\Covv{X\lvert \CR}v^\top>\minSigma$ for all $v\perp a$, $\N{v}=1$ & implied by the \emph{constant conditional covariance assumption} used sometimes for sufficient dimension reduction, \emph{e.g.} \cite{dennis2000save,li2007directional,li2005contour} \\[1ex]\midrule
    \ref{ass:A2} &  there exists $B>0$ such that $\N{X}\leq B<\infty$ & existing methods require $B=0$ to prove regression rates that do not depend exponentially on $D$, e.g. \cite{bickel2007local,kpotufe2011k}\\[1ex]\midrule
    \ref{ass:A3} & $a^\top X$ is absolutely continuous with respect to the Lebesgue measure on the image of $a^\top X$ & this is common to ensure that the manifold is covered well enough, e.g. \cite{liao2016learning} \\[1ex]
\bottomrule
\end{tabular}
\end{center}
\caption{Comparison of NSIM assumptions \ref{ass:A5}-\ref{ass:A3} with assumptions in SIM and manifold regression theory. Here $a$ denotes the (unit) index vector in SIM.
Assumptions \ref{ass:A5} - \ref{ass:A41}, are common in the study of linear sufficient dimension reduction, whereas \ref{ass:A2} - \ref{ass:A3} reflect the constraints imposed by the non-linearity of the setting, and are common in manifold regression problems.
We add though that \ref{ass:A2} is a significant relaxation of standard assumptions in manifold regression, which require $B=0$ or $B\ll\reach$.}
\label{tab:NSIMvsAssumptions}
\end{table}



\section{Learning localized index vectors}
\label{sec:geometry}
We now begin with the analysis of our estimator by providing guarantees for the estimation of
local index vectors in terms of $N$, the number of samples, and $J$, the number of level sets.
The estimation of local index vectors follows three steps:
\begin{enumerate}[label=\underline{\textsf{Step} \arabic*}, leftmargin=\widthof{[Step 3]}, topsep = 0pt, itemsep = 0ex]
\item\label{enum:partition} \textbf{Partition} $X$'s according to a dyadic partitioning of the range\footnote{Technically, we ought to use $\CR_1 = [-\sigma_{\varepsilon}, J^{-1}]$, and $\CR_J = [(J-1)/J, 1+\sigma_{\varepsilon}]$ to account for noise at the boundaries, but for the sake of simplicity we assume $Y$ is thresholded to $[0,1]$, such that $\SN{\CR_j} = J^{-1}$ for all $j$.} $\Im(f) = [0,1]$,
\begin{equation}
\label{eq:dyadic_partitioning}
\text{let } \CR_{j}:=\left[\frac{j-1}{J},\frac{j}{J}\right] \text{ and define } \CY_j := \CY \cap \CR_{j},\text{ and } \CX_j := \left\{X_i \in \CX: Y_i \in \CY_j\right\}.
\end{equation}
\item\label{enum:regularisation} \textbf{Estimate local index vectors} with $\hat a_j : = \hat b_j/\Vert \hat b_j \Vert$, where $\hat b_j$ is the solution of (local) linear regression  for samples $\CX_j,\,\CY_j$,
\begin{equation}\label{eq:hatbj_geom}\hat b_j := \hat \Sigma_j^{\dagger}\,\lsmean{(\CX_j,\CY_j)}{{(Y- \lsmean{\CY_j}{Y})(X - \lsmean{\CX_j}{X})}}.\end{equation}
\item\label{enum:extension} \textbf{Assign index vectors} to samples $\{X_i : i \in [N]\}$ by setting $\hat a(X_i) := \hat a_{j(X_i)}$.\\
\end{enumerate}

Denote now the tangent vectors by $a(X) := \gamma'(t)$ and $a_j := \gamma'(\bbE[t |\CR_j])$, where $t = \gamma^{-1}\circ \pi_{\gamma}(X)$.
Because of the quantization in \ref{enum:extension}, index vector estimation error can be decomposed as
\begin{equation}
\label{eq:decomposition_tangent_error}
\N{\hat{a}(X_i) - a(X_i)} \leq \N{\hat{a}_{j(X_i)} - a_{j(X_i)}} + \N{a_{j(X_i)} - a(X_i)} \leq
\N{\hat{a}_{j(X_i)} - a_{j(X_i)}} + \kappa_j\len{\CS_j},
\end{equation}
where $\CS_j$ is the infimum of all connected pieces
of $\Im(\gamma)$ such that $\bbP(V \in \CS_j|\CR_j) = 1$, and $\kappa_j$ its curvature bound.
Since $\SN{\CS_j} \lesssim L_f \SN{\CR_j}$ as long as $J^{-1} = \SN{\CR_j} \gg \sigma_{\varepsilon}$ (by Lemma \ref{lem:bound_curve_segment}), the second term
can be improved by increasing the number of level sets $J$.
On the other hand, for the first term we can prove the following concentration bound.

\begin{theorem}
\label{thm:midpoint_estimation}
Let $J \in \bbN$, $j \in [J]$, and $u > 1$. Define $\minSigmaYcool{j} := \Var{a_j^\top X, Y|\CR_j}(\SN{\CS_j}\SN{\CR_j})^{-1  }$.
Provided  Assumptions \ref{ass:A5} - \ref{ass:A2} hold, there exist constants  $C_N, C_A, C_E>0$, depending polynomially on
$L_f$, $\curvtor_j$, $B$, $\lengamma$, $\sigmadiff^* = (\sigmadiff \vee 3L_f \minSigmaYcool{j}\SN{\CR_j})$, $\sigma_{j,Y}^{-1}$, $\sigma_\perp^{-1}$,
such that whenever
\begin{equation}
\label{eq:thm3_conds} 4\sigma_{\varepsilon} < J^{-1} < \left(\frac{2}{3}\right)^{3/2}\frac{\minSigmaYcool{j} \sqrt{\sigma_{\perp}}}{L_f\curvtor_j \sigmadiff^*},
\quad
 \text{ and }\quad  \SN{\CX_j} \geq \max\{C_N (\log(D) + u)^2,D\},\end{equation}
we have
\begin{align}
\label{eq:thm_3_equation}
\bbP\left(\N{\hat a_j - a_j} \leq  C_A\frac{{\kappa_j}}{J^2}
+  {C_E}\frac{\log(D) + u}{\sqrt{\SN{\CX_j}}J}\right) \geq 1-\exp(u).
\end{align}
\end{theorem}

\noindent
The first condition in \eqref{eq:thm3_conds} deals with linearization, and effectively bounds the
influence of the cross-covariance term $\N{\Covv{\Prperp X,\Prpara X\lvert \CR}}$ from \ref{ass:A42}.
The condition gets easier to satisfy for shorter $\CS_j$, or shorter $\CR_j$.
This goes in line with the discussion in Section \ref{sec:model}, since by isolating shorter segments of $\Im(\gamma)$, NSIM approaches the SIM,
where the condition in \eqref{eq:thm3_conds}, and assumption \ref{ass:A41}, are trivally satisfied.
For a weaker form of \ref{ass:A42}, namely $\N{\Covv{\Prperp X,\Prpara X\lvert \CR}} \leq \curvtor_j \sigmadiff \CS_j^{1+\alpha}$, we obtain the same result with $J^{-\alpha}$ replacing $J^{-1}$,
and $J^{-(1+\alpha)}$ replacing $J^{-2}$, see Theorem \ref{thm:midpoint_estimation_restated}.

The second condition in \eqref{eq:thm3_conds} implies that, locally, there is a
minimal number of samples needed to ensure that the norm of the linear regression solution $\Vert\hat b_j\Vert$ is bounded from below.

Lastly, we note that $C_N, C_A, C_E$ are proportional to powers of $\sigma_{j,Y}^{-1}$, which implies that they are uniformly upper bounded (independent of $j$) if $\minSigmaYcool{j}$ is uniformly bounded from below.
We show in Lemma \ref{lem:covariance_AXY} in the Appendix that this is indeed the case whenever $\SN{\CR_j} \gg \sigma_{\varepsilon}$ and $\textrm{Var}(a_j^\top X, f(X) | \CR_j) (\SN{\CS_j}\SN{\CR_j})^{-1}$ is bounded from below.
The latter is satisfied if for example $f\in\CC^2$, see Lemma \ref{lem:covariance_AXY}.
Due to the bi-Lipschitz property of $g$, it seems reasonable however that $\textrm{Var}(a_j^\top X, f(X) | \CR_j) (\SN{\CS_j}\SN{\CR_j})^{-1}$ is bounded from below in more general scenarios.
The requirement $\SN{\CR_j} \gg \sigma_{\varepsilon}$, on the other hand, is also  observed numerically, precisely because $\textrm{Var}(a_j^\top X, Y | \CR_j)$
vanishes as soon as $\SN{\CR_j} -\sigma_{\varepsilon}$ is small. This suggests that our analysis correctly identifies
the dependency on $\minSigmaYcool{j}$.

\begin{remark}[Special cases of Theorem \ref{thm:midpoint_estimation}]\hfill\\
\vspace{-10pt}
\begin{enumerate}[topsep=-40pt,itemsep=5pt,parsep=0pt,partopsep=-20pt, leftmargin = \widthof{\,\,\,\underline{$K_{\CR} = 0$}}]
\item[\underline{$\sigma_{\varepsilon} = 0$}:] In the noise-free case the lower bound for $J^{-1}$ is removed. Thus, provided $\SN{\CX_j}$ is kept constant and $J\asymp N$, we achieve $\Vert \hat a_j - a_j\Vert \asymp N^{-1}$. This proves a $N^{-1}$ rate for the estimation of the (local) index vector
with the ordinary least squares estimator for strictly monotonic link functions.
\item[\underline{$\curvtor_{j} = 0$}:] If $\CR_j$ corresponds to a flat piece of the curve the first term in \eqref{eq:thm_3_equation} vanishes.
Thus, $\hat a_j$ is an unbiased estimator of $a_j$, with convergence rate $\SN{N}^{-1/2}$, provided $J$ is kept constant and $N\asymp \SN{\CX_j}$.
This result covers the SIM, and our estimation rate matches  other results \cite{balabdaoui2019score,brillinger2012generalized,hristache2001direct}.
\end{enumerate}
\end{remark}

Recalling decomposition \eqref{eq:decomposition_tangent_error}, Theorem \ref{thm:midpoint_estimation} can now be used to bound
$\N{\hat a(X_i) - a(X_i)}$ for all $i \in [N]$ by invoking a union bound argument
over all level sets $\CR_j$, $j \in [J]$.

\begin{corollary}
\label{cor:tangent_bound_all}
Let Assumptions \ref{ass:A5} - \ref{ass:A3} hold. Let $u > 1$ and assume we have $N$ iid. copies of $(X,Y)$.
Assume we partition the data set into $J$ partitions according to \eqref{eq:dyadic_partitioning}, so that
\begin{equation}\label{eqn:global_otherconds}
4\sigma_{\varepsilon} < \frac{1}{J} <  \left(\frac{2}{3}\right)^{3/2}\frac{\minSigmaYcool{J} \sqrt{\sigma_{\perp}}}{L_f\curvtor \sigmadiff^*},\quad
\textrm{where } \minSigmaYcool{J} := \max_{j \in [J]}\frac{\Var{a^\top X, Y|\CR_j}}{\SN{\mathcal{S}_j}\SN{\CR_j}},
\end{equation}
and $\sigmadiff^* := (\sigmadiff \vee 3L_f \minSigmaYcool{J} J^{-1})$, and compute local index vectors $\{\hat a_j : j \in [J]\}$.
There exist constants $C_N, C_A, C_E>0$, depending polynomially on $L_f$, $\curvtor$, $B$, $\lengamma$, $\sigmadiff$, $\sigma_{J,Y}^{-1}$, $\sigma_{\perp}^{-1}$,
such that if
\begin{equation}
\label{eqn:global_Nbound}
N \geq C_N\max\{(\log(D) + \log(J)u)^2,D\}u J,
\end{equation}
we have
\begin{align}
\label{eq:tangent_bound_all}
\bbP\left(\max_{i \in [N]}\N{\hat a(X_i) - a(X_i)} \leq  C_A \frac{\curvtor}{J}
+   C_E\frac{\log(D)u + \log(J)u^2}{\sqrt{N J}}\right) \geq 1-\exp(u).
\end{align}
\end{corollary}

\noindent
Let us make two remarks. First, terms in the bound on the right hand side of \eqref{eq:tangent_bound_all} can also be written in a local form,
\emph{i.e.}, a global curvature  bound can be replaced with a curvature bound for a segment around the sample $\pi_{\gamma}(X_i)$.
Thus, the learning of local index index vectors is consistent on locally linear pieces.

{Second, \eqref{eqn:global_otherconds} and \eqref{eqn:global_Nbound}  suggest that to optimally balance bias and variance we ought to use $J= C\min\{N/\log^2(N), \sigma_{\varepsilon}^{-1}\}$ level sets,
where $C>0$ is small enough so that \eqref{eqn:global_Nbound} is satisfied.
Looking at \eqref{eq:tangent_bound_all}, this implies that there are two regimes.

In the first regime, in order to decrease the error we ought to increase $J$ as long as $J \gg \sigma_{\varepsilon}$, \emph{i.e.}, subdivide the data set into an increasing number of subsets, while keeping the number of samples within each subset roughly constant.
The rationale behind this is that further subdividing the data set not only reduces the approximation error (which is caused by the curvature), but it also reduces the variance in the linear regression part of the problem, \emph{i.e.}, when estimating $a_{j(X_i)}$ by $\hat{a}_{j(X_i)}$.
In the second regime function noise precludes further decreasing $\SN{\CR_j}$, since we cannot further decrease $\SN{\CS_j}$.
In other words, the noise level $\sigma_\varepsilon$ imposes a lower bound on $\SN{\CR_j}$, and the bias does not completely vanish.

Note also that in \eqref{eq:tangent_bound_all}, compared to \eqref{eq:thm_3_equation}, we lose an order in $J^{-1}$, \emph{i.e.}, in the interval length.
This is due to the use of quantization to approximate the entire tangent field over the respective level set.
This could be improved by learning a separate tangent for each sample $X_i$ from a level set centred around $X_i$, but the second term in \eqref{eq:tangent_bound_all} prohibits achieving $J^{-2}$ overall.


\section{Function estimation}
\label{sec:function_estimation}

\begin{algorithm}[b!]
\caption{Modified out-of-sample prediction for NSIM estimator}
\label{alg:main_modified}
\vspace{0.2cm}
\textbf{Out-of-sample prediction}
\begin{algorithmic}
  \REQUIRE sample $x$, data set $\{(X_i,Y_i) : i \in [N]\}$ with $\{\hat a(X_i) : i \in [N]\}$, second data set $\{(X_{\ell}',Y_{\ell}') : \ell \in [N]\}$, parameters $k$, $\tempradius$
  \vspace{0.2cm}
  \STATE For all $\ell \in [N]$: $\hat a(X_{\ell}') := \hat a(X_{i^*})$ where $i^* := \argmin_{i \in [N]} \intMetMod{X_{\ell}'}{X_i}{\tempradius})$
  \STATE Compute nearest neighbor ordering $1(x),\ldots,k(x)$ over $\{(X_{\ell}',Y_{\ell}') : \ell \in [N]\}$ based on $\intMetMod{x}{X_{\ell}'}{\tempradius}$.
\ENSURE  $\hat{f}_{k}(x) = k^{-1}\sum_{\ell = 1}^{k}Y_{\ell(x)}'$
\end{algorithmic}
\end{algorithm}

\noindent
In this section we use the guarantees on local index vector estimation to establish
function estimation guarantees.
We recall that the estimator \eqref{eq:prediction} predicts an output by averaging the responses of  $\{(X_{i(x)}, Y_{i(x)}) : i \in [k])\}$, the $k$ closest samples with respect to the metric $\intMetMod{x}{\cdot}{\tempradius}$. This makes the analysis challenging because the same data is used twice: first for estimating the geometry and then for predicting the function.
As a result, random variables $\left\{\varepsilon_{i(x)} : i \in [k]\right\}$ become statistically dependent
and their finite sample average could be biased.

To avoid this technical issue split the given data set (consisting of $2N$ samples) in two halves (reducing the effective sample size only by a factor of $1/2$) and use the first half, $\{(X_{i},Y_{i}) : i \in [N]\}$, for approximating the geometry, and the second half,  $\{(X'_{\ell},Y_{\ell}') : \ell \in [N]\}$, for function prediction.
We then extend the tangent field approximation through nearest neighbors, defining $\hat a(X_\ell') :=\hat a(X_{i^*})$, where $i^* := \argmin_{i\in[N]} \intMetMod{X_{\ell}'}{X_i}{\tempradius}$.
The prediction of $f(x)$ is then given by averaging the responses, $Y'_{\ell(x)},\, \ell \in [k]$, of $k$ closest samples  with respect to $\intMetMod{x}{\cdot}{\tempradius}$ from
$\{X_\ell' : \ell \in [N]\}$, see Algorithm \ref{alg:main_modified}.
Thus, random variables $\varepsilon_{\ell}'$ are not used in the selection
of $k$ closest neighbors of $x$ and we preserve unbiased finite sample averages,
\emph{i.e.}, $\bbE \varepsilon_{\ell(x)}'  = \bbE \varepsilon = 0$.

We split our analysis in two parts. The first concerns the case when $\gamma$ is close to an
affine space (see Definition \ref{def:almost_linearity}),  and we call it a perturbed single index model.
The second part extends the analysis to general curves $\gamma$. The reason for treating the first case separately
is that we can achieve theoretical guarantees even without restricting the search space of nearest neighbors, \emph{i.e.} setting $\eta = \infty$.
Furthermore, numerical experiments in Section \ref{sec:experiments} suggest that perturbed SIMs
fit well to several data sets that were previously used as benchmarks for the SIM.

\subsection{Function estimation for perturbed single index models}
\label{subsec:function_estimation}
We begin by defining the notion of almost linearity that is used to quantify the deviation
of the true model to an ordinary SIM, respectively, of the curve $\gamma$ to a straight line.

\begin{definition}
\label{def:almost_linearity}
Let $\FI$ be an interval and $\gamma : \FI \rightarrow \bbR^{D}$ an arc-length parametrized $\CC^1(\FI)$ curve. Let $0 < \theta \leq 1$.
We say $\gamma$ is \emph{$\theta$-almost linear} if $\EUSP{\gamma'(t)}{\gamma'(s)} > \theta$ for all $t, s \in \FI$.
\end{definition}
Definition \ref{def:almost_linearity} implies that if $\theta$ is close to $1$ then $\gamma$ is close to a
straight line. Furthermore, the Euclidean distance approximates the geodesic distance well, \emph{i.e.} $\N{v - v'} \asymp d_{\gamma}(v, v')$ for any $v, v' \in \Im(\gamma)$,
which allows to prove an equivalence between the (unrestricted) proxy metric
$\intMetMod{x}{\cdot}{\infty}$ and $d_{\gamma}(x,\cdot)$.

\begin{proposition}
\label{prop:intrinsic_bound_for_approximated_closest_neighbors}
Assume $\gamma$ is $\theta$-almost linear for some $\theta > \kappa \boundnormal$.
Let $\{\bar x_i : i \in [N]\} \subset \suppmarg$, and $\{\hat a(\bar x_i) : i \in [N]\}\subset\bbS^{D-1}$ be arbitrary sets.
Let $x\in\suppmarg$. If $\bar x_{k(x)}$ is $k$-th closest sample, based
on $\intMetMod{x}{\cdot}{\infty}$, and $\bar x_{k^*(x)}$ the $k$-closest sample, based on $d_{\gamma}(x,\cdot)$,
we have
\begin{equation}
\label{eqn:bound_on_approximated_and_true_ball}
d_{\gamma}(x,\bar x_{k(x)}) \leq \frac{2\vee (\SN{\CI} + 2B)}{\theta - \kappa \boundnormal}\LRP{d_{\gamma}(x, \bar x_{k^*(x)})
+  \max\limits_{i \in [N]} \N{\hat a(\bar x_i) - a(\bar x_i)}}.
\end{equation}
\end{proposition}

\noindent
Note that curvature and reach of a curve $\gamma$ always satisfy $\kappa \reach\leq 1$.
This means that $\kappa B<1$ is trivially satisfied, since $B<\reach$ by \ref{ass:A2}.
Thus, the requirement  $\theta>\kappa B$ is driven by linearization, namely,
by the fact that we are approximating the geodesic geometry of samples projected onto a curved
space, with a linear geometry of samples projected onto its linerization.

To show guarantees for function estimation we first need to derive bounds on the tangent field
$\max_{\ell \in [N]} \N{\hat a(X_{\ell}') - a(X_{\ell}')}$ from bounds on $\max_{i \in [N]} \N{\hat a(X_{i}) - a(X_{i})}$, given by Corollary \ref{cor:tangent_bound_all}.
Using Proposition \ref{prop:intrinsic_bound_for_approximated_closest_neighbors} with sets $\{X_{i} : i \in [N]\}$ and $\{\hat a(X_{i}) : i \in [N]\}$, for all $\ell \in [N]$ we have
\begin{equation}
\begin{aligned}
\label{eq:tangent_field_extension}
\N{\hat a(X_\ell') - a(X_\ell')} &= \N{\hat a(X_{1(X_{\ell}')}) - a(X_\ell')} \leq
\N{\hat a(X_{1(X_{\ell}')}) -  a(X_{1(X_{\ell}')}) +  a(X_{1(X_{\ell}')}) -  a(X_\ell')}\\
&\leq \max_{i \in [N]}\N{\hat a(X_i)-a(X_i)} + \kappa d_{\gamma}( X_{\ell}',X_{1(X_{\ell}')})\\
&\leq \frac{1+\curvtor(2\vee(\SN{\CI}+2B))}{\theta - \curvtor B}\left({\max_{i \in [N]}\N{\hat a(X_i)-a(X_i)} + d_{\gamma}(X_\ell', X_{1^*(X_{\ell}')})}\right)
\end{aligned}
\end{equation}
where $X_{1^*(X_{\ell}')}$ is the sample closest to $X_{\ell}'$ with respect to the geodesic distance.
We can now state the main result for function estimation.

\begin{theorem}
\label{thm:guarantees_f_general}
Assume \ref{ass:A5} - \ref{ass:A3}.
Let $\eta = \infty$ and assume that $\gamma$ is $\theta$-almost linear for some $\theta > \kappa \boundnormal$.
Whenever $N, J$ satisfy the conditions of Corollary \ref{cor:tangent_bound_all},
we have for arbitrary $x \in \suppmarg$ and $1 < u < N$
\begin{equation}
\label{eq:f_guarantee_first_regime}
\SN{\hat{f}_k(x) - f(x)}\leq C\frac{\sigma_{\varepsilon}u}{\sqrt{k}} +\frac{C_B}{(\theta-\kappa \boundnormal)^2}\left(u\frac{k}{N} +
C_E\frac{\log(D)u + \log(J)u^2}{\sqrt{NJ}}+ C_A\frac{\curvtor}{J}\right).
\end{equation}
 with probability at least $1-\exp(-u)$, where $C_A,C_E>0$ are constants from Corollary \ref{cor:tangent_bound_all}, $C>0$ is an absolute constant and
 $C_B=2L_f(2\vee (\SN{\CI}+2B))\left(1+\curvtor(2 \vee (\SN{\CI}+2B)\right)$.
\end{theorem}

\begin{proof}
We first decompose the left-hand side of \eqref{eq:f_guarantee_first_regime} as
\begin{align*}
&\SN{\hat{f}_k(x) - f(x)} = \SN{\frac{1}{k}\sum_{\ell=1}^{k}Y'_{\ell(x)} - f(x)} \leq \SN{\frac{1}{k}\sum_{\ell=1}^{k}\varepsilon'_{\ell(x)}} +  \SN{\frac{1}{k}\sum_{\ell=1}^{k}f(X'_{\ell(x)}) - f(x)}.
\end{align*}
The first term is a sum independent copies of $\varepsilon$. Since $\SN{\varepsilon}\leq \sigma_{\varepsilon}$ almost surely,
and $\bbE\varepsilon = 0$, H\"offding's inequality for bounded random variables gives, for an absolute constant $C>0$
\begin{align*}
\bbP\left(\SN{\frac{1}{k}\sum_{\ell=1}^{k}\varepsilon'_{\ell(x)}} \leq C \frac{\sigma_{\varepsilon}u}{\sqrt{k}}\right) \geq 1-\exp(-u^2)\geq 1-\exp(-u).
\end{align*}
Assume now $\{\pi_{\gamma}(X_\ell') : \ell \in [N]\}$ and $\{\pi_{\gamma}(X_i) : i \in [N]\}$ are $\delta$-nets for $\Im(\gamma)$ with respect to
$d_{\gamma}$. We can use the Lipschitz property of $g$ and apply Proposition \ref{prop:intrinsic_bound_for_approximated_closest_neighbors} to bound the second term as
\begin{align*}
\SN{\frac{1}{k}\sum_{\ell=1}^{k}f(X_{\ell(x)}') - f(x)} &\leq  \frac{L_f}{k}\sum_{\ell=1}^{k}d_{\gamma}(X_{\ell(x)}', x) \leq \frac{L_f (2\vee (\SN{\CI} + 2B))}{\theta - \kappa \boundnormal}\left(\delta k + \max_{\ell \in [N]}\N{\hat{a}(X'_\ell) - a(X'_\ell)}\right).
\end{align*}
Using \eqref{eq:tangent_field_extension} and $d_{\gamma}(X_\ell', X_{1^*(X_{\ell}')}) \leq \delta \leq \delta k$
we get
\begin{align*}
\SN{\frac{1}{k}\sum_{\ell=1}^{k}f(X_{\ell(x)}') - f(x)} &\leq \frac{C_B}{(\theta - \kappa \boundnormal)^2}\left(\delta k + \max_{i \in [N]}\N{\hat{a}(X_i) - a(X_i)}\right).
\end{align*}
Lemma \ref{lem:net_property} gives that $\{\pi_{\gamma}(X_\ell') : \ell \in [N]\}$ and $\{\pi_{\gamma}(X_i) : i \in [N]\}$ are $\delta$-nets for
$\delta = \lengamma u \left(\uniformC N\right)^{-1}$ with probability $1-2\exp(-u)$.
The claim then follows by Corollary
\ref{cor:tangent_bound_all}.
\end{proof}

Theorem \ref{thm:guarantees_f_general} reveals that the error in function estimation originates from three sources.
The first term accounts for the averaging of the noise, which is incurred by responses $Y_{\ell}'$.
Using $k = \CO(N^{2/3})$, as is standard for Lipschitz-smooth functions, it decays at a rate $N^{-1/3}$.
The second term bounds the geodesic distance to the nearest neighbor, and comes from the covering of the curve by the projected samples.
The last two terms are from the approximation of the geodesic metric with the proxy metric $\intMetMod{x}{\cdot}{\infty}$ through
tangent approximations $\{\hat a(X_i) : i \in [N]\}$, and behave according to Corollary \ref{cor:tangent_bound_all}.
Setting $k = \CO(N^{2/3})$ and $J= C\min\{N/\log^2(N), \sigma_{\varepsilon}^{-1}\}$, as in Section \ref{sec:geometry},
yields
\begin{equation}
\label{eq:f_guarantee_with_parameter_choices}
\SN{\hat{f}_k(x) - f(x)}\lesssim \frac{(1+\sigma_{\varepsilon})u}{(\theta-\curvtor \boundnormal)^2} N^{-1/3}
+ \frac{\curvtor}{(\theta-\curvtor \boundnormal)^2} \max\left\{\frac{\log^2(N)}{N},\sigma_{\varepsilon}\right\}.
\end{equation}
We see that the estimator is generally biased since the error
tends to $\curvtor/(\theta-\curvtor \boundnormal)^2 \sigma_{\varepsilon}$ for $N\rightarrow \infty$.

\begin{remark}[Special cases of Theorem \ref{thm:guarantees_f_general}]\hfill\\
\vspace{-10pt}
\begin{enumerate}[topsep=-10pt,itemsep=5pt,parsep=0pt,partopsep=-10pt, leftmargin = \widthof{\,\,\,\underline{$K_{\CR} = 0$}}]
\item[\underline{$\sigma_{\varepsilon} = 0$}:] In the noise-free case the first term in \eqref{eq:f_guarantee_first_regime} vanishes, and thus
choosing $k = 1$, and $J = CN/\log(N)^2$ with $C$ small enough so that \eqref{eqn:global_Nbound} holds, we have
\begin{equation}
\label{eq:f_noise_free_case}
\SN{\hat{f}_k(x) - f(x)}\lesssim \frac{1}{(\theta-\curvtor \boundnormal)^2}\frac{u + \log(D)\log(N)u + \log^2(N) (\curvtor + u^2)}{N}.
\end{equation}
Up to logarithmic factors, this matches the optimal rate for noise-free estimation of Lipschitz functions, see \cite{kohler2014optimal}.

\item[\underline{$\curvtor = 0$}:] If the model follows the ordinary SIM  the second term in \eqref{eq:f_guarantee_with_parameter_choices} vanishes.
Thus we achieve a $N^{-1/3}$ rate, which is optimal for Lipschitz smooth functions \cite{gyorfi2006distribution}.
\end{enumerate}
\end{remark}

\noindent
Before moving to general curves, let us remark why achieving consistent estimation is a challenging task in the noisy, nonlinear case.
The presented estimator is based on localization and linearization, where localization hinges on the fact that conditional distributions $(X,Y)|\CR_j$ are increasingly SIM-like when reducing the level set width $\SN{\CR_j} = J^{-1}$.
This reduces the effects of curvature and linearization becomes increasingly accurate.
On the other hand, relating the width of $\CR_j$ with the length of corresponding segment $\CS_j$ of the curve, as $\SN{\CR_j}\asymp \SN{\CS_j}$, is by Lemma \ref{lem:bound_curve_segment}
valid only if  $\SN{\CR_j} > 2\sigma_{\varepsilon}$.
Namely, reducing $\CR_j$ beyond that threshold does not reduce $\SN{\CS_j}$, \emph{i.e.}, $X\lvert \CR_j$ does not become more SIM-like.
This predicament can not be further improved under our noise model.

Results in this section imply that having a consistent estimator of the tangent field of $\Im(\gamma)$, whose sample complexity does not depend exponentially on $D$,
is sufficient to construct a consistent estimator for $f$, with a similar sample complexity.
At the same time, a consistent, low-complexity estimator of $f$ can be used to estimate the tangent
field, by approximating $\nabla f$ through finite sample differences.
This suggests a certain equivalence between estimating $f$ and estimating the
tangent field of $\Im(\gamma)$, and to some extent the manifold $\Im(\gamma)$ itself.

Minimax rates for estimating a manifold from $N$ samples $\{X_i : i \in [N]\}$ that are spread around it have been extensively studied
in \cite{genovese2012minimax,genovese2012manifold}.
Moreover, in \cite{genovese2012manifold} the authors provide
a theoretical estimator that converges at a $(\log(N)/N)^{2/(2+d)}$ rate (measured in the Hausdorff distance), where $d$ is
the dimensionality of the manifold. However, they emphasize that the estimator is not practical
and pose the development of a practical alternative as an important open problem. To the best of our knowledge,
this problem still has not been solved.

\subsection{Extension to general curves}
\label{subsec:function_estimation_general_curves}
In the general case the unrestricted proxy metric $\intMetMod{x}{\cdot}{\infty}$ is not
equivalent to the geodesic metric $d_{\gamma}(x,\cdot)$, and thus cannot be used to reliably select nearest neighbors.
To better illustrate this point, let $\gamma$ be a segment of the unit circle that contains two antipodal points $\pi_\gamma(x)$ and $\pi_\gamma(x')$,
and assume we have access to the true tangents $a(x)$, $a(x')$, so that $a(x)=-a(x')$.
Thus, on one hand we have $d_\gamma(x,x')=\pi$, and on the other $\intMetMod{x}{x'}{\infty} = \intMetMod{x'}{x}{\infty}=0$ since $a(x)\perp x'-x $.

To avoid this and establish an equivalence between $d_{\gamma}(x,\cdot)$ and $\intMetMod{x}{\cdot}{\tempradius}$,
similar to Proposition \ref{prop:intrinsic_bound_for_approximated_closest_neighbors},
we thus have to restrict the search space.
Considering the unit circle example, we ought to choose $\tempradius>0$ that ensures there are no two points
$x, x'$, such that $d_\gamma(\pi_\gamma(x),\pi_\gamma(x'))\gg 0$, but $\N{x-x'}\leq \tempradius$ and $\SN{a(x)^\top(x-x')} =0 = \SN{a(x')^\top(x'-x)}$.
It can be shown that this is satisfied for $\tempradius < 2(\reach - \boundnormal)$, provided assumption \ref{ass:A2} holds,
see Figure \ref{fig:q_ahalf} and Lemma \ref{lem:general_curve_ball_auxiliary}.
On the other hand, $\tempradius$ needs to be large enough to ensure there are enough samples within $\CB_{\N{\cdot}}(x,\tempradius)$, with respect to $N$, to achieve optimal
function prediction rates.
Under the uniformity assumption  \ref{ass:A3}, this is ensured whenever
 $\tempradius > 2\boundnormal$.

\begin{figure}[]
  \subfloat[$B=1/4\reach$]{\includegraphics[width=0.32\linewidth]{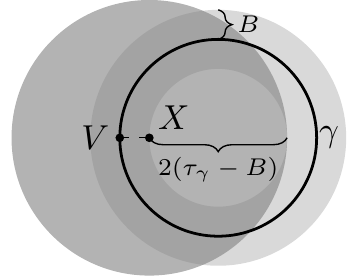}\label{fig:q_one_half_1}}
  \subfloat[$B=1/2\reach$]{\includegraphics[width=0.32\linewidth]{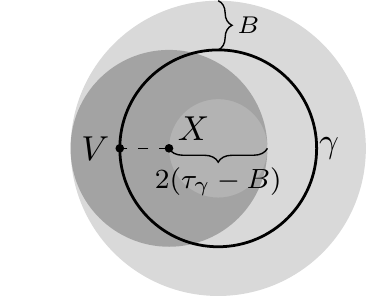}\label{fig:q_one_half_2}}
  \subfloat[$B=3/4\reach$]{\includegraphics[width=0.32\linewidth]{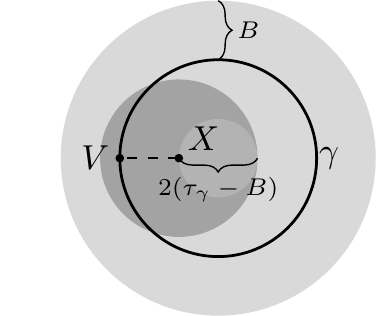}\label{fig:q_one_half_3}}
\caption{Consider $\gamma = \bbS^{1}$ and a point $X=V+W$ with $\N{W} = \boundnormal = \{1/4,1/2,3/4\}\reach$. The ball $\CB_{\Vert \cdot \Vert}(X, 2(\reach -\boundnormal))$ never intersects the antipodal region. Furthermore if $B \leq 1/2\reach$, $\pi_{\gamma}^{-1}(V)$ entirely is covered by $\CB_{\Vert \cdot \Vert}(X, 2(\reach -\boundnormal))$, implying that it has lower bounded probability mass.}
\label{fig:q_ahalf}
\end{figure}

Balancing these two demands we get $2\boundnormal<\tempradius<2(\reach - \boundnormal)$
and thus  $\boundnormal < 1/2\reach$.
Therefore, we require a more restrictive version of \ref{ass:A2}.
To compensate for errors in tangent approximations we further impose $\tempradius < \reach$.
This allows to prove a guarantee for metric equivalence.
\begin{proposition}
\label{prop:local_metric_equivalence_general}
Assume \ref{ass:A2} for $\boundnormal = (1/2-q)\reach$ for some $q > 0$, and choose any $\tempradius \in (2\boundnormal, \reach)$.
Let $\{\bar x_i : i \in [N]\} \subset \suppmarg$, and $\{\hat a(\bar x_i) : i \in [N]\}\subset\bbS^{D-1}$ be arbitrary sets.
For an arbitrary $x\in\suppmarg$ let $\bar x_{k(x)}$ be its $k$-th closest sample based
on $\intMetMod{x}{\cdot}{\tempradius}$, and $\bar x_{k^*(x)}$ be its $k$-closest sample based on $d_{\gamma}(x,\cdot)$.
Whenever $\{\pi_{\gamma}(X_i) : i \in [N]\}$ forms a $\delta$-net on $\Im(\gamma)$, and
\begin{equation}
\label{eq:general_curve_requirement_on_net}
\delta k< \max\Big\{\tempradius - 2B, \frac{1}{2}\big(q\reach - (\reach + \lengamma + 2\boundnormal)\max_{i \in [N]}\N{\hat a(\bar x_i) - a(\bar x_i)}\big)\Big\}
\end{equation} we have
\[d_{\gamma}(x, \bar x_{k(X)}) \leq  4(2\vee (\SN{\CI}+2B) \vee \reach) \LRP{d_{\gamma}(x, \bar x_{k^*(x)}) + \max_{i \in [N]}\N{\hat a(\bar x_i) - a(\bar x_i)}}.\]
\end{proposition}

\noindent
Covering the manifold $\Im(\gamma)$ with a sufficiently fine $\delta$-net $\{\pi_{\gamma}(X_i) : i \in [N]\}$, and condition \eqref{eq:general_curve_requirement_on_net}, are satisfied
with high probability as soon as $N$ is sufficiently large, due to Corollary \ref{cor:tangent_bound_all} and Lemma \ref{lem:net_property}, respectively.
In that case, Theorem \ref{thm:guarantees_f_general} holds also for general curves, by simply replacing
Proposition \ref{prop:intrinsic_bound_for_approximated_closest_neighbors} with Proposition \ref{prop:local_metric_equivalence_general} in the proof.
Since for $N\rightarrow \infty$ the term $\max_{i \in [N]}\N{\hat a(\bar X_i) - a(\bar X_i)}$ converges to $\CO(\curvtor \sigma_{\varepsilon})$
by Corollary \ref{cor:tangent_bound_all}, we are ensured to enter the valid regime
whenever the noise $\sigma_{\varepsilon}$ is small enough compared to $q$ (in particular in the noise-free case).

\begin{proposition}
Assume \ref{ass:A5} - \ref{ass:A3}, and
the conditions of Proposition \ref{prop:local_metric_equivalence_general} hold.
Let $\tempradius \in (2\boundnormal, \reach)$.
Whenever $N$, $J$ satisfy the conditions of Corollary \ref{cor:tangent_bound_all},
we have for arbitrary $x \in \suppmarg$ and $1 < u < N$
\begin{equation}
\label{eq:f_guarantee_global_regime}
\SN{\hat{f}_k(x) - f(x)}\leq C\frac{\sigma_{\varepsilon}u}{\sqrt{k}} +\frac{C_B}{(\theta-\kappa \boundnormal)^2}\left(u\frac{k}{N} +
C_E\frac{\log(D)u + \log(J)u^2}{\sqrt{NJ}}+ C_A\frac{\curvtor}{J}\right),
\end{equation}
 with probability at least $1-\exp(-u)$, where $C_A,C_E>0$ are constants from Corollary \ref{cor:tangent_bound_all},
 $C>0$ is an absolute constant and $C_B=32L_f(2\vee (\SN{\CI}+2B) \vee \reach)\left(1+\curvtor(2\vee (\SN{\CI}+2B) \vee \reach)\right)$.
\end{proposition}



\section{Numerical Experiments}
\label{sec:experiments}
In this section, we present experimental results of the proposed estimator in two settings.
First, we conduct synthetic experiments to validate theoretical results of Sections \ref{sec:geometry} and \ref{sec:function_estimation}.
Second, we benchmark the estimator against commonly used methods on a selection of real-world data sets.
The source code for Algorithm \ref{alg:main} and synthetic experiments is available at \url{https://github.com/soply/nsim_algorithm}.
Moreover, real-world data sets, code for their preprocessing, and implementations of competing estimators (or references, if publicly available source code is used)
are readily available at  \url{https://github.com/soply/local_sim_experiments}.
\subsection{Experiments with synthetic data}
\label{subsec:synthethic_experiments}
\begin{figure}[]\noindent
\subfloat[\texttt{Line}]{\includegraphics[trim={2cm 1.5cm 1cm 2.cm},clip,width=0.32\linewidth]{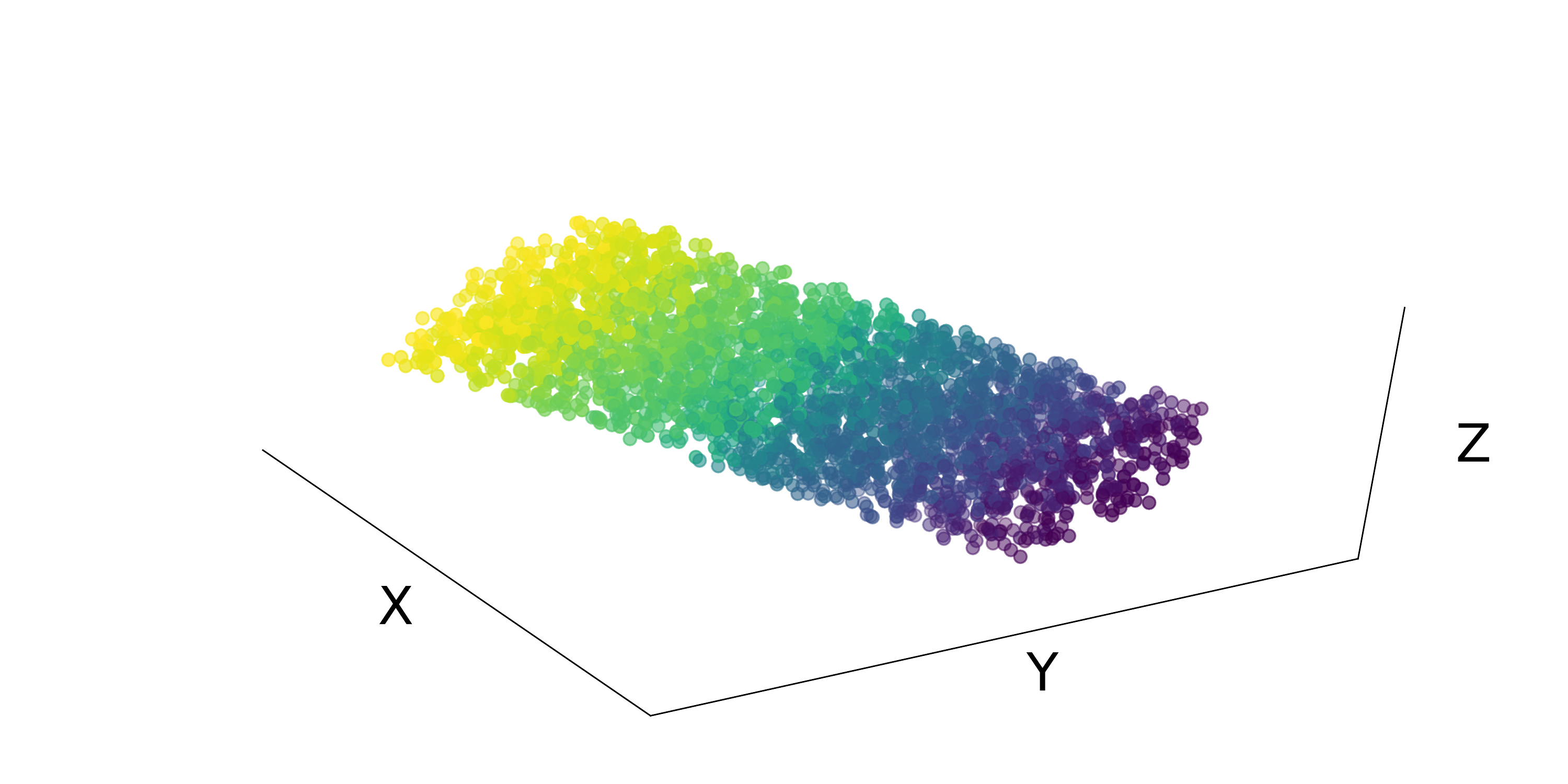}\label{fig:marginal_X_id}}
\subfloat[\texttt{S-Curve}]{\includegraphics[trim={2cm 1.5cm 1cm 2.cm},clip,width=0.32\linewidth]{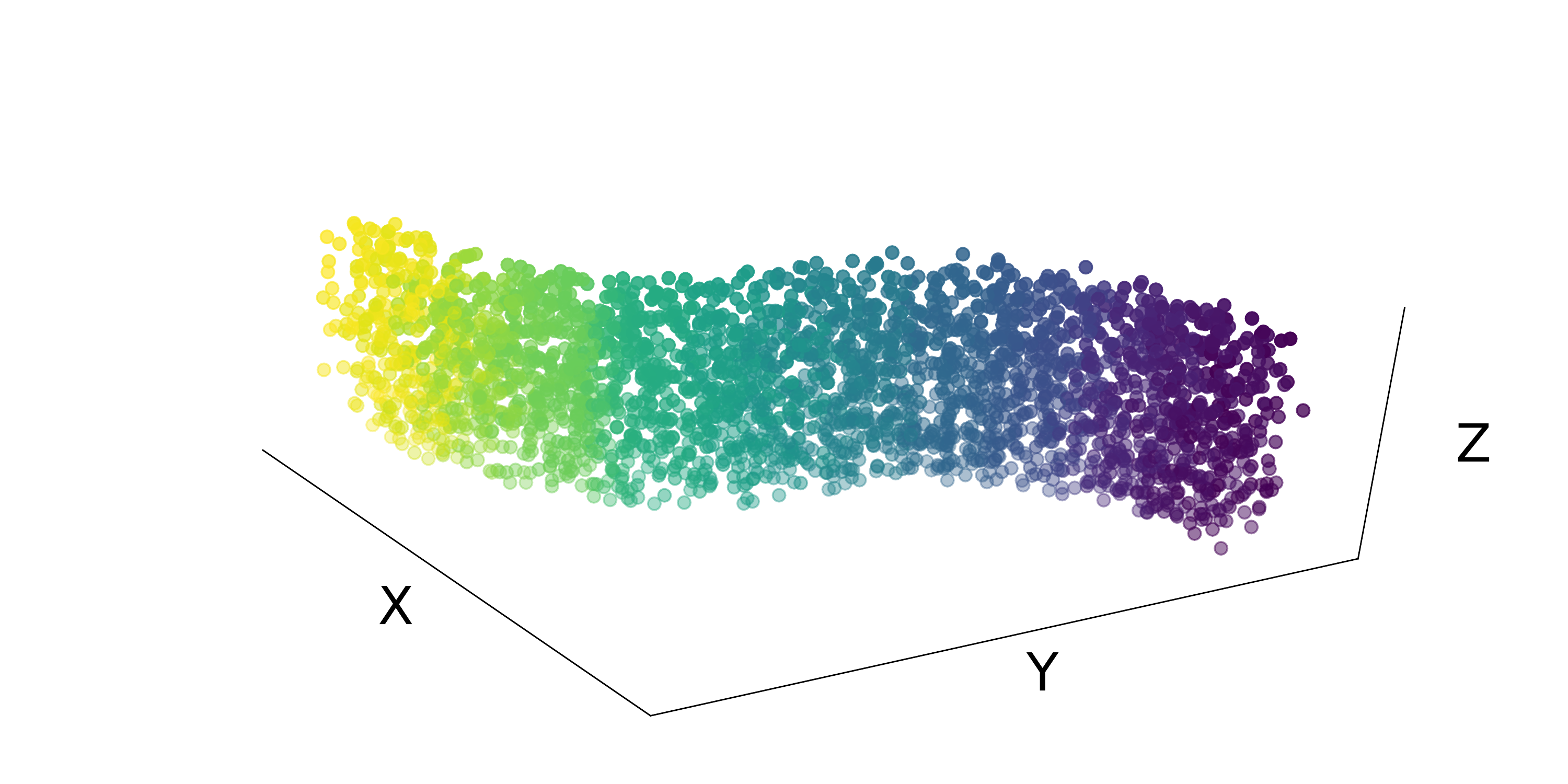}\label{fig:marginal_X_scurve}}
\subfloat[\texttt{Helix}]{\includegraphics[trim={2cm 1.5cm 1cm 2.cm},clip,width=0.32\linewidth]{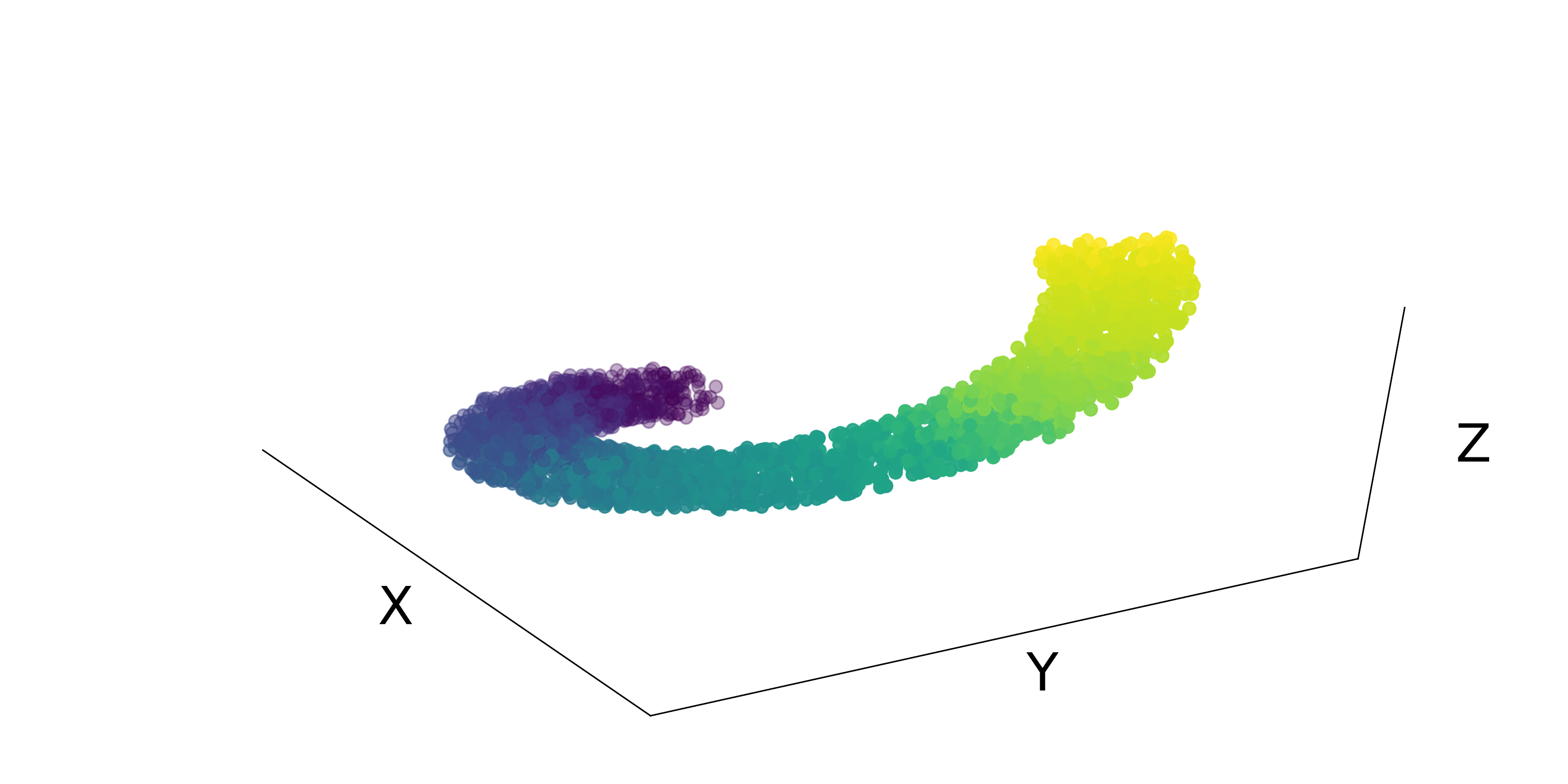}\label{fig:marginal_X_helix}}\\
\subfloat[\texttt{Line}]{\includegraphics[width=0.32\linewidth]{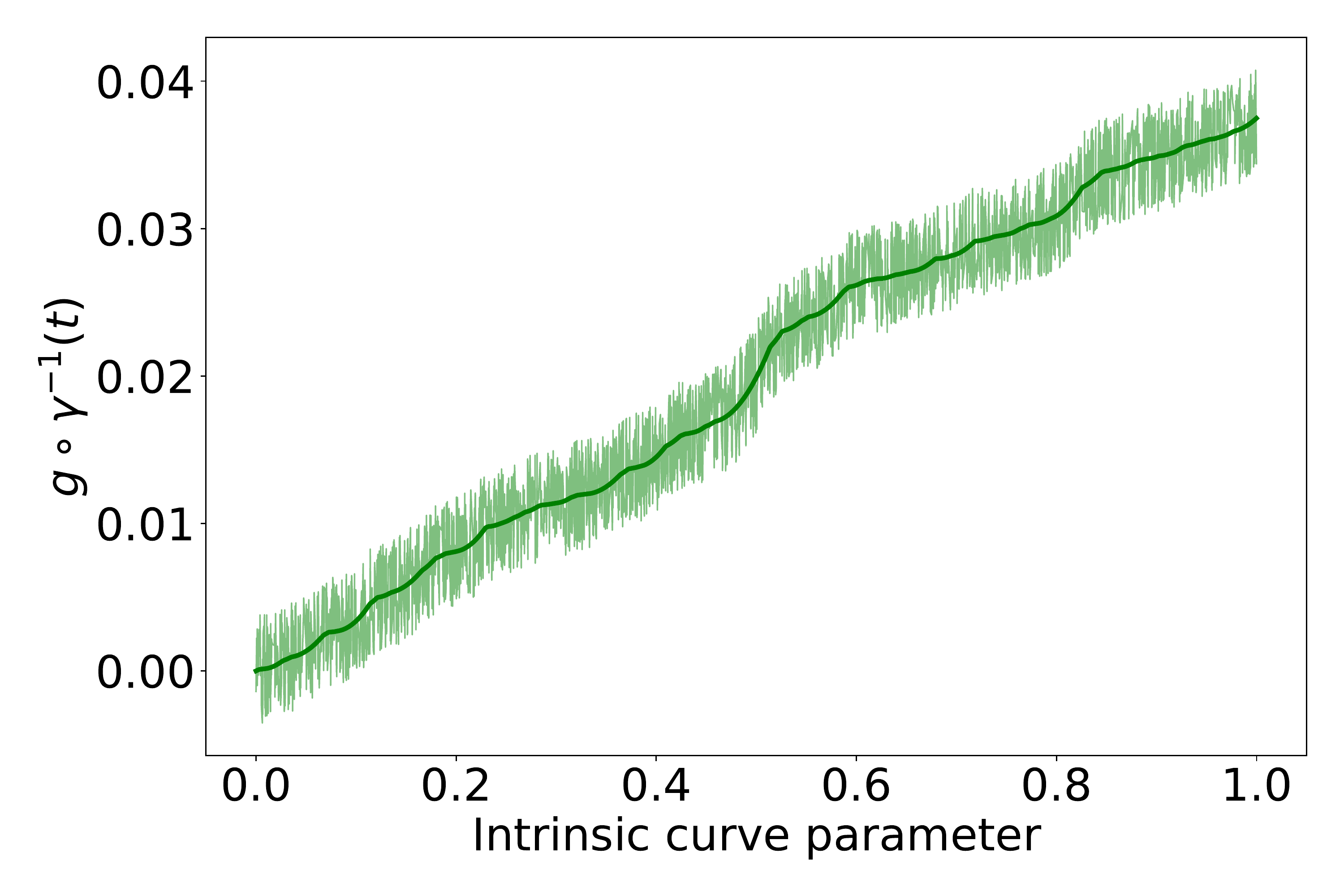}\label{fig:function_id}}
\subfloat[\texttt{S-Curve}]{\includegraphics[width=0.32\linewidth]{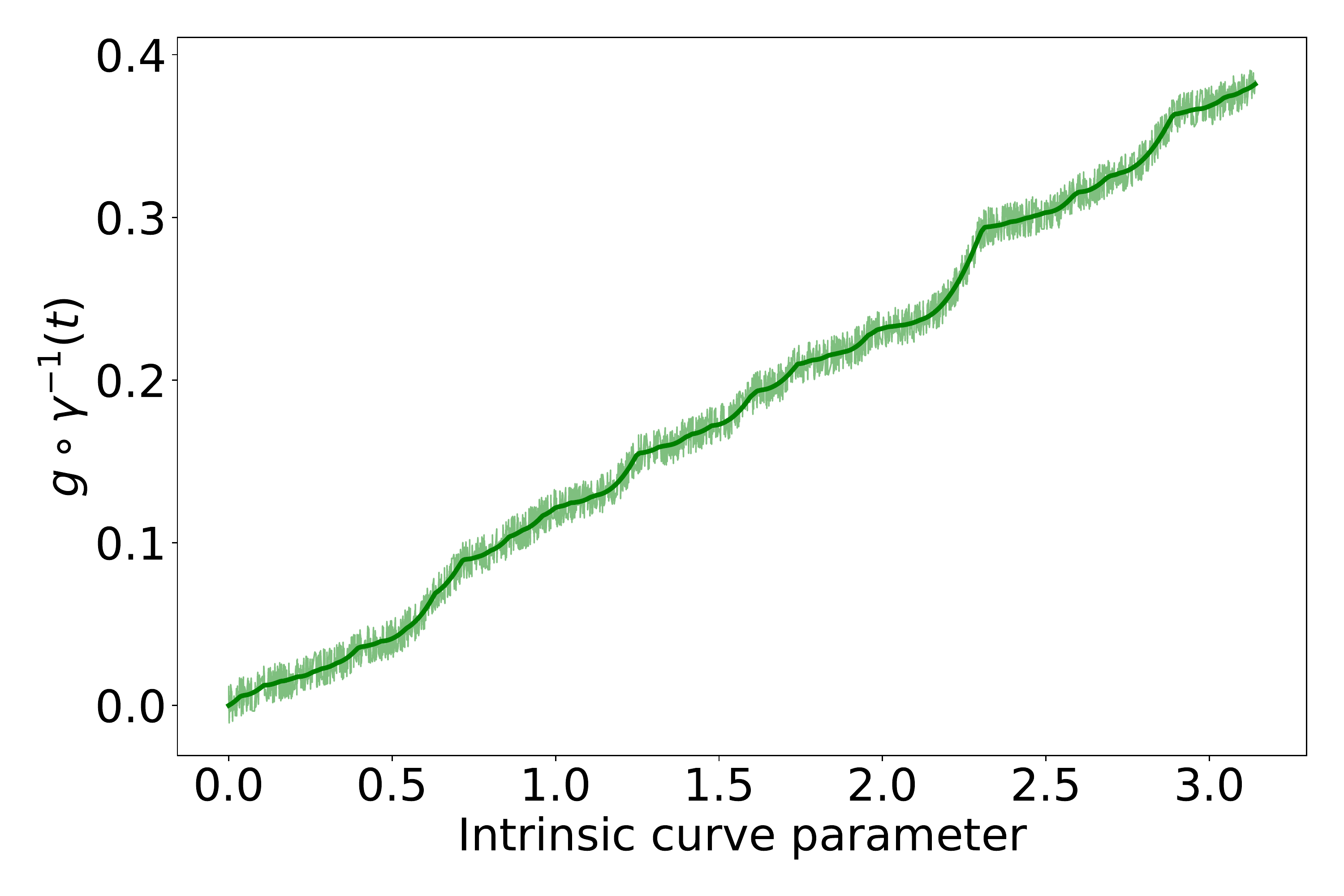}\label{fig:function_scurve}}
\subfloat[\texttt{Helix}]{\includegraphics[width=0.32\linewidth]{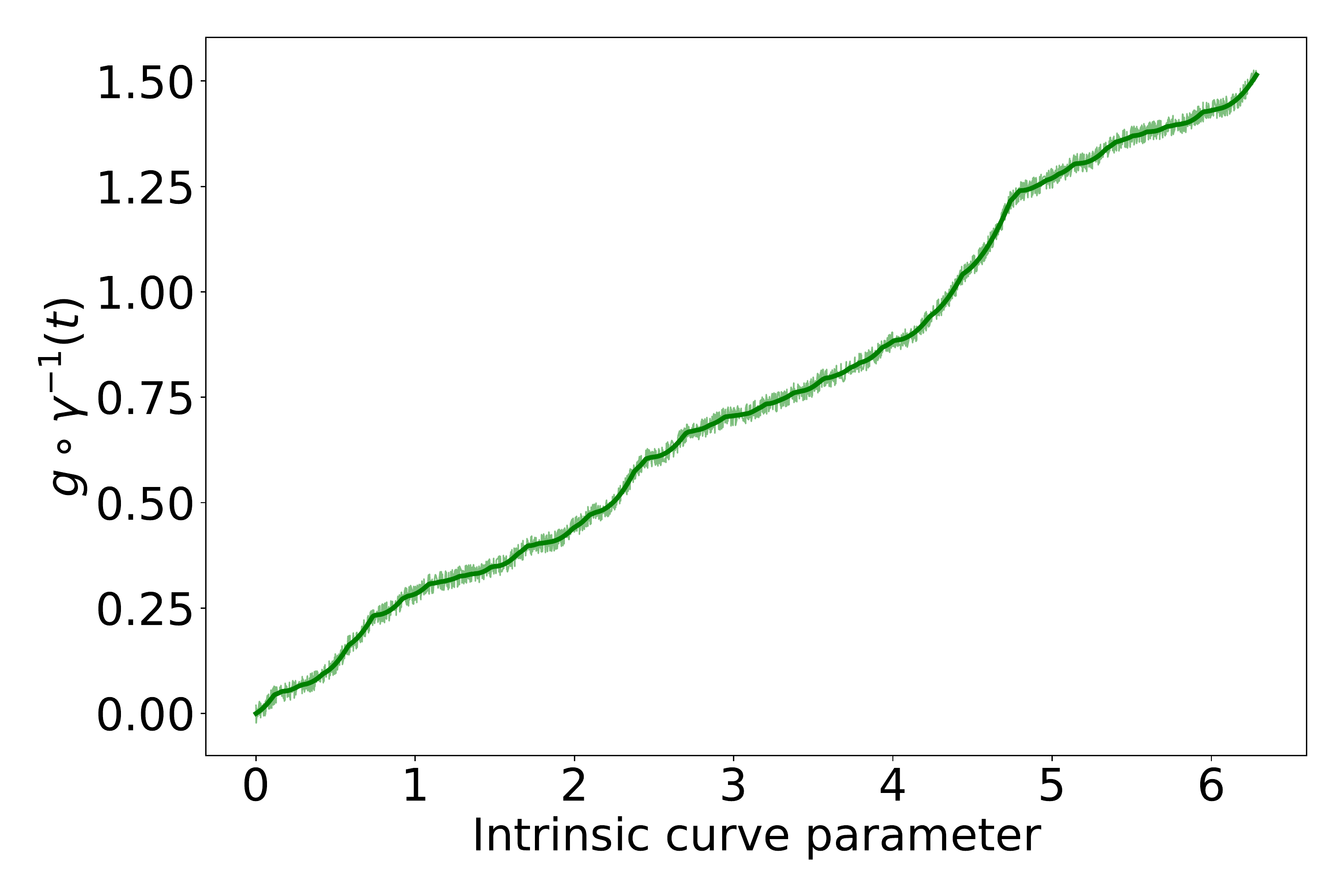}\label{fig:function_helix}}
\caption{Top row: random draws from marginal distributions around $\Im(\gamma)$ visualized for $D = 3$ and for
three different curves. Bottom row: link functions in dark green with maximum noise
level in light green. Note that the visually small noise for the S-curve and Helix manifold is due to a different scaling,
and becomes large when isolating small level sets $Y \in \CR$.}
\label{fig:synthethic_problems}
\end{figure}
\paragraph{General setup.}
We consider the following three curves
\begin{align*}
\texttt{Line: }\quad \Im(\gamma) &= \{ta : a = (1,1,1),\ t \in [0,1]\},\\
\texttt{S-Curve: }\quad \Im(\gamma) &= \{(\cos(t), \sin(t)) : t \in [-\pi/2,0]\} \cup \{(2-\cos(t), \sin(t)) : t \in [0,\pi/2]\},\\
\texttt{Helix: }\quad \Im(\gamma) &= \{(\cos(t/\sqrt{2}), \sin(t/\sqrt{2}), t/\sqrt{t}) : t \in [0,2\pi] \},
\end{align*}
and embed them into $\bbR^D$ for $D \in \{4,8,12\}$. We set $X = V + F(V)U$, where $V$
is sampled uniformly on $\Im(\gamma)$, $U$ is sampled uniformly on
$\CB_{\N{\cdot}}(0, 0.25)$, and the rows of $F(V) \in \bbR^{D\times (D-1)}$
form an orthonormal basis for the normal space of $\Im(\gamma)$ at $V$.
Examples of such marginal distributions are illustrated in the top row of Figure \ref{fig:synthethic_problems}.
The target function $g \circ \gamma^{-1}$ is a strictly monotonic, piecewise quadratic polynomial.
We set $Y = g\circ \gamma^{-1}(\pi_{\gamma}(X)) + \varepsilon$ with $\varepsilon \sim \Uni{[-\sigma_{\varepsilon},\sigma_{\varepsilon}]}$.
Different noise levels are used: $\sigma_{\varepsilon} = c\Delta f$ with $c \in \{0\}\cup\{ 10^{-\ell} : \ell = 1,\ldots,4\}$ and $\Delta f:= (\max_i f(X_i) - \min_i f(X_i))\lengamma^{-1}$.

Parameter selection for the NSIM estimator is guided by Section \ref{sec:function_estimation}.
Namely, we use $k=1$ and $J=(15D)^{-1}N$ if $\sigma_{\varepsilon} = 0$,
and $k=1/2N^{2/3}$ with cross-validation over $J \in \{2^\ell : \ell \in [13]\}$ in the noisy case.
Furthermore, the  restricting radius for the nearest neighbor search is  $\tempradius = 0.5$.
We also train an ordinary kNN-regressor with $k=1$ in the noise-free case,
and $k=1/2N^{2/3}$ in noisy case, to demonstrate that in these problems ordinary kNN-regression indeed
suffers from the curse of dimensionality.

For evaluating the NSIM estimator, we report the root mean squared errors (RMSE)
\begin{align*}
\textrm{RMSE}(\hat f - f) := \sqrt{\frac{\sum_{m=1}^{1000}\left(\hat f(Z_m) - f(Z_m)\right)^2}{\sum_{m=1}^{1000}f(Z_m)^2}},\quad
\textrm{RMSE}(\hat a - a) := \sqrt{\frac{1}{J}\sum_{j=1}^{J}\N{\hat a_j - a_j}^2},
\end{align*}
where $\{Z_m : m \in [1000]\}$ are test samples iid. from $\rho_X$, and $J = J(N)$ is
chosen as described above. The results are averaged over
20 repetitions of the same experiment. {The standard deviation is indicated by vertical bars.}

\begin{figure}[]\noindent
\subfloat[NSIM for \texttt{Line}]{\includegraphics[width=0.33\linewidth]{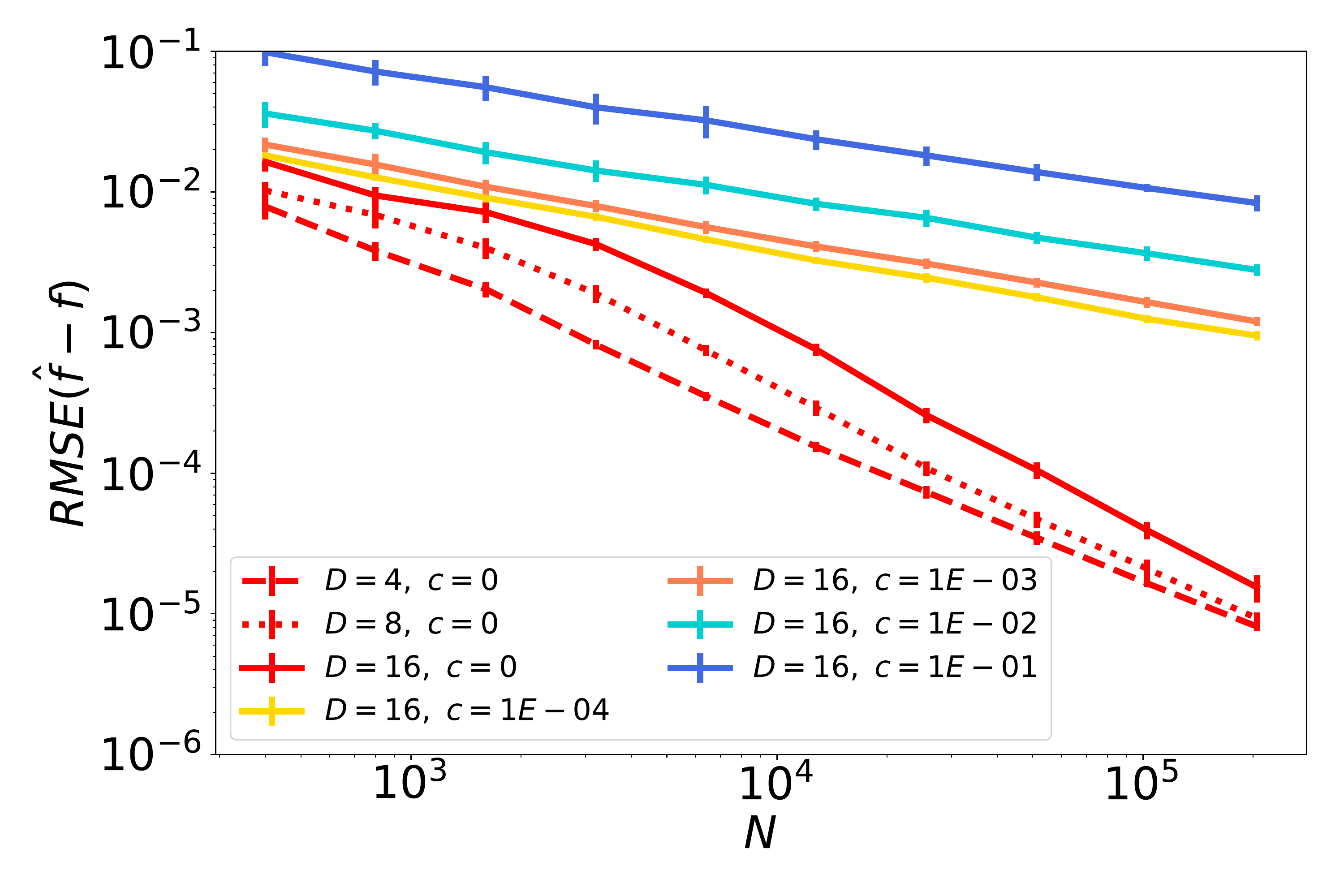}\label{fig:nsim_id_function}}
\subfloat[NSIM for \texttt{S-curve}]{\includegraphics[width=0.33\linewidth]{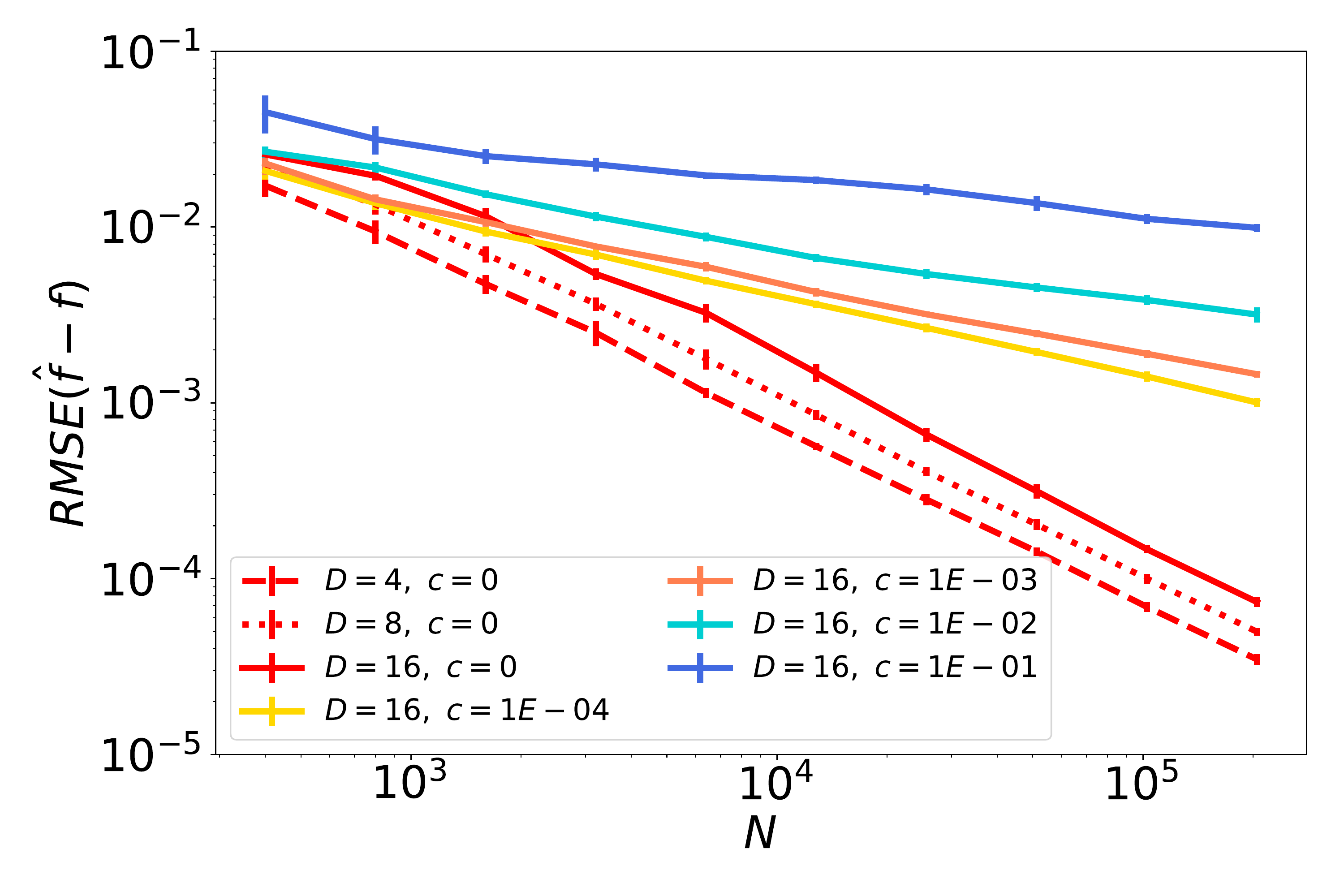}\label{fig:nsim_scurve_function}}
\subfloat[NSIM for \texttt{Helix}]{\includegraphics[width=0.33\linewidth]{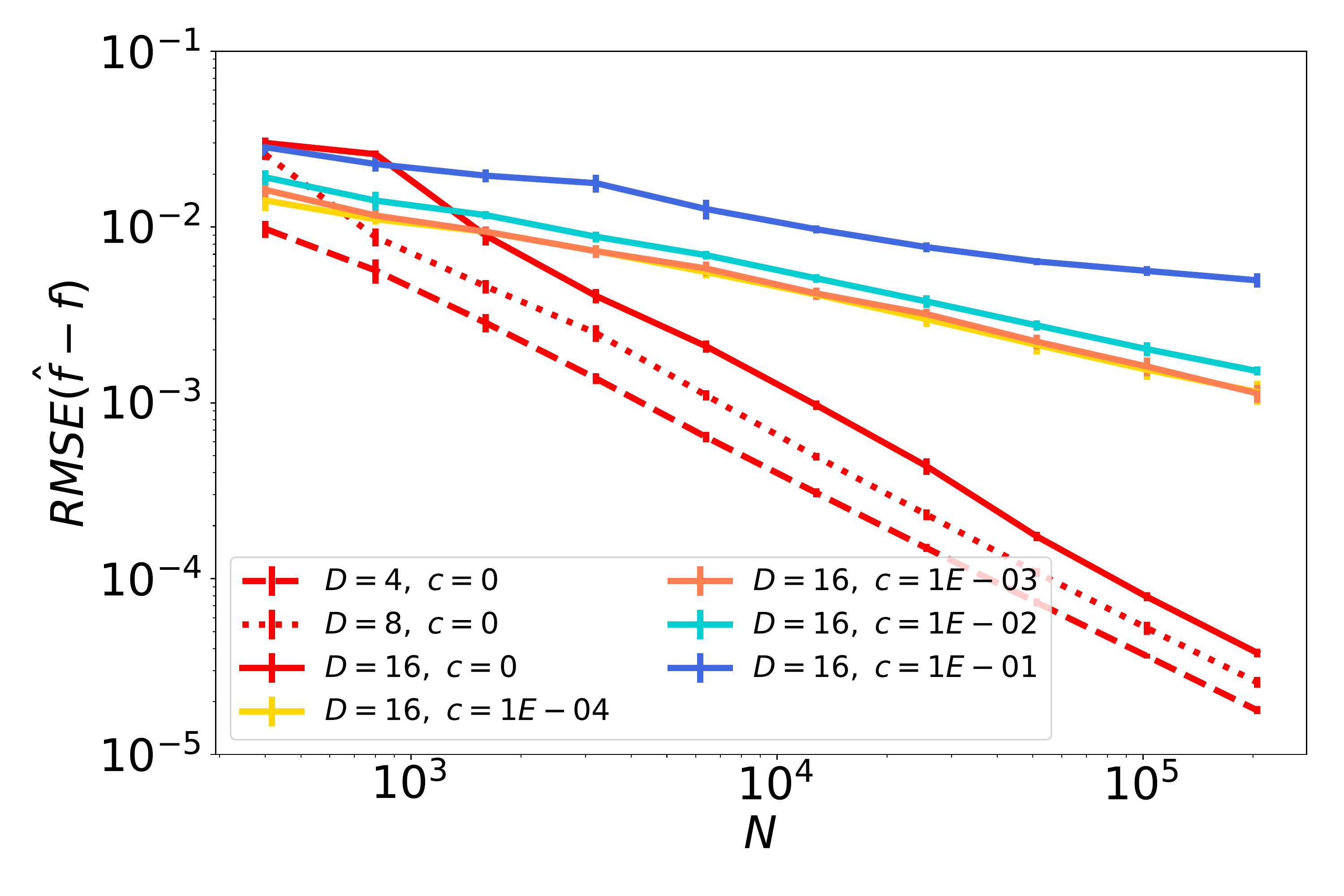}\label{fig:nsim_helix_function}}\\ \noindent
\subfloat[NSIM for \texttt{Line}]{\includegraphics[width=0.33\linewidth]{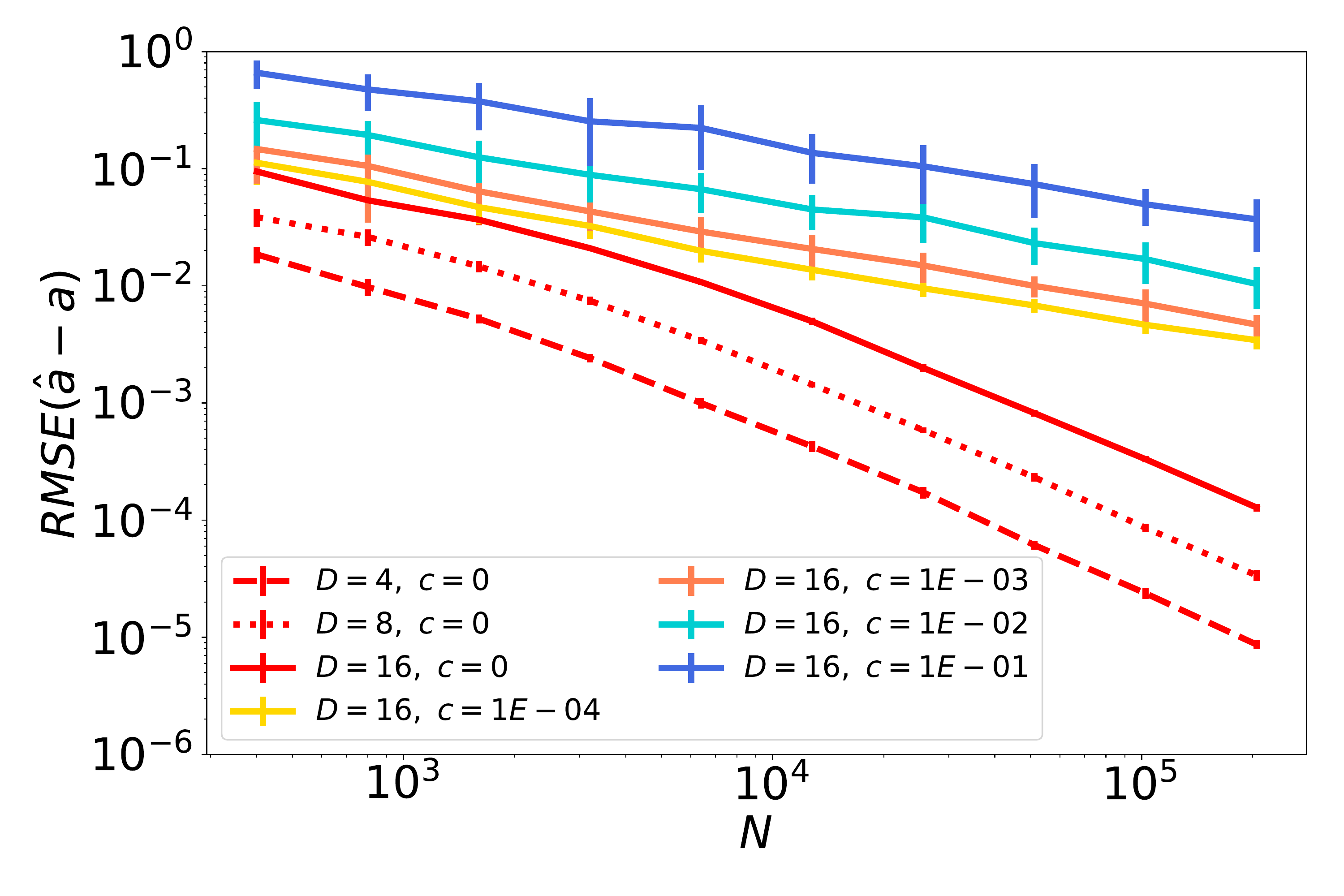}\label{fig:nsim_id_tangent}}
\subfloat[NSIM for \texttt{S-curve}]{\includegraphics[width=0.33\linewidth]{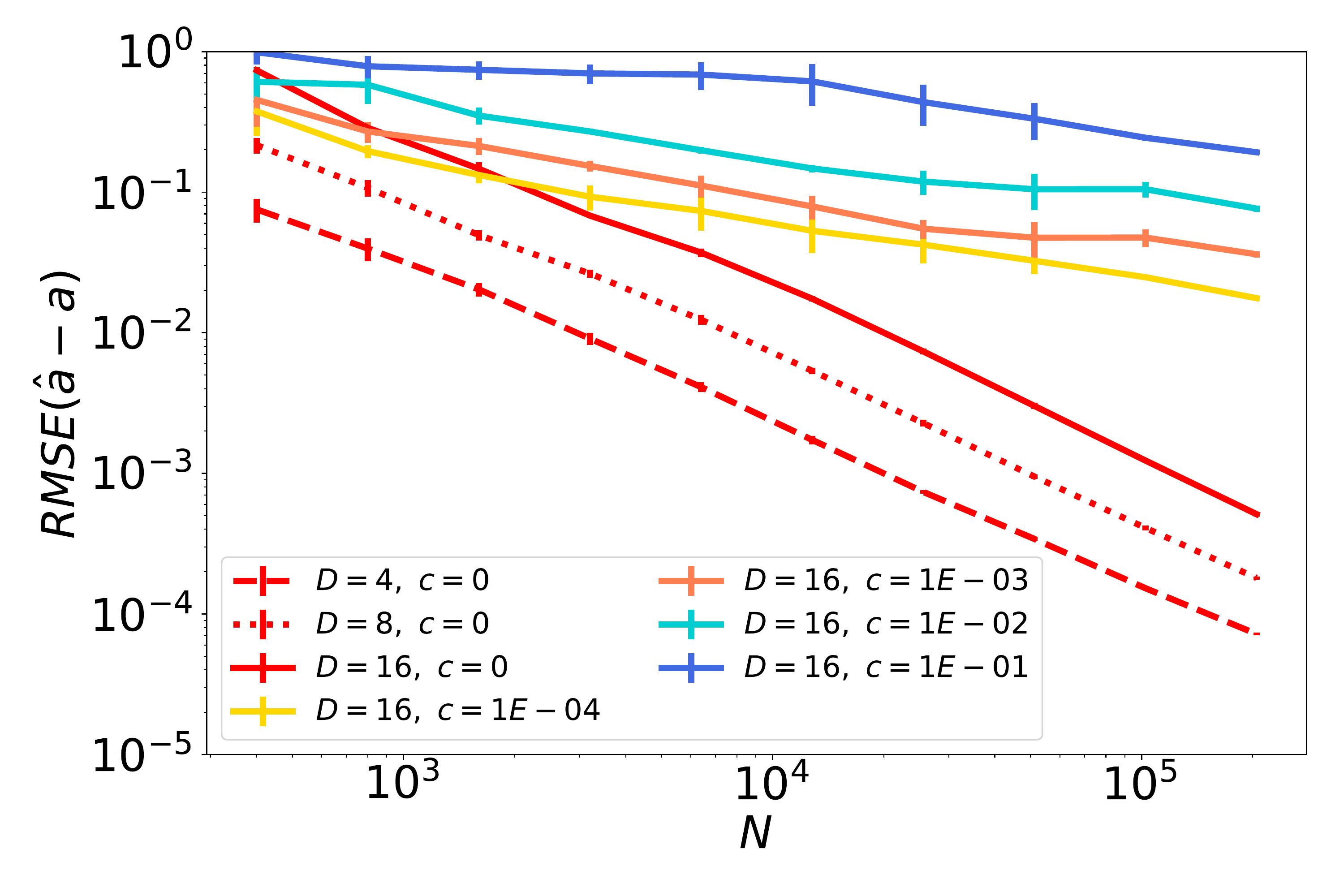}\label{fig:nsim_scurve_tangent}}
\subfloat[NSIM for \texttt{Helix}]{\includegraphics[width=0.33\linewidth]{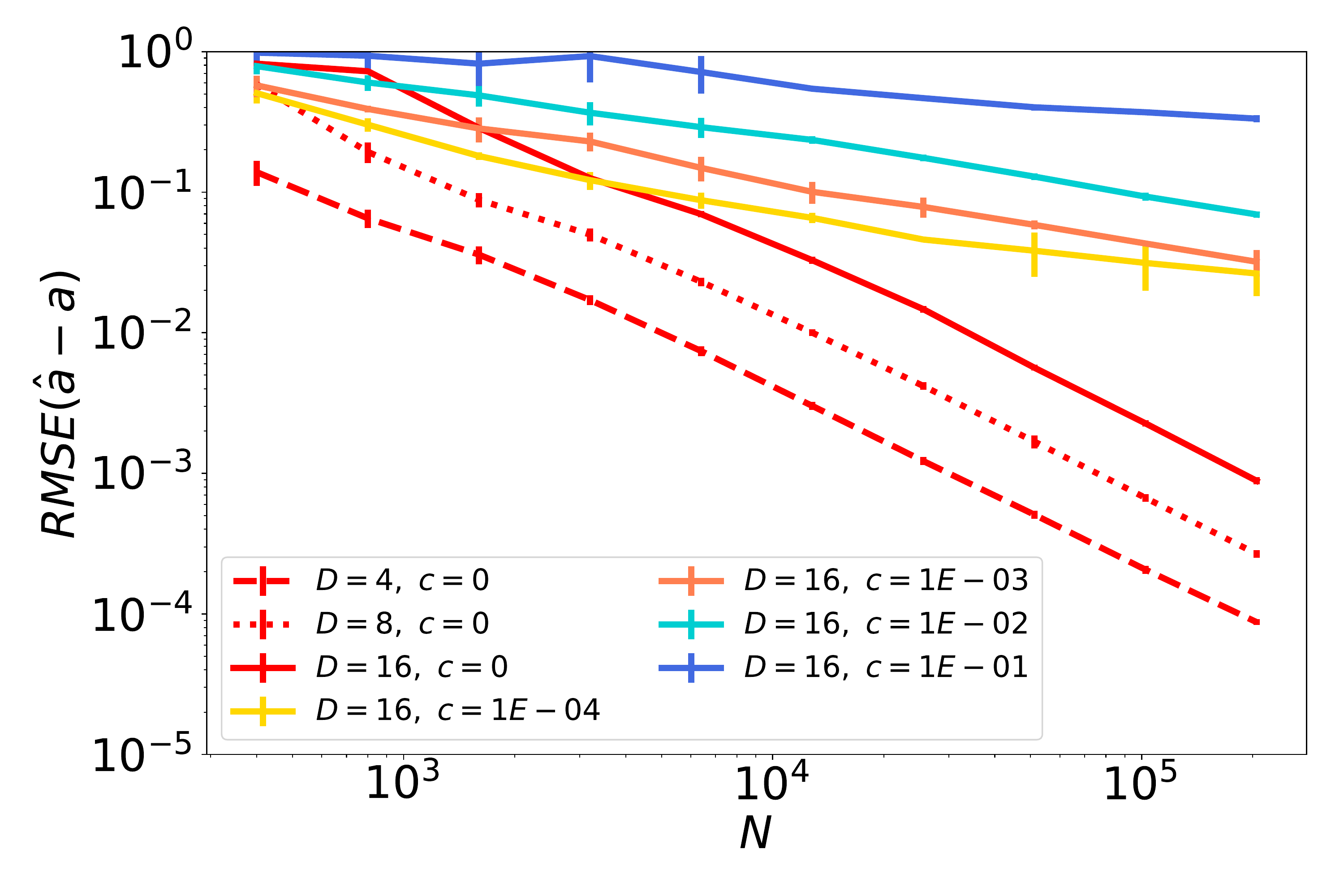}\label{fig:nsim_helix_tangent}}\\\noindent
\subfloat[kNN for \texttt{Line}]{\includegraphics[width=0.33\linewidth]{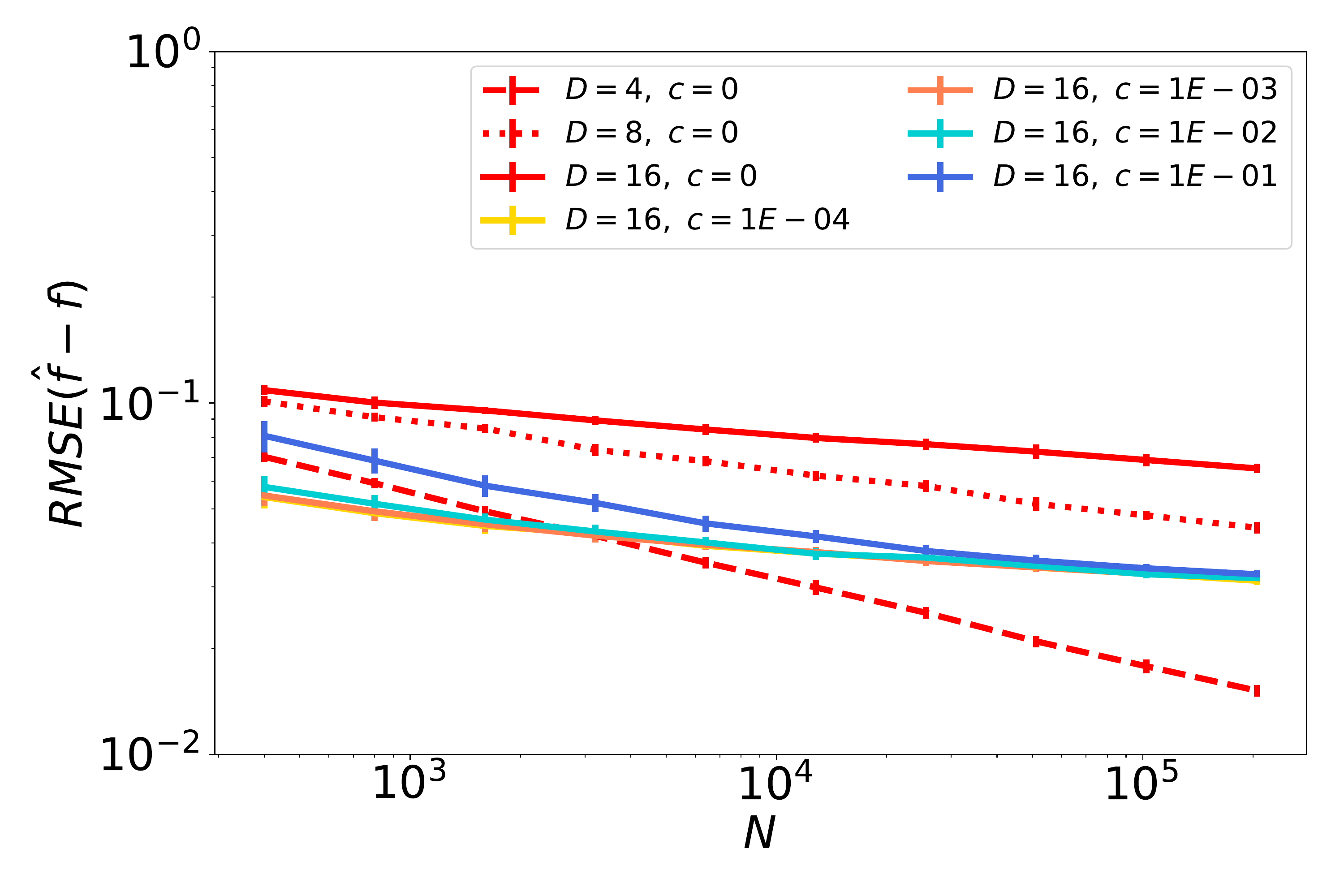}\label{fig:knn_id}}
\subfloat[kNN for \texttt{S-curve}]{\includegraphics[width=0.33\linewidth]{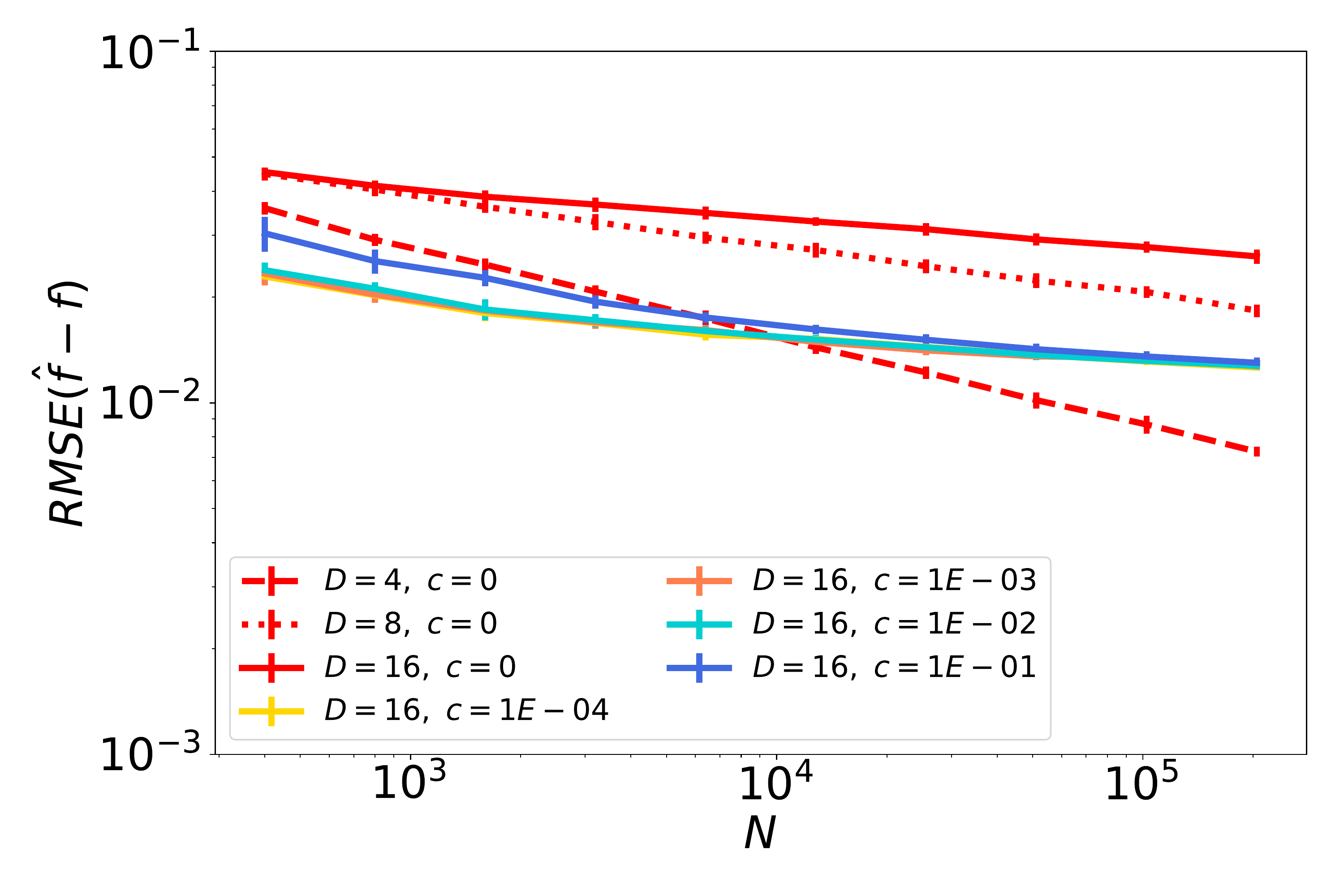}\label{fig:knn_scurve}}
\subfloat[kNN for \texttt{Helix}]{\includegraphics[width=0.33\linewidth]{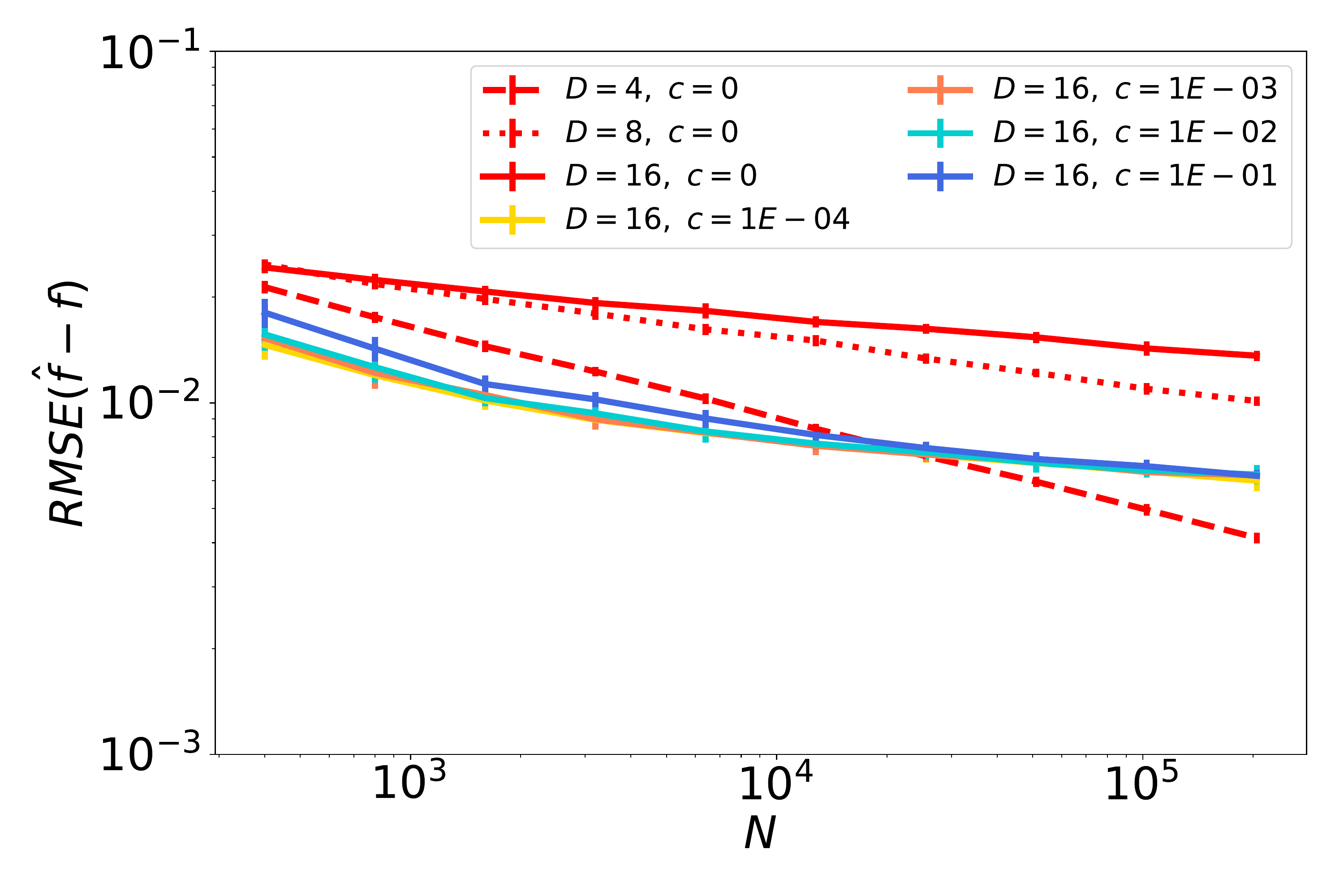}\label{fig:knn_helix}}
\caption{Error decay with respect to $N$ for NSIM and ordinary kNN-estimation. Red lines correspond to
noisefree cases, where the ambient dimension $D$ is varied. Other colors represent
different noise levels, and $D = 12$. Top: RMSE for function estimation when using the proposed NSIM estimator in Algorithm \ref{alg:main}.
Middle: RMSE for localized index vectors over all level sets.
Bottom: RMSE for function estimation when using ordinary kNN-regression.}
\label{fig:results_synthethic_problems}
\end{figure}

\paragraph{Discussion}
The results of our studies are presented in Figure \ref{fig:results_synthethic_problems}.
In red plots, which correspond to cases with $\sigma_{\varepsilon} = 0$,
we observe a $N^{-1}$ decay of the function error (Figures \ref{fig:nsim_id_function} - \ref{fig:nsim_helix_function}),
and similarly a $N^{-1}$ decay of the tangent field error
(Figures \ref{fig:nsim_id_tangent} - \ref{fig:nsim_helix_tangent}).
In particular, the ambient dimension $D$ affects the error only in terms of a multiplicative
constant but not in the rate of decay. Therefore, the NSIM estimator does not suffer
from the curse of dimensionality, which is not the case for ordinary kNN-regression as
shown in Figures \ref{fig:knn_id} - \ref{fig:knn_helix}.

The remaining plots in Figure \ref{fig:results_synthethic_problems} represent  noisy cases,
where the highest noise level corresponds to blue lines. Considering the first column,
where $\Im(\gamma)$ is a straight line, and therefore the data follows an ordinary SIM,
we see that the error for function and index vector estimation steadily decreases
at a $N^{-1/3}$ rate. This confirms our theoretical result, \emph{i.e.}, the NSIM estimator
is consistent, and achives the optimal rate, in case of an ordinary SIM. If we have a curved geometry
and function noise on the other hand, errors for function prediction and tangent field
estimation stall after reaching a certain quality. This can be seen e.g.
in the blue plots in Figures \ref{fig:nsim_scurve_function} and \ref{fig:nsim_helix_function}.

We remark here that estimators, that are used for comparison on real data sets in
the next section, have been tested on these synthethical problems as well. We omit corresponding results
because none show any improvement as the sample size $N$ increases (apart from SIM estimators
and the $\texttt{Line}$ problem). This is expected for SIM estimators because they can not
resolve the underlying nonlinear geometry during training.

\subsection{Real data}
\label{subsec:realdata_experiments}
We will now test the NSIM algorithm
and compare it to other commonly used algorithms on a variety of real-worlds data sets.
We report the mean RMSE 
and its standard deviation over 30 repetitions of each experiment.
In each run, we use $15\%$ of the data as the test set, and we tune hyper-parameters
for each estimator using 5-fold cross-validation on exhaustive parameter grids.

\paragraph{Data sets.}
We use 6 UCI data sets (\texttt{Air Quality}, \texttt{Boston Housing},
\texttt{Concrete}, \texttt{Istanbul Stock Exchange}, \texttt{Skillcraft1},
\texttt{Yacht}) and the \texttt{Ames Housing} data set in our study.
For each data set, the components of $X$ are standardized and we exclude clearly irrelevant features.
Moreover, if the marginal of $\tilde Y = \log(Y)$ resembles the uniform distribution better (compared to $Y$), we use $\tilde Y$ instead of $Y$.
The preprocessed the data sets are readily available at
\url{https://github.com/soply/db_hand}.

\setlength{\tabcolsep}{9pt}
\begin{table}[h!]
\scriptsize
\begin{tabular}{@{}cccccccc@{}}
\centering
Characteristics& {Yacht} & {Istanbul} & {Ames} & {Concrete} & {Air Quality} &  {Boston} & Skillcraft\\ \toprule
{$\log$-TF} & {Yes} & {No} & {Yes} & {No} & {No} & {Yes} & {Yes} \\ \midrule
{$D,\ N$} & $6,\ 307$& $7,\ 536$& $7,\ 1197$ & $8,\ 1030$ & $11,\ 7393$ & $12,\ 506$ & $16,\ 3338$\\ \midrule
{Factor} & $10^{1}$ & {$10^{-2}$} & $10^5$ & $10^{1}$ & $10^{-1}$ & $10^1$ & $10^2$\\ \midrule
{$\bar Y\pm\textrm{STD}(Y)$} & $1.05 \pm 1.51$ &  $0.16 \pm 2.11$   & {$1.74 \pm 0.67$} &
{$3.58\pm1.67$} & {$9.95 \pm 4.03$} & {$1.27 \pm 0.71$} & {$1.15 \pm 0.48$} \\\bottomrule \\
{Method} & \\ \toprule
{NSIM-dyad} & {$0.15\pm0.04$} & $1.52\pm0.14$  & $\highlight{0.23\pm0.04}$ & $0.9\pm0.06$ & $0.82\pm0.04$ & $\highlight{0.42\pm0.06}$ & $0.08\pm0.01$\\
$k$ & {$11.6$} & {$19.9$} & {$14.6$} & {$46.0$} & {$60.8$} & {$33.1$} & {$19.0$} \\
$J$ & {$2.4$} & {$1.1$} & {$2.4$} & {$3.9$} & {$5.6$} & {$1.0$} & {$4.1$} \\ \midrule
{NSIM-stat} & $\highlight{0.12\pm0.03}$ & $\highlight{1.39\pm0.18}$  & $\highlight{0.23\pm0.03}$ & $0.97\pm0.06$ & $0.80\pm0.02$ & $\highlight{0.42\pm0.04}$ & $0.08\pm0.01$\\
$k$ & {$8.6 $} & {$19.3$} & {$18.2$} & {$41.6$} & {$69.3$} & {$43.0$} & {$17.7$} \\
$J$ & {$5.5 $} & {$1.0$}  & {$3.1$} & {$2.7$} & {$5.3$} & {$1.0$} & {$5.2$} \\ \midrule
{Lin-Reg} & $0.22\pm0.07$ & $\highlight{1.38\pm0.13}$  & $\highlight{0.23\pm0.02}$ & $1.06\pm0.06$ & $1.22\pm0.03$ & $0.50\pm0.11$ & $0.14\pm0.03$\\ \midrule
{kNN}  & $0.76\pm0.11$ & $1.52\pm0.16$ & $0.26\pm0.03$ & $0.89\pm0.08$ & $1.03\pm0.02$ & $\highlight{0.41\pm0.06}$ & $0.17\pm0.01$ \\
$k$ & {$1.1$} & {$17.8$} & {$9.8$} & {$5.5$} & {$25.0$} & {$6.8$} & {$9.8$} \\ \midrule
{SIR-kNN} &$0.26\pm0.11$ & $1.48\pm0.16$ & $0.25\pm0.03$ & $1.05\pm0.06$ &$1.87\pm0.04$ & $0.47\pm0.05$ & $0.17\pm0.01$\\
$k$ & {$10.4$} & {$21.7$} & {$20.0$} & {$48.4$}  & {$137.5$} & {$43.5$} & {$37.1$} \\
$J$ & {$10.8$} & {$7.4$}  & {$21.8$} & {$3.0$} & {$4.8$} & {$8.5$} & {$25.6$} \\ \midrule
{Isotron} &$0.15\pm0.05$ & $1.42\pm0.11$  & $\highlight{0.24\pm0.03}$ & $1.03\pm0.05$ & $0.83\pm0.03$ & $\highlight{0.42\pm0.05}$ & $0.08\pm0.01$\\
Iterations & {$460.0$} & {$343.75$} & {$338.75$} & {$392.5$} & {$596.25$} & {$280.0$} & {$425.0$} \\ \midrule
{ELM-Sig} & $0.44\pm0.30$ & $1.46\pm0.15$  & $\highlight{0.23\pm0.04}$ & $0.72\pm0.05$  & $0.58\pm0.12$ & $0.44\pm0.06$ & $0.20\pm0.04$\\
Nodes & {$88.8$} & {$15.2$}  & {$54.0$} & {$86.3$} &  {$91.8$} & {$46.2$} & {$77.1$} \\ \midrule
{SNN-Tan} & $0.48\pm0.20$ & $1.61\pm0.21$  & $0.25\pm0.04$ & $0.80\pm0.07$ & $\highlight{0.14\pm0.04}$ & $\highlight{0.41\pm0.05}$ & $\highlight{0.04\pm0.01}$\\
Nodes & {$9.4$} & {$3.0$} & {$18.95$} & {$15.1$} & {$15.95$} & {$13.0$} & {$14.0$} \\ \midrule
{SNN-Sig} & $0.30\pm0.11$ & $1.65\pm0.27$ & $\highlight{0.23\pm0.03}$ & $\highlight{0.63\pm0.05}$ & $0.18\pm0.02$ & $\highlight{0.41\pm0.05}$ & $\highlight{0.04\pm0.00}$\\
Nodes & {$13.0$} & {$3.9$}  & {$8.1$} & {$16.9$}& {$21.5$} & {$7.5$} & {$10.4$} \\
\end{tabular}
\caption{RMSE, standard deviation, and cross-validated hyper-parameters, over $30$ repetitions for several estimators and real-world data sets.
 Values for $k$, $J$, and for numbers of iterations and nodes, are averages over different runs of each experiment.
First $5$ rows describe the data sets and their characteristics, and the remaining rows contain the results.
For a simplified presentation, we divide the mean and STD of RMSE, and the mean and STD of the data (5th row)
by the value in row \emph{Factor}.}
\label{tab:real_data_experiments}
\end{table}

\paragraph{Estimators.}
\begin{itemize}
\item \texttt{NSIM-dyad}, respectively, \texttt{NSIM-stat} refer to Algorithm \ref{alg:main} using a dyadic partition, respectively statistically equivalent blocks.
$k$ and $J$ are chosen via cross-validation. The radius is the intersecting Euclidean ball is determined by $\eta=\infty$.
\item \texttt{Lin-Reg} and \texttt{kNN} are standard linear regression and kNN-regression.
\item \texttt{SIR-kNN} uses sliced inverse regression \cite{li1991sliced} to find an index vector $a$, and then
kNN on projected samples $(a^\top X, Y)$.
Replacing \texttt{SIR} by \texttt{SAVE} \cite{dennis2000save} uniformly worsens the results.
\item \texttt{Isotron}, \cite{kalai2009isotron}, iteratively fits the link function $g$ using isotonic regression  \cite{isotonicregression}
on projected samples $(a^\top X, Y)$, and then updates the index vector $a$. The iteration is initialized with $a=0$ and stopped
when the validation error stalls on a hold-out set.
\item \texttt{ELM-Sig}, \cite{huang2006extreme}, is a shallow neural network with sigmoid activation where inner biases and weights are randomly sampled, and only the outer layer is trained on data. This can be done by solving a simple linear system, which makes the
algorithm very efficient. {We also tested the hyperbolic tangent activation function, but the results were uniformly worse.}
\item \texttt{SNN-Tan} and \texttt{SNN-Sig} are standard shallow neural networks with hyperbolic tangent and sigmoid activation functions, respectively.
We train them using stochastic gradient descent (learning rate $0.01$), and stop the iteration when the validation error stalls on an inner validation set. As for ELM, we use 5-fold cross-validation for the number
of hidden nodes.
\end{itemize}
\paragraph{Discussion.}
The results are presented in Table \ref{tab:real_data_experiments}.
It is helpful to divide these estimators into two groups. The first group consists of simple estimators (kNN and linear regression) and of estimators that use a reduced ($1D$) representation of the data (NSIM, SIR and Isotron).
The second group are shallow neural networks which search for an estimator in a considerably richer class of functions.
Among the first group, NSIM variants achieve very convincing results as they always belong to the best performing group of
estimators. Moreover, experiments suggest that our approach adapts well to the complexity of a given data set. For example,
on a data set where linear regression performs best (\texttt{Istanbul}), NSIM achieves
roughly the same performance, and automatically chooses (most of the time) $J =1$.
On the other hand, for the \texttt{Concrete} data set, where all models that use a single index vector perform rather poorly,
the added model flexibility of the NSIM approach proves beneficial, and we achieve the same performance
as kNN, despite reducing the dimensionality. This is not the case for \texttt{SIR-kNN} and
\texttt{Isotron}, both of which use a \textit{linear} $1D$ projection.
{Finally, on \texttt{Air Quality} and \texttt{Yacht}, \texttt{NSIM-stat} achieves superior performance
 while leveraging the enhanced model flexibility with $J \approx 5$ level sets.}

Estimators in the second group enjoy a greater model flexibility, but are at the same time more prone to overfitting. For data sets
with a lot of samples (\texttt{Air Quality}, \texttt{Concrete}, and \texttt{Skillcraft}), these
methods are better than the estimators in the first group.
On the other hand for data sets with smaller sample sizes (\texttt{Istanbul} and \texttt{Yacht}), the model
can not be fitted easily, and we observe exactly the opposite effect.
Considering the results for the \texttt{Ames} data set, all estimators perform roughly the same.

\begin{figure}[]
  \centering
\subfloat[\texttt{Air Quality}]{\includegraphics[trim={0cm 4.cm 0 4.cm},clip, width=0.8\linewidth]{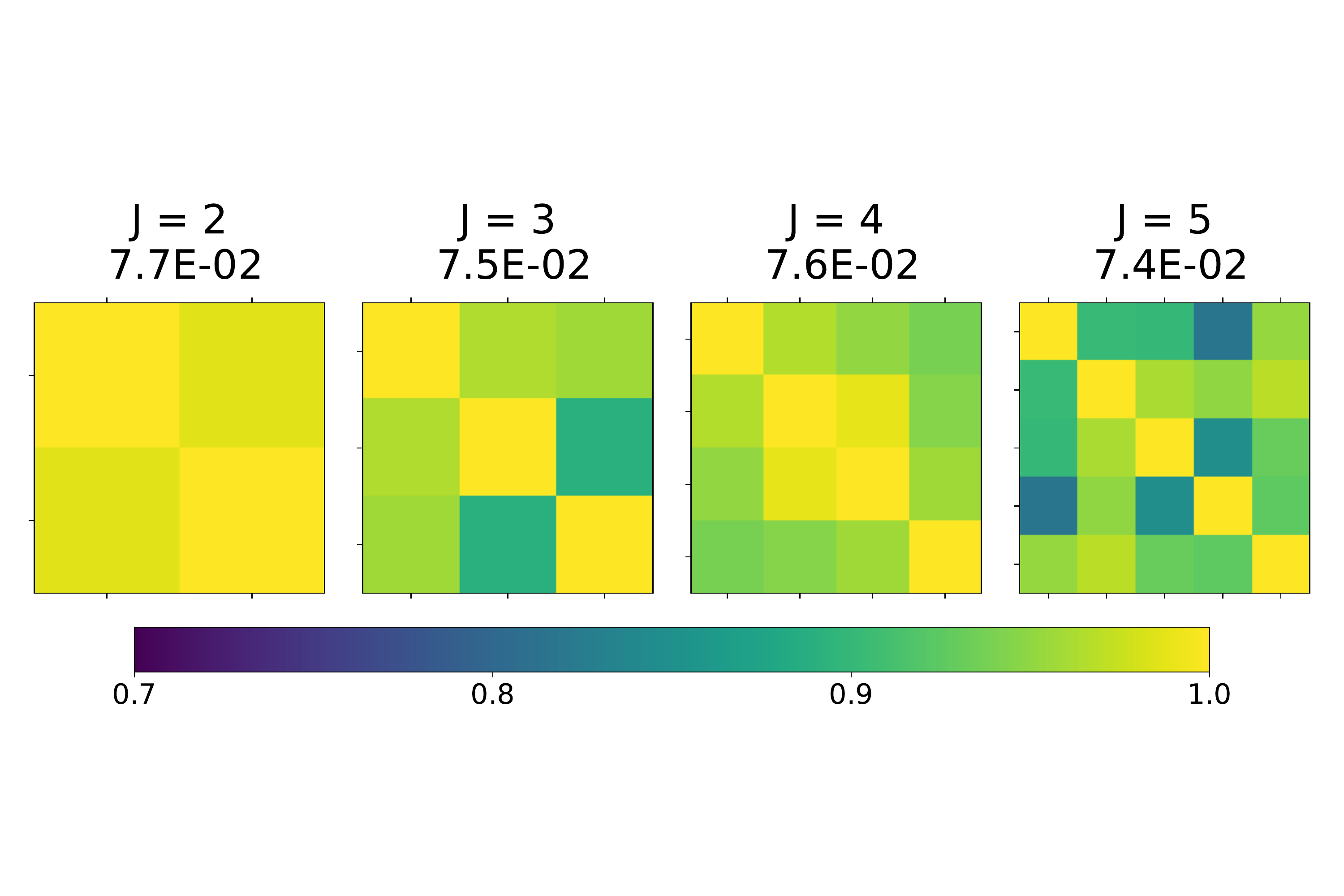}\label{fig:airquality}}\\\noindent
\subfloat[\texttt{Concrete}]{\includegraphics[trim={0cm 4.cm 0 4.cm},clip, width=0.8\linewidth]{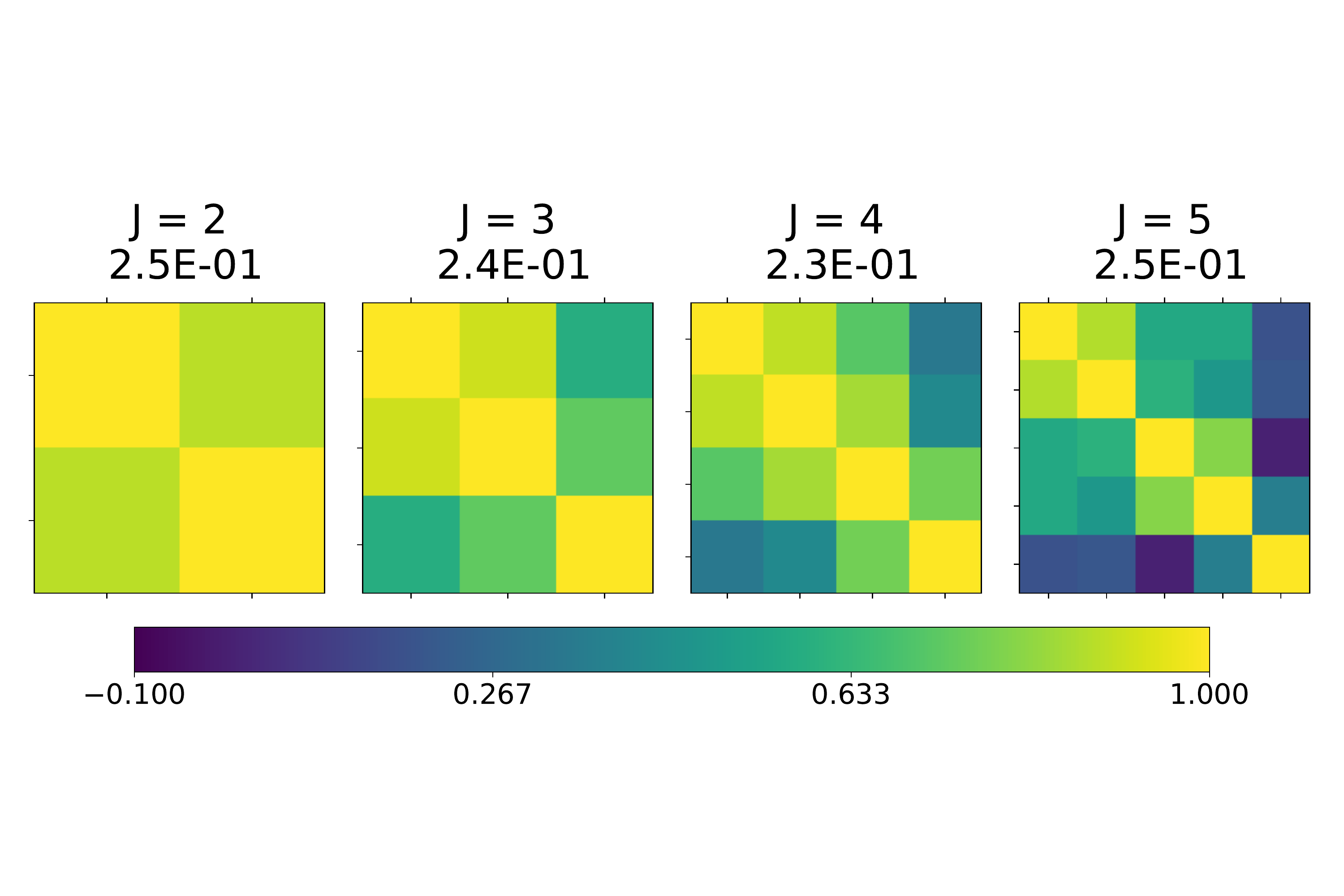}\label{fig:concrete}}\\\noindent
\subfloat[\texttt{Skillcraft}]{\includegraphics[trim={0cm 4.cm 0 4.cm},clip, width=0.8\linewidth]{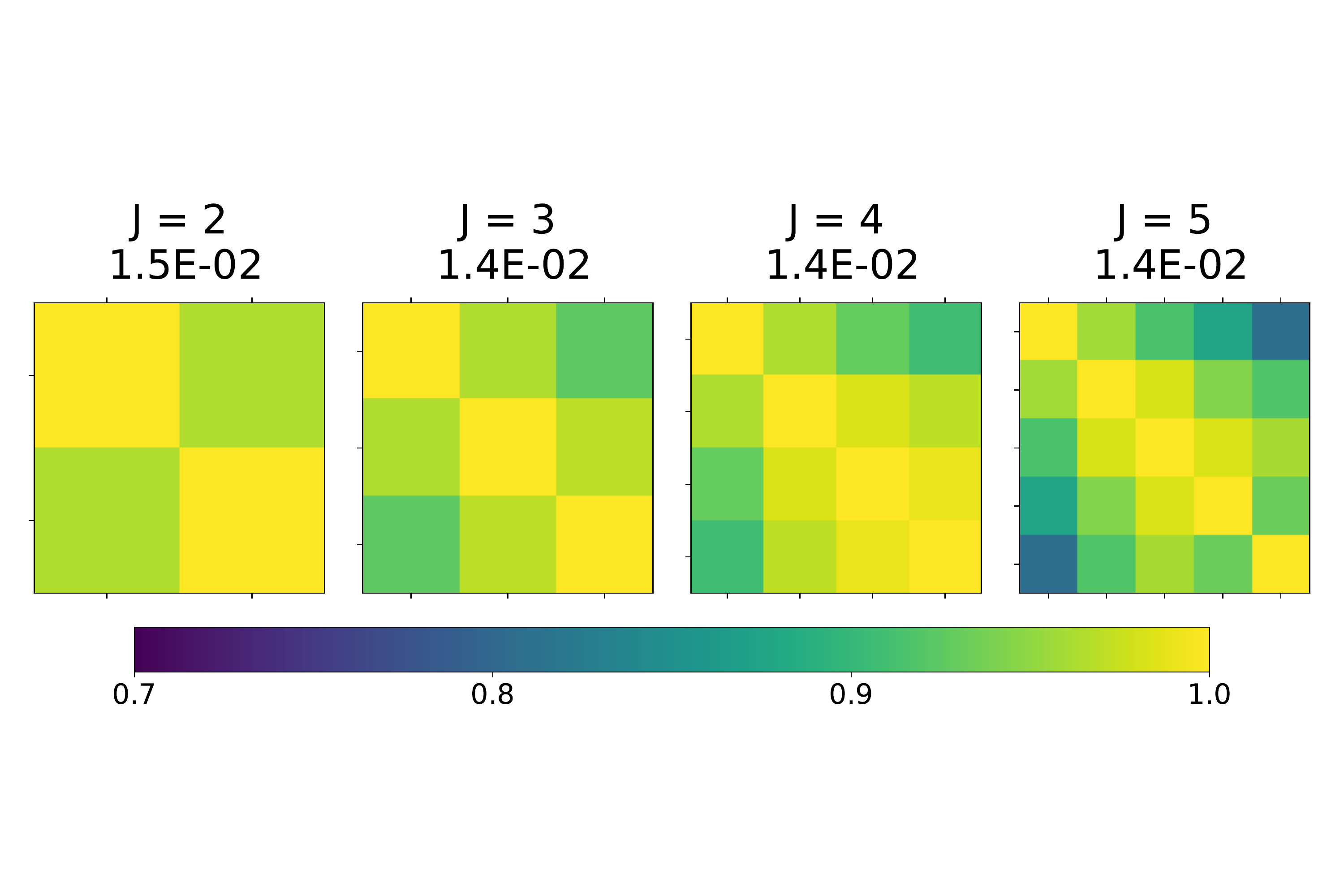}\label{fig:skillcraft}}
\caption{The Grammian matrices $G \in \bbR^{J \times J}$ of local index vectors for different data sets and different parameters $J$. The second value in the title is the relative prediction error when using parameter $k$ as cross-validated in Table \ref{tab:real_data_experiments}. Similarity of two local index vectors $\hat a_j$  and $\hat a_i$ for $(X,Y) | Y \in \CR_i$ and $(X,Y) | Y \in \CR_j$ is implied if the $(i,j)$ entry of $G$ is close to one. Since neighboring entries are conditioned on neighboring sets $\CR_j$, the similarity is usually inverse proportional to the distance  $\SN{i-j}$.}
\label{fig:interpretability}
\end{figure}

\paragraph{Interpretability.}
An important feature of the SIM is its interpretability, because the recovered index vector
describes the relationship between each feature and the response $Y$. Namely, the $i$-th entry
of the index vector $\hat a$ should have a large magnitude if the corresponding
feature has a strong influence on $Y$ (relative to other features), and its sign
indicates if the feature increases or decreases Y (when keeping other entries fixed).
NSIM retains these properties and allows for a more refined analysis, since it
considers conditional distributions, $X|\CR_j$, for different ranges of the response.
By inspecting and comparing local index vectors we can thus analyze whether the
influence of features changes across different regimes.

To that end, we propose to study off-diagonal entries of Grammian matrices,
$G \in \bbR^{J\times J}$, where $G_{ij}=\hat a_i^\top \hat a_j$, after fitting the model for a range of $J$’s.
If $G_{ij}\approx 1$ everywhere, and for all $J$, then the model most likely follows
the traditional, monotone SIM. On the other hand, if roughly $G_{ij}\asymp \SN{i-j}^{-1}$,
then local index vectors indeed vary, with certain regularity, as a function of $Y$.

In Figure \ref{fig:interpretability} we plot the results for this method on
\texttt{Air Quality}, \texttt{Concrete}, and \texttt{Skillcraft} data sets.
We see that the the pair-wise similarity $G_{ij}$ is indeed inverse proportional to $\SN{i-j}^{-1}$,
suggesting that NSIM fits the data better than SIM. Results in Table \ref{tab:real_data_experiments}
confirm this, by showing that NSIM outperforms SIM-based estimators (\texttt{Lin-Reg}, \texttt{SIR-kNN}, and \texttt{Isotron}).


\section{Conclusions}\label{sec:conclusion}
In this paper we propose a nonlinear relaxation of the single index model for data sets with inherent monotonicity between
features and outputs. We propose to estimate the model by combining localization through level set partitioning,
local linear regression and a kNN-regressor for out-of-sample prediction.
Our theoretical results provide guarantees on the error of the quantization of the tangent field of $\Im(\gamma)$, and yield guarantees for out-of-sample prediction.
In the noise free case we provide optimal learning rates, while in the noisy case we generally have a biased estimator. If the NSIM reduces to
a SIM, \emph{i.e.} if $\Im(\gamma)$ is a straight line, we recover the optimal learning rates for estimating the SIM also in the noisy case.

Our numerical experiments show that the NSIM estimator yields superior results when compared to estimators of similar model complexity.
Moreover, the estimator outperforms shallow neural network models on data sets with rather few samples.
On the other hand, if the data sets are sufficiently rich to properly fit shallow networks models, their
additional flexibility pays off and NSIM does not achieve similar predictive accuracy.
Consequently, our future research direction aims at further enhancing the model space of our estimator,
by replacing kNN with more sophisticated regressors and learning multiple index
vectors, \emph{i.e. multi index models}, in each level set.
\section*{Supplementary Materials}
Code to replicate the experiments in the article is available at IMAIAI online.
\section*{Funding}
This work was supported by the Research Council of Norway [251149/O70 to V.N.].
\section*{Acknowledgements}
T.K. thanks Prof. Mauro Maggioni, Stefano Vigogna and Alessandro Lanteri for helpful
discussions about the project.


\bibliographystyle{alpha}
\ifx\undefined\BySame
\newcommand{\BySame}{\leavevmode\rule[.5ex]{3em}{.5pt}\ }
\fi
\ifx\undefined\textsc
\newcommand{\textsc}[1]{{\sc #1}}
\newcommand{\emph}[1]{{\em #1\/}}
\let\tmpsmall\small
\renewcommand{\small}{\tmpsmall\sc}
\fi

\appendix
\section{Appendix}
\subsection{Proofs for Section \ref{sec:geometry}}
\label{subsec:nsim_app_proofs_sec_3}
This section is split into two parts. The first concerns a local analysis and
establishes Theorem \ref{thm:midpoint_estimation}.
The second part deals with the global analysis and proves Corollary \ref{cor:tangent_bound_all}.

\subsubsection{Local analysis}
\label{subsubsec:appendix_local_analysisz}
Before we begin with the proof of Theorem \ref{thm:midpoint_estimation} we collect some
required auxiliary results. All these results describe local phenomena, which means we
can consider consider a fixed, arbitrary closed interval $\CR \subset [0,1]$
with corresponding minimal $\CS := \CS(\CR) \subset \Im(\gamma)$ such that $\bbP(V \in \CS(\CR) | \CR) = 1$.
We denote $\bar t := \bbE[t|\CR]$, $a := \gamma'(\bar t)$, $\Ptan := aa^\top$, $\Pperp := \Id - \Ptan$.  For notational simplicity,
we do not use a subscript $\CR$ for e.g. $\Sigma, \kappa, N$ and so on, but keep in mind that
all quantities are understood locally.
We use $\lesssim$ to absorb universal numeric constants.

\paragraph{Auxiliary results}
The following result shows that the length of $\CR$ and $\CS$ are equivalent up to
the Lipschitz constant $L_f$, and provided $\CR \gg \sigma_{\varepsilon}$.

\begin{lemma}
\label{lem:bound_curve_segment}
Take an interval $\CR \subset \Im(f)$ and let $\CS \subset \Im(\gamma)$ be the shortest segment such that $\bbP(V \in \CS | Y \in \CR) = 1$.
Then $L_f^{-1} (\ssler - 2\sigma_{\varepsilon}) \leq \sslen \leq L_f (\ssler + 2\sigma_{\varepsilon})$.
\end{lemma}
\begin{proof}
For any $V, V'\in\CS$ we have $\SN{Y-Y'} - 2\sigma_\epsilon\leq \SN{f(V)-f(V')}\leq \SN{Y-Y'} + 2\sigma_\epsilon$, almost surely, where $Y,Y'$ are such that $Y=f(V)+\epsilon$, $Y'=f(V')+\epsilon'$.
Using \eqref{eq:g_bilipschitz} we have
\[ d_\gamma(V, V') \leq  L_f \SN{f(V)-f(V')}\leq L_f\LRP{\SN{Y-Y'} + 2\sigma_\epsilon}\leq L_f\LRP{\SN{\CR} + 2\sigma_\epsilon},\]
and the upper bound follows after taking the supremum over $V,V'$.
For the converse, taking $(X,Y),(X',Y')$ be such that $\SN{Y-Y'}=\SN{\CR}$, we have
\[ d_\gamma(V, V') \geq L_f^{-1} \SN{f(V)-f(V')} \geq L_f^{-1}\LRP{\SN{\CR} -2\sigma_\epsilon}.\]
\end{proof}
\noindent
Next we provide some basic bounds on spectral properties of the conditional
covariance matrix. We use in the proof that a random vector $Z$ satisfying $\N{Z-Z'} \leq M$
almost surely, where $Z'$ is an independent copy of $Z$, satisfies
$\N{\Covv{Z}} \leq \bbE\N{Z-\bbE Z}^2 = 1/2\bbE\N{Z-Z'}^2 \leq 1/2M^2$.

\begin{lemma}
\label{lem:technical_inequalities}
Let \ref{ass:A5}, \ref{ass:A1} and \ref{ass:A2} hold. Take an interval $\CR \subset \Im(f)$ and let $\CS \subset \Im(\gamma)$ be the shortest
segment such that $\bbP(V \in \CS| \CR) = 1$. Then the following holds:
\begin{align}
\label{eq:variance_bound_PW}
\N{\Covv{\Ptan W|\CR}}&\leq B^2\curvtor^2 \sslen^2 \leq \sslen^2,\\
\label{eq:variance_bound_PV}
\N{\Covv{\Pperp V|\CR}}&\leq \bbE[\N{\Pperp (V-\bbE[V|\CR])}^2|\CR]  \leq 1/2 \curvtor^2 \sslen^{4},\\
\label{eq:variance_bound_PV_V}
\N{\Covv{V, \Pperp V|\CR}} &\leq 1/2 \curvtor \sslen^{3},\\
\label{eq:almost_sure_bounds_X_aX}
\N{X - \mu_X} \leq \boundnormal  +\sslen, \text{ and } &\SN{a^\top(X - \mu_X)} \leq 2\sslen\quad \textrm{almost surely}.
\end{align}
\end{lemma}
\begin{proof}
\noindent For \eqref{eq:variance_bound_PW} we use $a(V) \perp W$ and \ref{ass:A2} to get $\SN{a^\top W} = \SN{(a-a(V))^\top W} \leq \kappa \sslen B$.
Since $\bbE[W|Y] = 0$ by \ref{ass:A5} and \ref{ass:A1} it follows that $\Var{a^\top W|\CR} =\bbE[(a^\top W)^2|\CR] \leq (B\kappa \sslen)^2$.
\noindent For \eqref{eq:variance_bound_PV}, we have by the fundamental theorem of calculus and $\Pperp \perp \gamma'(\bar t)$
\begin{align*}
\N{\Covv{\Pperp V|\CR}}& \leq \frac{1}{2} \bbE\left[\N{\Pperp (V-V')}^2|\CR\right]\leq \frac{1}{2} \bbE\left[\left(\int_{t'}^{t}\N{\Pperp \left(\gamma'(s) - \gamma'(\bar t)\right)}ds\right)^2|\CR\right] \leq \frac{1}{2}\curvtor^2\sslen^4.
\end{align*}
\eqref{eq:variance_bound_PV_V} follows by the Cauchy-Schwarz inequality $\N{\Covv{V, \Pperp V|\CR}} \leq \sqrt{\N{\Covv{V|\CR}}\N{\Covv{\Pperp V|\CR}}}
$, $\N{\Covv{V|\CR}} \leq 1/2\sslen^2$ since $\N{V-V'} \leq \sslen$, and using \eqref{eq:variance_bound_PV}.
\end{proof}

\noindent
While upper bounds for spectral norms of covariance matrices are easily
obtained in the previous Lemma, lower bounds for variances are generally more challenging to establish.
In particular they have to rely on an assumption such as \ref{ass:A3},
which asserts that the marginal distribution of $V$ is (measure-theoretically)
equivalent to the uniform distribution. Our analysis in Section \ref{sec:geometry}
hinges on the relation $\Var{a^\top X, Y|\CR} \asymp \sslen\ssler$. The following two results
show that this is true for example if $f \in \CC^2$ and $\ssler \gg \sigma_{\varepsilon}$.
However we believe that more general conditions just relying on the monotonicity/bi-Lipschitz properties
of $f$ could be established.

\begin{lemma}
\label{lem:covariance_AV}
Let Assumptions \ref{ass:A5}, \ref{ass:A1} and \ref{ass:A3} hold.
For any interval $\CR \subset \Im(f)$ with $\ssler > 2\sigma_{\varepsilon}$ and $\CS \subset \Im(\gamma)$ as the shortest
segment such that $P(V \in \CS|\CR) = 1$ we have
\[
\frac{(1-\kappa \sslen)^2}{27{c_V}^4 L_f^{2}} (\SN{\CR} - 2\sigma_\epsilon)^2 \leq  \Var{a^\top V|\CR}\leq \Var{a^\top X|\CR}  \leq \frac{3}{2}\sslen^2
\]
\end{lemma}
\begin{proof}
Note that \ref{ass:A5} and \ref{ass:A1} imply $\Covv{V,W|\CR} = 0$ and therefore
$\Var{\EUSP{a}{X}|\CR} = \Var{\EUSP{a}{V}|\CR} + \Var{\EUSP{a}{W}|\CR}$.
The upper bound follows from \eqref{eq:variance_bound_PW} and the fact that $\N{V-V'} \leq \sslen$ almost surely,
for an independent copy $V'$ of $V$, implies $\Var{a^\top V|\CR} \leq 1/2 \sslen^2$.

\noindent
For the lower bound it suffices to concentrate on $\Var{\EUSP{a}{V}|\CR}$.
We first use the identity $\bbE \SN{Z-\bbE[Z]}^2 = 1/2\bbE \SN{Z - Z'}^2$ ($Z'$ is an independent copy of $Z$) to get
\begin{equation}
\begin{aligned}
\label{eq:aux_cov_lower_bound_1}
\Var{\EUSP{a}{V}|\CR} & = \frac{1}{2}\bbE\left[\left(a^\top\left(V-V'\right)\right)^2 |\CR\right] = \frac{1}{2}\bbE[(t-t')^2 (a^\top \gamma'(t_\zeta))^2 |\CR] \\
&\geq \frac{1}{2} \min\limits_{s : \gamma(s) \in \CS} (a^\top \gamma'(s))^2\bbE[(t-t')^2|\CR] =  \min\limits_{s : \gamma(s) \in \CS}  (a^\top \gamma'(s))^2\bbE[(t-\bar t)^2|\CR].
\end{aligned}
\end{equation}
The first term is bounded from below by $\left\langle a, \gamma'(s)\right\rangle \geq 1 - \curvtor \sslen$.
For the second term, we fix $c > 0$ (is optimized later) and use Chebyshev's inequality to get
$\bbE[(t-\bar t)^2|\CR] \geq c^2\bbP(\SN{t-\bar t} > c | \CR)$.
Let now $\CI^{-}$ be any interval satisfying $\bbP(Y \in \CR | V \in \gamma(\CI^{-}))= 1$.
Then by using \ref{ass:A3} it follows that
\begin{align*}
\bbP(\SN{t - \bar t} > c |\CR) &= 1- \bbP(\SN{t - \bar t} \leq c |\CR) \geq 1 - \frac{\bbP(\SN{t - \bar t} \leq c)}{\bbP(Y \in \CR)}
\geq 1 - \frac{2 c {c_V}^2}{\SN{\CI^{-}}}.
\end{align*}
Optimizing now over $c$ we find $c= 1/3{c_V}^{-2}\SN{\CI^{-}}$ gives the bound $\bbE[(t-\bar t)^2|\CR] \geq 1/27 {c_V}^4 \SN{\CI^{-}}^2$
which implies that we ought to make $\CI^{-}$ as large as possible. Clearly, this is the case when
setting $\CI^{-} := \gamma^{-1} \circ f^{-1}([\inf \CR +\sigma_\epsilon, \sup \CR -\sigma_\epsilon])$ with
$\SN{\CI^{-}} > L_f^{-1}(\SN{\CR} - 2\sigma_\epsilon)$.
\end{proof}

\begin{lemma}
\label{lem:covariance_AXY}
Let \ref{ass:A5}, \ref{ass:A1}, and \ref{ass:A3} hold.
If $f \in \CC^2(\Omega)$ for $\Omega := \{tv +(1-t)\bar \gamma : t \in [0,1], v \in \suppmarg\}$ and the Hessian satisfies $\sup_{x \in \Omega}\N{\nabla^2f (x)} \leq L_H$
we have
\begin{align*}
\Var{a^\top X, Y|\CR} \geq \frac{(1-\kappa \sslen)^2}{27c_V^4 L_f^{3}} (\SN{\CR} - 2\sigma_\epsilon)^2 - \frac{1}{2}\sslen \sigma_{\varepsilon} - \frac{L_H}{2}\sslen^3.
\end{align*}
\end{lemma}
\begin{proof}
Assumptions  \ref{ass:A5} and \ref{ass:A1} imply $\bbE[W|Y] = 0$ and by the law of total covariance
\begin{equation}
\label{eq:cov_w_y_zero}
\Covv{W,Y|\CR} = \bbE_Y[\Covv{W,Y|Y} | \CR] + \Covvn{\bbE[W|Y], Y|\CR}{Y} = 0.
\end{equation}
Therefore we have $\Var{a^\top X, Y|\CR} = \Var{a^\top V, Y|\CR}$.
Furthermore, if $f \in \CC^2$ we can use the Taylor expansion of $f$ to rewrite for some $\zeta \in \bbR^D$
\[
f(V) - f(\bar \gamma) - (V-\bar \gamma)^\top \nabla f(\bar \gamma) = \frac{1}{2}(V-\bar \gamma)^\top \nabla^2 f(\zeta) (V-\bar \gamma).
\]
Using that $\nabla f$ is aligned with the tangent field of $\gamma$ (by choice
of the parametrization) we have $\nabla f(\bar  \gamma) = \N{\nabla f(\bar  \gamma)} a$ and we get
\begin{align*}
&\Var{a^\top V, f(V) |\CR} = \Var{a^\top V, f(V) - f(\bar \gamma) |\CR} \\
&\quad\quad\quad\quad=\Var{a^\top V, (V-\bar \gamma)^\top \nabla f(\bar \gamma)|\CR}  +
\frac{1}{2}\Var{a^\top V, (V-\bar \gamma)^\top \nabla^2 f(\zeta)(V-\bar \gamma)|\CR} \\
&\quad\quad\quad\quad\geq \N{\nabla f(\bar \gamma)}\Var{a^\top V|\CR} - \frac{L_H}{2}\sslen^3 \geq
L_f^{-1}\Var{a^\top V|\CR} - \frac{L_H}{2}\sslen^3.
\end{align*}
The result follows by Lemma \ref{lem:covariance_AV}, and $\Covv{a^\top V,\varepsilon |\CR}\leq \frac{1}{2}\sslen \sigma_{\varepsilon}$
which implies
\begin{align*}
\Var{a^\top V, Y |\CR}  \geq L_f^{-1}\Var{a^\top V|\CR} - \frac{1}{2}\sslen \sigma_{\varepsilon} - \frac{L_H}{2}\sslen^3.
\end{align*}
\end{proof}

\noindent
The last tool required for proving Theorem \ref{thm:midpoint_estimation} are the following concentration results for
mean and covariance estimation of bounded random variables.
\begin{lemma}
\label{lem:concentration_results}
Let $A\in \bbR^{d_A \times D}$ and $B \in \bbR^{d_B \times D}$, and assume $\N{A(X-\bbE X)} \leq C_A$,
$\N{B(X-\bbE X)} \leq C_B$ almost surely. Let $\hat{\bbE}_X$ be the sample mean,
and $\hat \Sigma$ the sample covariance from $N$ i.i.d. copies of $X$. For any $u > 0$,
we have
\begin{align}
\label{eq:directional_covariance}
&\bbP\left(\N{A(\bbE X - \smean{X})} \lesssim (1+u)C_A N^{-1/2}\right) \geq 1- \exp(-u),\\
&\bbP\left(\N{A\left(\Sigma -\hat{\Sigma}\right)B^\top} \lesssim C_A C_B(\log(D) + u) N^{-1/2}\right) \geq 1 - \exp(-u).
\end{align}
\end{lemma}
\begin{proof}
The first bound is a standard result that follows from the bounded differences inequality \cite{mcdiarmid1989method}.
For \eqref{eq:directional_covariance} denote $\tilde \Sigma = \smean{(X - \bbE X)(X - \bbE X)^\top}$
and decompose the error into
\[
\N{A\left(\Sigma - \hat{\Sigma}\right)B} \leq \N{A\left(\Sigma - \tilde{\Sigma}\right)B} + \N{A(\hat\bbE X - \bbE X)}\N{(\hat\bbE X - \bbE X)^\top B}.
\]
By the first result in \eqref{eq:directional_covariance} the second term is of order $\CO( C_A C_B N^{-1})$ with probability $1-2\exp(u)$,
and can thus be neglected. For the first term, denote
$S_k := \frac{1}{N} A \tilde  X_k \tilde X_k^\top B -\frac{1}{N} A \Sigma B$ and $S := \sum_{k=1}^{N}S_k$, where $\tilde X_k = X_k - \bbE X$.
Since $\bbE[\tilde X_k \tilde X_k^\top]=\Sigma$ we have $\bbE[S_k] = 0$, and since $\tilde X_k$ and $\tilde X_j$ are independent for $k\neq j$ we get $\bbE[S_kS_j^\top]= \bbE[S_k]\bbE[S_j^\top]= 0$. Thus,
\begin{align*}
\bbE[SS^\top]&=\sum_{k=1}^N \bbE[S_kS_k^\top] + \sum_{k\neq j}\bbE[S_kS_j^\top] = \sum_{k=1}^N \bbE[S_kS_k^\top]
\end{align*}
{Since $\N{S_k}\leq 2N^{-1}C_A C_B$ holds  almost surely we have $\N{\bbE SS^\top}\leq 4 N^{-1}C_A^2 C_B^2$ and by an analogous argument we have the same bound for $\N{\bbE S^\top S}$.}
Thus, the variance statistic (cf. Remark \ref{rem:matrix_bernstein_version}) satisfies $\sqrt{m(S)} \leq 2 N^{-1/2}C_A C_B$ and Theorem \ref{thm:matrix_bernstein} yields the desired result.
\end{proof}

\paragraph{Proof of Theorem \ref{thm:midpoint_estimation}}

We prove a more detailed version of Theorem \ref{thm:midpoint_estimation} given as follows.
\begin{theorem}
\label{thm:midpoint_estimation_restated}
Let \ref{ass:A5}, \ref{ass:A1}, \ref{ass:A41} and \ref{ass:A2} hold.
Let $u > 1$, $\CR \subset [0,1]$ be a closed interval with $\ssler > 4\sigma_{\varepsilon}$, and
$\CS \subset \Im(\gamma)$ the smallest segment such that $P(V \in \CS|Y \in \CR) = 1$.
Denote $\minSignal := \Var{a^\top X, Y|\CR}(\ssler \sslen)^{-1} > 0$, and
assume that for some $\alpha \geq 0$, $\sigmadiff \geq 2\minSignal \sslen^{2-\alpha}$
\begin{align}
\label{eq:sim_p_bound_assumption}
\N{\Covv{\Ptan X, \Pperp X | \CR}} \leq \curvtor \sigmadiff \sslen^{1+\alpha}.
\end{align}
Furthermore denote the scalars $B_{+} := B + \lengamma$,
\[
\summaryCrit := \left(1-\frac{(\curvtor\sigmadiff\sslen^{\alpha})^{2}}{\minSignal^2\minSigma}\right)^{-1},\quad
\summaryMax := (1 \vee B_{+}^2)\left(\frac{1}{\minSignal^2} \vee \frac{1}{\minSigma} \vee 1\right)
\]
There exists a universal constant $C$ such that whenever $\summaryCrit < 3$ and
\begin{equation}
\label{eq:thm_3_requirement_on_N}
N \geq \max\left\{C\frac{L_f^4 \summaryMax^4 (1+\minSigmaY^{-1}\kappa \sslen^2)^2(\log(D) + u)^2}{(3-\summaryCrit)^2},D\right\}
\end{equation}
we have with probability $1-\exp(-u)$
\begin{align}
\label{eq:thm3_restated_equation}
\N{\hat a - a} \lesssim L_f \frac{\curvtor\sigmadiff  \sslen^{\alpha}}{(3-\summaryCrit)\minSignal^2 \minSigma} \ssler
+ \frac{L_f\summaryMax^2 (1+\minSigmaY^{-1}\kappa \sslen^2)}{3-\summaryCrit} \frac{\log(D) + u}{\sqrt{N}}\ssler.
\end{align}
\end{theorem}
\begin{proof}[Proof of Theorem \ref{thm:midpoint_estimation} from Theorem \ref{thm:midpoint_estimation_restated}]
We apply Theorem \ref{thm:midpoint_estimation_restated} with $\sigmadiff^* = \sigmadiff \vee 3L_f\minSigmaYcool{j}\SN{\CR_j}
> 2\minSigmaYcool{j}\SN{\CS_j}$, where the second inequality follows from
$\SN{\CS_j} \leq 3/2L_f \SN{\CR_j}$ (Lemma \ref{lem:bound_curve_segment}) and $\SN{\CR_j }> 4\sigma_{\varepsilon}$.
Algebraic manipulation reveals that $\summaryCrit < 3$ is implied by the first condition in \eqref{eq:thm3_conds},
and \eqref{eq:thm_3_requirement_on_N} is implied by the second condition in \eqref{eq:thm3_conds}.
The result follows by $\SN{\CS_j} \leq 3/2L_f \SN{\CR_j} \leq 3/2L_f J^{-1}$.
\end{proof}

\noindent
The proof of Theorem \ref{thm:midpoint_estimation_restated} is given at the end of this
section because it requires a few tools that we develop first. Bringing forward a step of the proof
already now, we obtain the estimate
\begin{equation}
\label{eq:aux_thm_3_initial}
\N{\hat a - a} \leq   \sqrt{2} \frac{\Vert \Pperp b\Vert + \Vert \Pperp (\hat b - b)\Vert}{\N{\Ptan b} - \Vert \Ptan(\hat b - b)\Vert}.
\end{equation}
Thus, it suffices to bound $\Vert\Pperp b\Vert$, $\Vert\Ptan(\hat b-b)\Vert$ and
$\Vert\Pperp(\hat b-b)\Vert$ from above. In order to achieve optimal dependencies
of the bounds with respect to both $\sslen$ (or $\ssler$) and $N$,
we have to decompose $\Pperp b$, $\Ptan(\hat b-b)$ and $\Pperp(\hat b-b)$ into separate terms that
reflect how $\Sigma, \hat\Sigma, \Sigma^{\dagger}$ and $ \hat \Sigma^{\dagger}$ act on $b$ and $\hat b$.
This requires three tools: first we analyze spectral norms of $\Sigma^{\dagger}$
when paired with directions $\Ptan$, $\Pperp$ (Lemma \ref{lem:GI_pop}). Then
we need to bound perturbations $A(\hat \Sigma^{\dagger}-\Sigma^{\dagger})B$ for $A, B \in \{\Ptan, \Pperp\}$
to control the deviation of $\hat \Sigma^{\dagger}$ to $\Sigma^{\dagger}$ (Lemma \ref{lem:GI_sample}).
Finally, we need to analyze $r = \Covv{X,Y|\CR}$ since $b = \Sigma^{\dagger}r$, and similarly
we require concentration bounds of the finite sample counterpart $\hat r$ around $r$ (Lemma \ref{lem:right_hand_side_new}).
These results are then combined to prove Theorem \ref{thm:midpoint_estimation_restated}.

We begin by analyzing spectral bounds for  $\Sigma$.
It will be convenient to use $\minLambda:= 4\minSignal^2$ instead of $\minSignal$
since $\minLambda$ satisfies the relation $\Var{a^\top X|\CR} \geq \minLambda\sslen^2$
as we will see below.

\begin{lemma}
\label{lem:GI_pop}
If \eqref{eq:sim_p_bound_assumption}, \ref{ass:A41}, and $\summaryCrit < \infty$ hold we have
\begin{align}
\label{eq:GI_all}
\N{\Ptan\Sigma^{\dagger}\Ptan} &\leq \frac{\eta}{\minLambda \sslen^2},\quad
\N{\Pperp\Sigma^{\dagger} \Pperp} \leq \frac{\eta}{\minSigma},\quad
\N{\Ptan\Sigma^{\dagger}\Pperp}  \leq \frac{\eta \curvtor\sigmadiff}{\minLambda \minSigma \sslen^{1-\alpha}}.
\end{align}
\end{lemma}
\begin{proof}
Establishing \eqref{eq:GI_all} is challenging because the eigenspace of $\Sigma$
does not separate into eigenspaces related to $\Ptan$ and $\Pperp$. Instead, we have
to relate $\Sigma$ to the auxiliary matrix $\Sigma_P := \Covv{\Ptan X|\CR} + \Covv{\Pperp X|\CR}$.
Since we have
\[
\Sigma - \Sigma_{P} = \Covv{\Ptan X, \Pperp X | \CR} +  \Covv{\Pperp X, \Ptan X | \CR}
\]
Eqn. \eqref{eq:sim_p_bound_assumption} implies $\N{\Sigma - \Sigma_{P}} \leq 2\curvtor \sigmadiff \sslen^{1+\alpha}$,
which becomes small when $\sslen$ tends to $0$. Based on this observation we use the
following proof strategy: In the first step,
we show that $\Sigma$ and $\Sigma_P$ share the same range under the assumptions in the statement.
We can then derive the spectral decomposition of $\Sigma_P$ in the second step,
and use $\minSignal$ and $\minSigma$ from \ref{ass:A41} to bound spectral norms of  $\Sigma_P$.
In the third step we translate these bounds via perturbation theory to $\Sigma$.

\noindent
1. We show $\Im(\Sigma_P) = \Im(\Sigma)$.
First note that $\Im(\Sigma_P) = \Im(\Ptan \Sigma \Ptan ) \oplus \Im(\Pperp\Sigma \Pperp) \subset
\Im(\Ptan \Sigma) \oplus \Im(\Pperp\Sigma) = \Im(\Sigma)$, which implies that
it suffices to show $\rank(\Sigma_P) = \rank(\Sigma)$.
Since $\eta < \infty$ implies $\minSigmaY > 0$ and therefore $\Covv{\EUSP{a}{V}|\CR} > 0$, we have
$\rank(\Sigma_P) = \rank(\Ptan \Sigma_P \Ptan ) + \rank(\Pperp\Sigma_P \Pperp) = 1 + \rank(\Pperp\Sigma_P \Pperp)$.
To find a lower bound for $\rank(\Pperp\Sigma_P \Pperp)$, we note that, by \ref{ass:A41}, any unit norm $v \in \Im(\Sigma)\cap \Im(\Pperp)$
obeys
\[
v^\top \Pperp\Sigma_P \Pperp v = v^\top \Sigma_P v = v^\top \Sigma v > C_{\perp}.
\]
Therefore, $\rank(\Pperp\Sigma_P \Pperp) \geq \dim(\Im(\Sigma)\cap \Im(\Pperp))$. The result now follows by
$\dim(\Im(\Sigma)\cap \Im(\Pperp)) = \rank(\Sigma) - \dim(\Im(\Sigma)\cap \Im(\Ptan )) \geq \rank(\Sigma) - 1$.

\noindent
2. Denote $d = \rank(\Sigma) = \rank(\Sigma_P)$. By construction, the
eigendecomposition of $\Sigma_P$ is
\begin{equation*}
\Sigma_{P} = \Var{a^\top X|\CR} a a^\top + \sum\limits_{i=2}^{d} \sigma_iu_i u_i^\top,
\end{equation*}
where $\{u_2,\ldots,u_d\}$ is an eigensystem for $\Pperp\Sigma \Pperp$.
As $\Sigma_{P}^{\dagger}$ has the same eigen-decomposition with eigenvalues inverted, we have $\Ptan \Sigma_{P}^{\dagger}\Pperp = 0$.
Furthermore, $\Vert \Pperp \Sigma_{P}^{\dagger} \Pperp\Vert \leq 1/\minSigma$ follows by \ref{ass:A41}. For $\Ptan \Sigma_P^{\dagger}\Ptan$
using Popoviciu's inequality for the variance of the random variable $Y|\CR$ we get
\[
\Var{a^\top X|\CR} \geq \frac{\Var{a^\top X,Y|\CR}^2}{\Var{Y | \CR}} \geq 4\frac{\minSignal^2 \sslen^2 \ssler^2}{\ssler^2} = 4\minSignal^2 \sslen^2 = \minLambda\sslen^2,
\]
which implies $\Vert \Ptan \Sigma_P^{\dagger} \Ptan \Vert \leq (4\minSignal^2 \sslen^2)^{-1} = \minLambda^{-1}\sslen^{-2}$.

\noindent
3. Finally we transfer the bounds on $\Sigma_P$ to the true covariance matrix $\Sigma$. We use the shorthand $\Delta := \Sigma - \Sigma_P$.
We first note that $\Im(\Sigma) = \Im(\Sigma_P)$ implies the identity $\Sigma^{\dagger} =\Sigma_P^{\dagger} -\Sigma^{\dagger}\Delta \Sigma_P^{\dagger}$
by \cite{wedin1973perturbation}. Multiplying with $\Ptan, \Pperp$ in different combinations from left and right,
and using $\Ptan + \Pperp = \Id$, $P\Delta P = Q\Delta Q = 0$, and $\Ptan \Sigma_P^{\dagger} \Pperp = 0$ we obtain a system of equations
given by
\begin{align}
\label{eqn:aux_pop_dir_1}
\Pperp \Sigma^{\dagger} \Pperp &= \Pperp \Sigma_{P}^{\dagger} \Pperp - \Pperp \Sigma^{\dagger} \Ptan\Delta\Pperp\Sigma_{P}^{\dagger}\Pperp,\\
\label{eqn:aux_pop_dir_2}
\Ptan \Sigma^{\dagger}\Pperp &=-\Ptan  \Sigma^{\dagger}\Ptan\Delta\Pperp\Sigma_{P}^{\dagger}\Pperp,\\
\label{eqn:aux_pop_dir_3}
{\Ptan \Sigma^{\dagger} \Ptan} &={\Ptan\Sigma_{P}^{\dagger} \Ptan} - {\Ptan\Sigma^{\dagger} \Pperp}\Delta{\Ptan\Sigma_{P}^{\dagger}\Ptan}.
\end{align}
Consider now first $\Ptan\Sigma\Ptan$.  By plugging \eqref{eqn:aux_pop_dir_2} into \eqref{eqn:aux_pop_dir_3} and rearranging the terms, we get
\begin{align}
\label{eq:aux_ptan_ptan}
{\Ptan \Sigma^{\dagger} \Ptan}\LRP{\Id - \underbrace{\Delta\Pperp\Sigma_{P}^{\dagger}\Pperp\Delta \Ptan{\Sigma_{P}^{\dagger}\Ptan}}_{=: H}} &={\Ptan\Sigma_{P}^{\dagger} \Ptan}.
\end{align}
The matrix $H$ satisfies $\N{H} \leq 4\curvtor^2\sigmadiff^2\sslen^{2+2\alpha}/(4\minSignal^2 \sslen^2 \minSigma) = (\curvtor\sigmadiff\sslen^{\alpha})^2/( \minSignal^2 \minSigma) < 1$
under the condition $\summaryCrit < \infty$. Therefore the inverse of $\Id - H$ is explicitly given
by $\sum_{i=0}^{\infty}H^{k}$ by a von Neumann series argument. Using this and submultiplicativity
of the spectral norm we get
\begin{align*}
\N{{\Ptan \Sigma^{\dagger} \Ptan}} \leq \N{\sum_{i=0}^{\infty}(-H)^{k}}\frac{1}{\minLambda\sslen^2} \leq \frac{1}{1-\N{H}}\frac{1}{\minLambda \sslen^2}
\leq \frac{1}{1-\frac{(\curvtor\sigmadiff\sslen^{\alpha})^2}{\minSignal^2\minLambda}}\frac{1}{\minLambda\sslen^2} = \frac{\eta}{\minLambda\sslen^2}.
\end{align*}
By a symmetrie argument, we could have followed the same steps with $\Pperp$ instead, which immediately
implies the bound on $\Pperp \Sigma^{\dagger} \Pperp$. Finally, the bound on the cross term follows from \eqref{eqn:aux_pop_dir_2}
and using the bounds on $\Ptan \Sigma^{\dagger} \Ptan$, $\Delta$ and $\Pperp \Sigma_P^{\dagger}\Pperp$.
\end{proof}

We shall next bound $\Ptan(\hat \Sigma^{\dagger} - \Sigma^{\dagger})\Ptan$, $\Ptan(\hat \Sigma^{\dagger} - \Sigma^{\dagger})\Pperp$,
and $\Pperp(\hat \Sigma^{\dagger} - \Sigma^{\dagger})\Pperp$. This step is the most technical one because we need to keep close track of the dependencies
of $\hat \Sigma - \Sigma$ on directions they are evaluated in to achieve optimal
bounds with respect to both $N$ and $\sslen$. In particular, applying Lemma \ref{lem:concentration_results}
in conjunction with \eqref{eq:almost_sure_bounds_X_aX} in Lemma \ref{lem:technical_inequalities}, we have
with probability $1-3\exp(-u)$
\begin{align}
\label{eq:directional_covariance_bounds}
\N{A(\hat \Sigma - \Sigma)B} \lesssim
\begin{cases}
\sslen^2(\log(D) + u)N^{-1/2},\quad &\textrm{ if } A = B = \Ptan,\\
\sslen B_{+} (\log(D) + u)N^{-1/2},\quad &\textrm{ if } A = \Ptan, B = \Pperp,\\
B_{+}^2 (\log(D) + u)N^{-1/2},\quad &\textrm{ if } A = B = \Pperp.
\end{cases}
\end{align}
\begin{lemma}
\label{lem:GI_sample}
Assume \eqref{eq:sim_p_bound_assumption}, \ref{ass:A41}, \ref{ass:A2} and $\summaryCrit < \infty$.
Fix a confidence level $u > 0$. There exists a universal constant $C$ such that
whenever $N \geq \max\{C\summaryCrit^2 \summaryMax^2(\log(D) + u)^2, D\}$
we have with probability $1-\exp(-u)$ simultaneously
\begin{align}
\label{eq:GI_parpar_sam}
&\N{\Ptan \left(\hat\Sigma^{\dagger} -\Sigma^{\dagger}\right) \Ptan} \lesssim
\summaryCrit^2 \theta^2\frac{\log(D)+u}{\sslen^2\sqrt{N}},
\\
\label{eq:GI_nornor_sam}
&\N{\Pperp\left(\hat\Sigma^{\dagger} - \Sigma^{\dagger}\right)\Pperp} \lesssim
\summaryCrit^2 \summaryMax^2 \frac{\log(D)+u}{\sqrt{N}},
\\
\label{eq:GI_parnor_sam}
&\N{\Ptan\left(\hat\Sigma^{\dagger}-\Sigma^{\dagger}\right)\Pperp}  \lesssim
\summaryCrit^2 \summaryMax^2 \frac{\log(D)+u}{\sslen\sqrt{N}}.
\end{align}
\end{lemma}
\begin{proof}
We first note that we have $\Im(\hat \Sigma) = \Im(\Sigma)$ since $N \geq D$ and we assume that
$X| Y \in \CR$ is absolutely continuous with respect to $\Im(\Sigma)$, see Section \ref{sec:model}.
Now denote the shorthand $\Delta := \hat \Sigma - \Sigma$. From \cite{wedin1973perturbation} we obtain the identity
\[
\hat \Sigma^{\dagger} - \Sigma^{\dagger} =
(\hat \Sigma^{\dagger} - \Sigma^{\dagger})^\top = - \Sigma^{\dagger}\Delta \hat \Sigma^{\dagger} = - \Sigma^{\dagger}\Delta \Sigma^{\dagger} - \Sigma^{\dagger}\Delta (\hat \Sigma^{\dagger}-\Sigma^{\dagger}),
\]
and by using $P + Q = \Id$ and rearranging the terms, this implies
\begin{align}
\label{eq:MP_Inverse_initial_1}
\Ptan(\hat \Sigma^{\dagger} - \Sigma^{\dagger})\Ptan &= - \Ptan\Sigma^{\dagger}\Delta \Sigma^{\dagger}\Ptan - \Ptan\Sigma^{\dagger}\Delta \Ptan(\hat \Sigma^{\dagger}-\Sigma^{\dagger})\Ptan - \Ptan\Sigma^{\dagger}\Delta \Pperp(\hat \Sigma^{\dagger}-\Sigma^{\dagger})\Ptan,\\
\label{eq:MP_Inverse_initial_2}
\Pperp(\hat \Sigma^{\dagger} - \Sigma^{\dagger})\Ptan  &= - \Pperp\Sigma^{\dagger}\Delta \Sigma^{\dagger}\Ptan - \Pperp\Sigma^{\dagger}\Delta \Ptan(\hat \Sigma^{\dagger}-\Sigma^{\dagger})\Ptan - \Pperp\Sigma^{\dagger}\Delta \Pperp(\hat \Sigma^{\dagger}-\Sigma^{\dagger})\Ptan,\\
\label{eq:MP_Inverse_initial_3}
\Pperp(\hat \Sigma^{\dagger} - \Sigma^{\dagger})\Pperp &= - \Pperp\Sigma^{\dagger}\Delta \Sigma^{\dagger}\Pperp - \Pperp\Sigma^{\dagger}\Delta \Pperp(\hat \Sigma^{\dagger}-\Sigma^{\dagger})\Pperp - \Pperp\Sigma^{\dagger}\Delta \Ptan(\hat \Sigma^{\dagger}-\Sigma^{\dagger})\Pperp.
\end{align}
Considering only the first two equations, they contain two unknowns $\Ptan(\hat \Sigma^{\dagger}-\Sigma^{\dagger})\Ptan$ and $\Ptan(\hat \Sigma^{\dagger}-\Sigma^{\dagger})\Pperp$.
Hence we can solve for these unknowns by solving a linear system $SU = R$ with
\begin{align*}
U &= \left[
\begin{array}{c}
\Ptan(\hat \Sigma^{\dagger} - \Sigma^{\dagger})\Ptan\\
\hline \Pperp(\hat \Sigma^{\dagger} - \Sigma^{\dagger})\Ptan
\end{array},
\right] \in \bbR^{2D\times D},\\
R &= \left[
\begin{array}{c}
-\Ptan\Sigma^{\dagger}\Delta \Sigma^{\dagger}\Ptan \\
\hline
-\Pperp\Sigma^{\dagger}\Delta \Sigma^{\dagger}\Ptan
\end{array}
\right] \in \bbR^{2D\times D},\\
S &= \left[
\begin{array}{c|c}
\Id_D + \Ptan\Sigma^{\dagger}\Delta \Ptan & \Ptan \Sigma^{\dagger}\Delta \Pperp\\
\hline
\Pperp\Sigma^{\dagger}\Delta \Ptan & \Id + \Pperp\Sigma^{\dagger}\Delta \Pperp
\end{array}
\right] := \left[
\begin{array}{c|c}
S_{11} & S_{12}\\
\hline
S_{21} & S_{22}
\end{array}
\right] \in \bbR^{2D \times 2D}.
\end{align*}
It is well-known that, provided $S_{11}$ and $S_{22} - S_{21}S_{11}^{-1}S_{12}$ are invertible, the inverse of $S$
is precisely
\begin{align*}
S^{-1} = \left[
\begin{array}{c|c}
S_{11}^{-1} + S_{11}^{-1}S_{12}\left(S_{22} - S_{21}S_{11}^{-1}S_{12}\right)^{-1}S_{21}S_{11}^{-1} & -S_{11}^{-1}S_{12}\left(S_{22} - S_{21}S_{11}^{-1}S_{12}\right)^{-1} \\
\hline
-\left(S_{22} - S_{21}S_{11}^{-1}S_{12}\right)^{-1}S_{21}S_{11}^{-1} & \left(S_{22} - S_{21}S_{11}^{-1}S_{12}\right)^{-1}
\end{array}
\right].
\end{align*}
This allows to establish an identity for $\Pperp(\hat \Sigma^{\dagger}-\Sigma^{\dagger})\Ptan$ by known terms
after we have computed related entries of the inverse $S^{-1}$. This will be our first goal
in the following.

\noindent
Whenever $\N{\Ptan\Sigma^{\dagger}\Delta \Ptan} < 1$, we have $S_{11}^{-1} = \sum_{k=0}^{\infty}(-\Ptan\Sigma^{\dagger}\Delta \Ptan)^k$
using a von Neumann series argument.  Following the same argument, the matrix
$S_{22} - S_{21}S_{11}^{-1}S_{12} = \Id + \Pperp\Sigma^{\dagger}\Delta \Pperp - S_{21}S_{11}^{-1}S_{12}$
is invertible whenever, for $H := \Pperp\Sigma^{\dagger}\Delta \Pperp - S_{21}S_{11}^{-1}S_{12}$, we have $\N{H} < 1$.
In that case
\begin{align*}
&(S_{22} - S_{21}S_{11}^{-1}S_{12})^{-1} = \sum_{k=0}^{\infty}(-H)^k,\quad \textrm{and thus}\\
&-(S_{22} - S_{21}S_{11}^{-1}S_{12})^{-1}S_{21}S_{11}^{-1} = - \sum_{k=0}^{\infty}(-H)^k Q \Sigma^{\dagger}\Delta \Ptan \sum_{k=0}^{\infty}(-\Ptan\Sigma^{\dagger}\Delta \Ptan)^k,\\
&\textrm{implying } Q(\hat \Sigma^{\dagger} - \Sigma^{\dagger})\Ptan = \sum_{k=0}^{\infty}(-H)^k \Pperp \Sigma^{\dagger}\Delta \Ptan \sum_{k=0}^{\infty}(-\Ptan\Sigma^{\dagger}\Delta \Ptan)^k \Ptan\Sigma^{\dagger}\Delta \Sigma^{\dagger}\Ptan - \sum_{k=0}^{\infty}(-H)^k \Pperp\Sigma^{\dagger}\Delta \Sigma^{\dagger}\Ptan.
\end{align*}
Taking the supremum norm and using norm submultiplicativity it follows that
\begin{align}
\label{eq:aux_cross_term_1}
\N{Q(\hat \Sigma^{\dagger} - \Sigma^{\dagger})P} \leq \frac{ \N{Q \Sigma^{\dagger}\Delta P}}{1-\N{H}} \frac{\N{P\Sigma^{\dagger}\Delta \Sigma^{\dagger}P}}{1-\N{P\Sigma^{\dagger}\Delta P}}  + \frac{\N{Q\Sigma^{\dagger}\Delta \Sigma^{\dagger}P}}{1-\N{H}}.
\end{align}
Moreover, we can simplify leading factors in \eqref{eq:aux_cross_term_1} by estimating $\N{H}$. Specifically we find
\begin{align*}
\N{\sum_{k=0}^{\infty}(-P\Sigma^{\dagger}\Delta P)^k} &\leq \sum_{k=0}^{\infty}\N{P\Sigma^{\dagger}\Delta P}^k \leq \frac{1}{1-\N{P\Sigma^{\dagger}\Delta P}}\\
\textrm{which implies}\quad \N{H} &\leq \N{Q\Sigma^{\dagger}\Delta Q} + \frac{\N{Q\Sigma^{\dagger}\Delta P}\N{P\Sigma^{\dagger}\Delta Q}}{1 - \N{P \Sigma^{\dagger}\Delta P}},
\end{align*}
and therefore after algebraic manipulations we get
\begin{align}
\label{eq:aux_cross_term_2}
\frac{1}{1-\N{H}} \leq \frac{1}{(1-\N{\Pperp\Sigma^{\dagger}\Delta \Pperp})(1-\N{\Ptan \Sigma^{\dagger}\Delta \Ptan }) - \N{\Pperp\Sigma^{\dagger} \Delta \Ptan }\N{\Ptan \Sigma^{\dagger} \Delta \Pperp}}.
\end{align}
Having \eqref{eq:aux_cross_term_1} and \eqref{eq:aux_cross_term_2} established,
we now need to bound terms like $\Vert A\Sigma^{\dagger}\Delta B\Vert_2$
and $\Vert A\Sigma^{\dagger}\Delta \Sigma^{\dagger} B\Vert$ where $A, B \in \{P,Q\}$.
This ensures on one hand the invertibility of
$P\Sigma^{\dagger}\Delta P$ and $H$, and on the other hand bounds remaining terms in
\eqref{eq:aux_cross_term_1}. All bounds are achieved similarly by decomposing them
further and using the triangle inequality, e.g. to get
\begin{align*}
\N{ \Ptan \Sigma^{\dagger}\Delta \Ptan} \leq \N{\Ptan \Sigma^{\dagger}\Ptan}\N{\Ptan \Delta \Ptan} + \N{\Ptan \Sigma^{\dagger}\Pperp}\N{\Pperp \Delta \Ptan}.
\end{align*}
Then application of Lemma \ref{lem:GI_pop} and \eqref{eq:directional_covariance_bounds}
yields a concentration bound. For simplicity, we list the resulting bounds in Table \ref{tab:bounds_for_ugly_result}
below. They hold with probability at least $1-3\exp(-u)$.
\begin{table}[H]
\begin{center}
\scriptsize
\begin{tabular}{c|c|c}
Term & Bound multiplied with $C(\log(D) + u)N^{-1/2}$ & Shorthand notation\\
\toprule
$\N{ \Ptan \Sigma^{\dagger}\Delta \Ptan}$ & $\summaryCrit\left(\frac{1}{\minLambda} + \frac{B_+}{\sqrt{\minLambda\minSigma}} \right) \leq 2 \summaryCrit \summaryMax$ & $T_1$\\
$\N{ \Pperp \Sigma^{\dagger}\Delta \Pperp}$ & $\summaryCrit \left(\frac{B_{+}^2}{\minSigma} + \frac{B_+}{\sqrt{\minLambda\minSigma}}  \right)\leq 2 \summaryCrit \summaryMax$ & $T_2$\\
$\N{ \Pperp \Sigma^{\dagger}\Delta \Ptan}$ & $\summaryCrit \sslen \left(\frac{B_{+}}{\minSigma} + \frac{1}{\sqrt{\minLambda\minSigma}} \right)\leq 2 \summaryCrit \summaryMax \sslen $ & $T_3$\\
$\N{ \Ptan \Sigma^{\dagger}\Delta \Pperp}$ & $\summaryCrit \sslen^{-1} \left(\frac{B_{+} }{\minLambda} + \frac{B_{+}^2}{\sqrt{\minLambda\minSigma}}\right)\leq 2 \summaryCrit \summaryMax \sslen^{-1}$ & $T_4$\\
$\N{ \Ptan \Sigma^{\dagger}\Delta \Sigma^{\dagger} \Ptan}$ &
$\left(\frac{\summaryCrit}{\sslen}\right)^2 \left(\frac{1}{\minLambda} + \frac{B_{+}}{\sqrt{\minSigma \minLambda}} \right)^2 \leq 4 \left(\frac{\summaryCrit}{\sslen}\right)^2 \summaryMax^2$ & $T_5$\\
$\N{ \Pperp \Sigma^{\dagger}\Delta \Sigma^{\dagger} \Pperp}$ & $ \summaryCrit^2 \left(\frac{B_{+}}{\minSigma} + \frac{1}{\sqrt{\minSigma\minLambda}}\right)^2 \leq 4 \summaryCrit^2 \summaryMax^2$  & $T_6$\\
$\N{ \Ptan \Sigma^{\dagger}\Delta \Sigma^{\dagger} \Pperp}$ & $ \frac{\summaryCrit^2}{\sslen} \left(\frac{1}{\minLambda} +
\frac{B_{+}}{\sqrt{\minSigma \minLambda}}\right)\left(\frac{B_{+}}{\minSigma} + \frac{1}{\sqrt{\minSigma\minLambda}}\right)
\leq 4 \frac{\summaryCrit^2 \summaryMax^2}{\sslen}$ & $T_7 \leq \sqrt{T_5 T_6}$
\end{tabular}
\end{center}
\caption{Bounds for the perturbation terms based  Lemma \ref{lem:GI_pop} and
and \eqref{eq:directional_covariance_bounds}. $C$
is a universal constant.
Here we used that $\eta < \infty$ implies $\curvtor \sigmadiff \sslen^{\alpha}(\minLambda \minSigma)^{-1/2}\leq 1$ to simplify the bounds.
}
\label{tab:bounds_for_ugly_result}
\end{table}

\noindent
Now, let us first ensure the invertibilities of $P\Sigma^{\dagger}\Delta P$ and $H$
that was needed to derive \eqref{eq:aux_cross_term_1}. Since $T_3 T_4 \leq T_1 T_2$ Eqn.
\eqref{eq:aux_cross_term_2} becomes $(1-\N{H})^{-1} \leq ((1-T_1)(1-T_2) - T_1 T_2)^{-1}$
which is less than $1$ e.g. if $\max\{T_1, T_2\} < 1/2$. Thus it suffices to require
\begin{align*}
\max\{T_1, T_2\} \leq C'  \summaryCrit \summaryMax (\log(D) + u)N^{-1/2} < 1.
\end{align*}
This is ensured by the assumption $N \geq C \summaryCrit^2 \summaryMax^2 (\log(D) + u)^2$
and therefore $(1-\N{H})^{-1} \lesssim 1$, $\N{P\Sigma^{\dagger}\Delta P} \lesssim 1$. Combining this with \eqref{eq:aux_cross_term_1}
we then obtain
\begin{align*}
\N{Q(\hat \Sigma^{\dagger} - \Sigma^{\dagger})P} &\lesssim \N{Q \Sigma^{\dagger}\Delta P} \N{P\Sigma^{\dagger}\Delta \Sigma^{\dagger}P} + \N{Q\Sigma^{\dagger}\Delta \Sigma^{\dagger}P}\\
  &\lesssim  \frac{ \summaryMax^3 \summaryCrit^3 (\log(D) + u)^2}{\sslen N} + \frac{\summaryMax^2 \summaryCrit^2(\log(D) + u)}{\sslen \sqrt{  N}}
  \lesssim  \frac{\summaryMax^2 \summaryCrit^2(\log(D) + u)}{\sslen \sqrt{N}},
\end{align*}
where we used $N \geq C \summaryCrit^2 \summaryMax^2 (\log(D) + u)^2$ again to simplify higher order term. This proves
\eqref{eq:GI_parnor_sam}.

\noindent
The remaining two bounds are easier since we can use \eqref{eq:GI_parnor_sam}.
For \eqref{eq:GI_parpar_sam} we recall \eqref{eq:MP_Inverse_initial_1}
and $\N{\Ptan\Sigma^{\dagger}\Delta \Ptan} < 1$ (whenever $N \geq C \summaryCrit^2 \summaryMax^2 (\log(D) + u)^2$) to get
\begin{align*}
\Ptan(\hat \Sigma^{\dagger} - \Sigma^{\dagger})\Ptan &=
\left(\Id + \Ptan\Sigma^{\dagger}\Delta \Ptan\right)^{-1}\left(- \Ptan\Sigma^{\dagger}\Delta \Sigma^{\dagger}\Ptan  - \Ptan\Sigma^{\dagger}\Delta \Pperp(\hat \Sigma^{\dagger}-\Sigma^{\dagger})\Ptan\right).
\end{align*}
Then, expressing the inverse by a von Neumann series and using $(1-\N{\Ptan \Sigma^{\dagger}\Delta \Ptan})^{-1} \lesssim 1$
we get
\begin{align*}
\N{\Ptan(\hat \Sigma^{\dagger} - \Sigma^{\dagger})\Ptan} &\leq \frac{\N{\Ptan\Sigma^{\dagger}\Delta \Sigma^{\dagger}\Ptan}}{1 - \N{\Ptan\Sigma^{\dagger}\Delta \Ptan}}
 + \frac{\N{\Ptan\Sigma^{\dagger}\Delta \Pperp(\hat \Sigma^{\dagger}-\Sigma^{\dagger})\Ptan}}{1 - \N{\Ptan\Sigma^{\dagger}\Delta \Ptan}} \\
 &\lesssim \N{\Ptan\Sigma^{\dagger}\Delta \Sigma^{\dagger}\Ptan}
  + \N{\Ptan\Sigma^{\dagger}\Delta \Pperp}\N{\Pperp(\hat \Sigma^{\dagger}-\Sigma^{\dagger})\Ptan}\\
&\lesssim \left(\frac{\summaryCrit \summaryMax}{\sslen}\right)^2\frac{(D + u)}{\sqrt{N}}
 + \frac{\summaryCrit \summaryMax(\log(D)+u)}{\sslen \sqrt{N}}\frac{\summaryMax^2 \summaryCrit^2(\log(D) + u)}{\sslen \sqrt{N}}
 \leq \summaryCrit^2 \summaryMax^2 \frac{(D + u)}{\sslen^2\sqrt{N}},
\end{align*}
where we used again $N \geq C \summaryCrit^2 \summaryMax^2 (\log(D) + u)^2$ to simplify the higher order term.
\eqref{eq:GI_nornor_sam} follows similarly by starting from \eqref{eq:MP_Inverse_initial_3}.
\end{proof}

\noindent
It remains to analyze the cross-covariance term $r = \Covv{X,Y | Y \in \CR}$,
and bounding its concentration when estimated from a finite data set.

\begin{lemma}
\label{lem:right_hand_side_new}
Assume \ref{ass:A5}, \ref{ass:A1}. For $r = \Covv{X,Y| \CR}$
we have $\N{P r} = \minSignal \sslen \ssler$ and $\N{\Pperp r}\leq 1/2 \kappa \sslen^2 \ssler$.
Furthermore, let now $\{(X_i,Y_i) : i \in [N]\}$ denote $N$ iid. copies of $(X,Y)$, and
denote $\hat r = N^{-1}\sum_{i=1}^{N}(X_i - \smean{X_i})(Y_i - \smean{Y_i})$.
Then we have for $u > 1$ concentration results
\begin{align*}
\bbP\left(\N{\Ptan(\hat r - r)} \lesssim \frac{u \sslen \ssler}{\sqrt{N}}\right) \geq 1 - \exp(-u),\textrm{ and }\quad \bbP\left(\N{\hat r - r} \lesssim \frac{u B_{+} \ssler}{\sqrt{N}}\right) \geq 1 - \exp(-u).
\end{align*}
\end{lemma}
\begin{proof}
$\N{P r} = \minSignal \sslen \ssler$ is precisely the definition of $\minSignal$
in Theorem \ref{thm:midpoint_estimation_restated}. For $Qr$ we first recall $\Covv{W,Y|\CR} = 0$
as in \eqref{eq:cov_w_y_zero}.
Therefore, we can write $\Pperp r = \Pperp\Covv{X,Y|\CR} = \Pperp \Covv{V,Y|\CR}$ which satisfies
by \eqref{eq:variance_bound_PV} in Lemma \ref{lem:technical_inequalities}
\begin{align*}
\N{\Pperp \Covv{V,Y|\CR}} &\leq \sqrt{\N{\Covv{\Pperp V|\CR}} \N{\Covv{Y|\CR}}} \leq 1/2 \curvtor \sslen^2 \ssler.
\end{align*}
For the concentration results, we denote $Z_i := (X_i - \bbE X)(Y_i - \bbE Y) - \Covv{X, Y}$, and let $A \in \{\Ptan, \Id\}$.
We can decompose the error as
\[
A(r - \hat r) = \smean{AZ_i} + (\smean{AX_i} - \bbE AX)(\bbE Y - \smean{Y_i}),
\]
and notice that, by Lemma \ref{lem:concentration_results}, the second term is always of higher order.
For the first term, we have $\bbE A Z_i = 0$, and
\begin{align*}
\N{A Z_i} &\leq \N{A(X_i- \bbE X)}\N{Y_i - \bbE Y} + \sqrt{\bbE\N{Y-\bbE Y}^2} \sqrt{\bbE\N{A(X - \bbE X)}^2} \leq 2 C_A \ssler,
\end{align*}
where $\N{A(X - \bbE X)} \leq C_A$  almost surely. Using \eqref{eq:almost_sure_bounds_X_aX} in
Lemma \ref{lem:technical_inequalities},
we can choose $C_A = 2 \sslen$ if $A = \Ptan$, and $C_A = B_{+}$ if $A = \Id$.
The results follows from \eqref{eq:directional_covariance} in Lemma \ref{lem:concentration_results}.
\end{proof}

\begin{proof}[Proof of Theorem \ref{thm:midpoint_estimation_restated}]
The proof is divided into three steps. First we use previously established Lemmata
\ref{lem:GI_pop}, \ref{lem:GI_sample}, and \ref{lem:right_hand_side_new} to provide
concentration bounds for $\Vert P(\hat b - b)\Vert$ and $\Vert Q(\hat b - b)\Vert$, where we recall $b = \Sigma^{\dagger}r$
and $\hat b = \hat \Sigma^{\dagger}\hat r$. Then we establish that the bound \eqref{eq:aux_thm_3_initial} is indeed true
under the conditions of the Theorem. Finally, we use the concentration bounds on
$\Vert P(\hat b - b)\Vert$ and $\Vert Q(\hat b - b)\Vert$ together with a bound
on $\Vert Qb\Vert$ to conclude the result.

\noindent
1. Let us begin with $\Vert \Ptan(\hat b - b)\Vert$. We first decompose the error into
\begin{align}
\label{eq:nsim_aux_final_decomposition_1}
\Ptan(\hat b - b) =
\Ptan(\hat \Sigma^{\dagger} \hat r -  \Sigma^{\dagger} r) =
 \Ptan(\hat \Sigma^{\dagger}  - \Sigma^{\dagger})r + \Ptan(\hat \Sigma^{\dagger}  - \Sigma^{\dagger})(\hat r - r) + \Ptan\Sigma^{\dagger} (\hat r -  r).
\end{align}
Now we apply Lemma \ref{lem:GI_pop}, \ref{lem:GI_sample}, and \ref{lem:right_hand_side_new} to bound these terms.
The second term has higher order and is thus neglected. For the first term we get
with probability $1-2\exp(-u)$
\begin{align*}
&\N{\Ptan(\hat \Sigma^{\dagger}  - \Sigma^{\dagger})r} \leq \N{\Ptan(\hat \Sigma^{\dagger}  - \Sigma^{\dagger})\Ptan}\N{\Ptan r} +
\N{\Ptan(\hat \Sigma^{\dagger}  - \Sigma^{\dagger})\Pperp}\N{\Pperp r}\\
&\lesssim  \frac{\summaryMax^2(\log(D) + u)}{\sqrt{N}} \frac{\minSignal\sslen \ssler}{\sslen^2}
+ \frac{\summaryMax^2(\log(D) + u)}{\sqrt{N}}\curvtor \frac{\sslen^2\ssler}{\sslen}
\lesssim L_f \summaryMax^2 \minSignal\left(1 + \frac{\kappa \sslen^{2}}{\minSigmaY}\right)\frac{\log(D) + u}{\sqrt{N}}
\end{align*}
where we used  $\ssler/\sslen \lesssim L_f$ since $\sslen \geq L_f^{-1}(\ssler - 2\sigma_{\varepsilon})$ by Lemma \ref{lem:bound_curve_segment}, and
$\ssler > 4\sigma_{\varepsilon}$.
For the third term in \eqref{eq:nsim_aux_final_decomposition_1} we have with probability $1-2\exp(-u)$
\begin{align*}
\N{P\Sigma^{\dagger} (\hat r -  r)} &\leq \N{P\Sigma^{\dagger}P}\N{P(\hat r -  r)} + \N{P\Sigma^{\dagger}Q}\N{Q(\hat r -  r)}\\
&\lesssim \frac{1}{\sigma_{\parallel}\sslen^2}\frac{u\sslen \ssler}{\sqrt{N}} + \frac{\curvtor \sigmadiff}{\minLambda\minSigma \sslen^{1-\alpha}} \frac{uB_{+}\ssler}{\sqrt{N}}
\lesssim L_f  \summaryMax \frac{u}{\sqrt{N}},
\end{align*}
where we used that $\eta < \infty$ implies $\curvtor\sigmadiff\sslen^{\alpha}/(\minLambda\minSigma) \leq 1$.
Since $\summaryMax^2 \minSignal \geq \summaryMax \max\{1,\minSigmaY^{-2}\}\minSignal \geq \summaryMax$ the bound for the third term
is dominated by the bound on $\Vert \Ptan(\hat \Sigma^{\dagger}  - \Sigma^{\dagger})r\Vert$,
and thus we get with probability $1-4\exp(-u)$
\begin{equation}
\label{eq:bound_final_parallel}
\N{\Ptan(\hat b - b)}\lesssim L_f \summaryMax^2 \minSignal\left(1 + \frac{\kappa \sslen^{2}}{\minSigmaY}\right)\frac{\log(D) + u}{\sqrt{N}}.
\end{equation}

\noindent
The same strategy is used for $\Pperp (\hat b - b)$. First we decompose into three terms
\begin{align*}
\Pperp(\hat b - b)  =\Pperp(\hat \Sigma^{\dagger}  - \Sigma^{\dagger})r + \Pperp(\hat \Sigma^{\dagger}  - \Sigma^{\dagger})(\hat r - r) + \Pperp\Sigma^{\dagger} (\hat r -  r),
\end{align*}
and notice that the second term is of higher order. The first term is bounded by
\begin{align*}
&\N{\Pperp(\hat \Sigma^{\dagger}  - \Sigma^{\dagger})r} \leq \N{\Pperp(\hat \Sigma^{\dagger}  - \Sigma^{\dagger})\Pperp}\N{\Pperp r}
+ \N{\Pperp(\hat \Sigma^{\dagger}  - \Sigma^{\dagger})\Ptan}\N{\Ptan r}\\
&\lesssim \frac{\summaryMax^2(\log(D) + u)}{\sqrt{N}} \curvtor \sslen^2 \ssler
+ \frac{ \summaryMax^2(\log(D) + u)}{\sqrt{N}} \frac{\minSignal\sslen\ssler}{\sslen}
\leq  \summaryMax^2 \minSignal\left(1 + \frac{\kappa \sslen^{2}}{\minSigmaY} \right)\frac{\log(D) + u}{\sqrt{N}}\ssler,
\end{align*}
and for the third summand we get
\begin{align*}
\N{\Pperp\Sigma^{\dagger} (\hat r -  r)} &\leq \N{\Pperp\Sigma^{\dagger}\Pperp}\N{\Pperp(\hat r -  r)} + \N{\Pperp\Sigma^{\dagger}\Ptan}\N{\Ptan(\hat r -  r)}\\
&\lesssim \frac{1}{\minSigma} \frac{uB_{+}\ssler}{\sqrt{N}} + \frac{\curvtor \sigmadiff}{\minSigma\minLambda \sslen^{1-\alpha}}\frac{u\sslen\ssler}{\sqrt{N}}
\lesssim  \summaryMax \frac{u }{\sqrt{N}}\ssler.
\end{align*}
As before the first term dominates and thus we have with
probability $1-4\exp(-u)$
\begin{equation}
\label{eq:bound_final_perp}
\N{\Pperp(\hat b - b)} \lesssim  \summaryMax^2 \minSignal\left(1 +\frac{\kappa \sslen^{2}}{\minSigmaY}\right)\frac{\log(D) + u}{\sqrt{N}}\ssler.
\end{equation}

\noindent
2. Next we prove the error decomposition \eqref{eq:aux_thm_3_initial}. This first requires to ensure $a^\top \hat b > 0$ (Step 2.1).

\noindent
2.1 We first note that the definition $b = \Sigma^{\dagger} r$ implies $r = \Sigma b$. Rewriting $a^\top r$ we get
\begin{align}
\label{eq:bound_ab_1}
a^\top r =a^\top \Sigma b = a^\top \Sigma a a^\top b + a^\top \Sigma \Pperp b\quad  \textrm{and thus}\quad a^\top b
\geq \frac{a^\top r - \N{P\Sigma \Pperp}\N{\Pperp b}}{a^\top \Sigma a} .
\end{align}
Furthermore using Lemma \ref{lem:GI_pop}, \ref{lem:right_hand_side_new} and $\sigmadiff \geq 2\minSignal \sslen^{2-\alpha}$, $\minLambda = 4\minSignal^2$ we can bound $\Vert \Pperp b \Vert$ by
\begin{equation}
\begin{aligned}
\label{eq:thm_3_bound_Qb}
\Vert \Pperp b\Vert &\leq \Vert \Pperp \Sigma^{\dagger} \Pperp \Vert \Vert \Pperp r\Vert + \Vert \Pperp \Sigma^{\dagger} \Ptan \Vert \Vert \Ptan r\Vert \leq
\frac{\summaryCrit \curvtor}{2\minSigma}\sslen^2 \ssler + \frac{\summaryCrit\curvtor\sigmadiff}{4\minSignal\minSigma}\sslen^{\alpha} \ssler\\
&\leq \frac{\summaryCrit \curvtor}{2\minSigma}\left(\sslen^{2-\alpha} + \frac{\sigmadiff}{2\minSignal}\right)\sslen^{\alpha}\ssler
\leq \frac{\summaryCrit\sigmadiff \curvtor}{2\minSignal\minSigma}\sslen^{\alpha}\ssler.
\end{aligned}
\end{equation}
Plugging this, $a^\top r = \Var{a^\top X, Y|\CR} = \minSignal\sslen \ssler$, $a^\top \Sigma a = \Var{a^\top X|\CR} \leq 2\sslen^2$
(Lemma \ref{lem:technical_inequalities}),
and $\N{\Ptan\Sigma \Pperp} \leq \curvtor \sigmadiff\sslen^{1+\alpha}$ into
\eqref{eq:bound_ab_1}, we obtain
\begin{align*}
2a^\top b &\geq \minSignal \frac{\ssler}{\sslen} - \frac{\curvtor \sigmadiff \sslen^{1+\alpha}}{\sslen^2} \frac{\summaryCrit \sigmadiff \curvtor}{2\minSignal \minSigma} \sslen^{\alpha}\ssler
= \minSignal \frac{\ssler}{\sslen} - \frac{\summaryCrit}{2} \frac{\curvtor^2 \sigmadiff^2 \sslen^{2\alpha}}{\minSignal\minSigma} \frac{\ssler}{\sslen}
\geq \frac{\minSignal (3-\summaryCrit)}{4L_f}
\end{align*}
where $\ssler/\sslen \geq 1/(2L_f)$ by Lemma \ref{lem:bound_curve_segment} in the last inequality.
By the requirement $\summaryCrit < 3$ it follows that $a^\top b > 0$. We can transfer the lower boundedness
to the estimate $a^\top \hat b$ by
\begin{align*}
a^\top \hat b \geq a^\top b - \N{\Ptan (b - \hat b)} \geq \frac{\minSignal(3-\summaryCrit)}{8L_f} -  C L_f \summaryMax^2 \minSignal\left(1 + \frac{\kappa \sslen^{2}}{\minSigmaY}\right)\frac{\log(D) + u}{\sqrt{N}}
\end{align*}
with probability $1-4\exp(-u)$, and where $C$ is some universal constant. Using the condition \eqref{eq:thm_3_requirement_on_N} that bounds $N$ from below
$a^\top \hat b > 0$ with probability $1-4\exp(-u)$.

\noindent
2.2 Now we can prove decomposition \eqref{eq:aux_thm_3_initial}.
First notice that Pythagoras gives $\N{\hat a - a}^2 = \N{\Ptan \hat a - a}^2 + \N{\Pperp \hat a}^2$.
Furthermore since $a^\top \hat b > 0$, we can rewrite $a =  \Vert \Ptan \hat b\Vert ^{-1} \Ptan \hat b$
to get
\begin{align*}
\N{\Ptan \hat a - a}^2 = \N{\frac{\Ptan \hat b}{\Vert \hat b\Vert} - \frac{\Ptan \hat b}{\Vert \Ptan \hat b\Vert}}^2
= \N{\Ptan\hat b}^2 \left(\frac{1}{\Vert\hat b\Vert} - \frac{1}{\Vert \Ptan \hat b\Vert}\right)^2
= \left(\frac{\Vert \Ptan \hat b\Vert - \Vert \hat b\Vert}{\Vert\hat b\Vert} \right)^2
\leq \frac{ \Vert \Pperp \hat b\Vert^2}{\Vert\hat b\Vert^2 },
\end{align*}
where we used the triangle inequality in the last step. Therefore, we get
$\N{\hat a - a}^2 \leq \N{\Ptan \hat a - a}^2 + \N{\Pperp \hat a}^2 \leq 2 \Vert \Pperp \hat b\Vert^2 \Vert\hat b\Vert^{-2}$
which implies
\begin{align}
\label{eq:aux_thm_3_initial_later}
\N{\hat a - a} &\leq \sqrt{2} \frac{\Vert \Pperp b\Vert + \Vert \Pperp (\hat b - b)\Vert}{\N{\Ptan b} - \Vert \Ptan(\hat b - b)\Vert}.
\end{align}

\noindent
3. In this final step we combine \eqref{eq:aux_thm_3_initial_later} with the other results of steps
1 and 2. First we notice that the denominator in \eqref{eq:aux_thm_3_initial_later}
is bounded from below by $1/16\minSignal (3-\summaryCrit)L_f^{-1}$ by choosing the universal $C$ in the
requirement \eqref{eq:thm_3_requirement_on_N} large enough. $\Vert \Pperp b\Vert$ is bounded
as in \eqref{eq:thm_3_bound_Qb}, and for $\Vert \Pperp (\hat b - b)\Vert$ we use
the concentration bound \eqref{eq:bound_final_perp}.
\end{proof}

\subsubsection{Global analysis}
\label{subsubsec:appendix_global_analysis}
In this part we analyze the global error of approximating the tangent field by
proving Corollary \ref{cor:tangent_bound_all}. The result can be established quickly from
Theorem \ref{thm:midpoint_estimation} once we ensure that each level set contains sufficiently many samples.
Indeed this is the case under \ref{ass:A3} as shown in the following Lemma.

\begin{lemma}
\label{lem:net_property}
Let \ref{ass:A3} hold, and let $\{X_i : i \in [N]\}$ be $N$ i.i.d. copies of $X$. For $0 < u < N$ we have
\begin{equation}\label{eqn:dg_net}
\bbP\left(\{V_i : i \in [N]\} \textrm{ is a }\frac{\lengamma u}{c_V N}\textrm{-net wrt. } d_{\gamma}(\cdot,\cdot) \right) \geq 1 - \exp(-u).
\end{equation}
Furthermore if $\{\CX_j : j \in [J]\}$ and $\{\CY_j : j \in [J]\}$ is a partition according to \eqref{eq:dyadic_partitioning} for
some $J^{-1} > 4\sigma_{\varepsilon}$ and $N > \frac{8L_f \lengamma u}{c_V }J$
we have
\begin{equation}\label{eqn:num_local_samples}
\bbP\left(\min_{j \in J}\SN{\CX_j} \geq \frac{1}{4L_f}\frac{c_V }{\lengamma u } \frac{N}{J} \right) \geq 1 - \exp(-u).
\end{equation}
\end{lemma}

\begin{proof}
Let $\epsilon = \frac{\lengamma u}{c_V N}$, and $V \in \Im(\gamma)$.
Since \ref{ass:A3} implies $\bbP\left(V' \in \CB_{d_{\gamma}}(V,\varepsilon)\right) > c_V\varepsilon \lengamma^{-1}$, where $V'$ is an independent copy of $V$, we have
\begin{align*}
\bbP\Big(\{V_i : i \in [N]\} \textrm{ is a }\frac{\lengamma u}{c_V N}\textrm{-net w.r.t.}\, d_{\gamma}\Big) &= 1 - \bbP\LRP{\exists V: (\forall i) \,V\not\in \CB_{d_{\gamma}}(V_i,\varepsilon)}\\&= 1 - \prod_{i=1}^N \LRP{1 - \bbP(V \in \CB_{d_{\gamma}}(V_i,\varepsilon)}\geq 1 - \exp(-u).
\end{align*}
For the second statement let $j \in [J]$ arbitrary and denote $\CR_j = [a_j,b_j]$, $\CR_j^{-} = [3/4a_j + 1/4 b_j, 1/4a_j + 3/4b_j]$.
Then, since $J^{-1} = \SN{\mathcal{\CR}_j} > 4\sigma_{\varepsilon}$ we have $\bbP(Y \in \CR_j | f(X) \in \CR_{j}^{-}) = 1$, and thus there exists a segment
$\CS_j \subset \Im(\gamma)$ with $\SN{\CS_j} \geq 1/2 L_f^{-1}\SN{\CR_j} =1/2 L_f^{-1}J^{-1}$
such that $\bbP(Y \in \CR_j | V \in \CS_j) = 1$. The result follows from
\begin{align*}
\bbP\left(\min_{j \in J}\SN{\CX_j} \geq \frac{1}{4L_f}\frac{c_V }{\lengamma u } \frac{N}{J} \right) \geq
\bbP\left(\min_{j \in J}\SN{\CX_j} \geq \frac{1}{2JL_f}\frac{c_V N}{\lengamma u} - 2\right)
\geq \bbP\left(\left\{V_i\right\}_{i=1}^{N} \textrm{ is a } \frac{\lengamma u}{c_V N}\textrm{-net}\right),
\end{align*}
where we used $N > \frac{8L_f \lengamma u}{c_V }J$ to simplify the bound on $\min_{j \in [J]}\SN{\CX_j}$ in the first inequality.
\end{proof}

\begin{proof}[Proof of Corollary \ref{cor:tangent_bound_all}]
Let us first check whether the conditions of Theorem \ref{thm:midpoint_estimation} are satisfied for each $j \in [J]$.
Clearly, \eqref{eqn:global_otherconds} implies \eqref{eq:thm3_conds} for all $j \in [J]$. Furthermore the number of
samples satisfies with probability exceeding $1-\exp(-u)$ by Lemma \ref{lem:net_property}
\begin{align*}
\min_{j \in [J]} \SN{\CX_j} \geq \frac{\uniformC}{4L_f\lengamma}\frac{N}{uJ} \gtrsim \max\left\{C_N (\log(D) + \log(J)u)^2, D\right\} =: N_{\textrm{LB}}.
\end{align*}
Thus, $\SN{\CX_j}$ satisfies \eqref{eq:thm3_conds} for $u\log(J)$ instead of $u$ for all $j \in [J]$ as soon as $C_N$ is equal to $C_N$ in Theorem \ref{thm:midpoint_estimation} multiplied by $4L_f \lengamma \uniformC^{-1}$. Denote now $e_j := \Vert \hat a_j - a_j\Vert$.
Using Theorem \ref{thm:midpoint_estimation} and the union bound we obtain
\begin{align*}
&\bbP\left(\max_{j \in [J]}e_j \leq \frac{C_A \kappa}{J} + C_E \frac{\log(D)u + \log(J)u^2}{\sqrt{N J}}\right)
\geq\bbP\left(\max_{j \in [J]}e_j \leq \frac{C_A \kappa}{J} + \tilde C_E \frac{\log(D) + \log(J)u}{\sqrt{\SN{\CX_j}}J}\right)\\
&\quad \geq \bbP\left(\max_{j \in [J]}e_j \leq \frac{C_A \kappa}{J^2} + \tilde C_E \frac{\log(D) + \log(J)u}{\sqrt{\SN{\CX_j}}J} \bigg| \min_{j \in [J]}\SN{\CX_j} \geq N_{\textrm{LB}}\right)
\bbP\left(\min_{j \in [J]} \SN{\CX_j} \geq N_{\textrm{LB}} \right)\\
&\quad \geq (1-J\exp(-\log(J)u))(1-\exp(-u)) = (1-\exp(-u))^2 \geq 1-\exp(-u),
\end{align*}
where $\tilde C_E$ equals $C_E$ in Theorem \ref{thm:midpoint_estimation} up to factors depending on $L_f, \uniformC, \lengamma$.
The result follows by using \eqref{eq:decomposition_tangent_error} and defining $C_A$
as the maximum of $C_A$ in Theorem \ref{thm:midpoint_estimation} and $\SN{\CS_j}\leq 2L_f$.
\end{proof}

\subsection{Proofs for Section \ref{sec:function_estimation}}
\label{subsec:nsim_app_proofs_sec_4}
\subsubsection{Proofs for Section \ref{subsec:function_estimation}}
\label{subsubsec:nsim_app_proofs_sec_41}
Almost linear curves allow to find an equivalent characterization of the geodesic metric
using projections onto the tangent field. This is made precise in the following
Lemma and is a key ingredient to establish the metric equivalency in Proposition \ref{prop:intrinsic_bound_for_approximated_closest_neighbors}.

\begin{lemma}
\label{lem:almost_linearity}
Let $\gamma : \FI \rightarrow \bbR^D$ be a $\theta$-almost linear curve. Then for ${t}' \geq t$ and $\tilde t \textrm{ arbitrary}$
\begin{align*}
\theta d_{\gamma}\,(\gamma(t),\gamma({t}')) &\leq \left\langle  \gamma'(\tilde t), \gamma({t}') - \gamma(t)\right\rangle \leq d_{\gamma}(\gamma(t), \gamma({t}')).
\end{align*}
\end{lemma}
\begin{proof}
The upper bound follows by Cauchy-Schwartz, $\N{\gamma(t)-\gamma({t}')} \leq d_{\gamma}(\gamma(t),\gamma({t}'))$ and $\N{\gamma'(\tilde t)} = 1$. For the lower bound
the fundamental theorem of calculus gives
\[
\left\langle  \gamma'(\tilde t), \gamma({t}') - \gamma(t)\right\rangle = \left\langle  \gamma'(\tilde t),\int\limits_{t}^{{t}'}  \gamma'(s)ds \right\rangle
=\int\limits_{t}^{{t}'} \left\langle  \gamma'(\tilde t), \gamma'(s) \right\rangle ds \geq \theta \left({t}' - t\right) = \theta d_{\gamma}(\gamma(t),\gamma({t}')).
\]
\end{proof}

\begin{proof}[Proof of Proposition \ref{prop:intrinsic_bound_for_approximated_closest_neighbors}]
We begin with an intermediate result.
\label{lem:intrinsic_metric_approx}
Let $x = v + w,\ x' = v' + w' \in \suppmarg$, where $v=\pi_\gamma(x)$ and $v'=\pi_\gamma(x')$, and let $S(v,v') \subset \Im(\gamma)$ be the curve segment
between $v$ and $v'$.
Assume $\gamma|_{S(v,v')}$ is $\theta$-almost linear for $\theta > \kappa(S(v,v'))\boundnormal$. We will show that for arbitrary $p\in\bbR^D$ we have
\begin{equation}
\begin{aligned}
\label{eq:delta_estimate}
&\frac{\SN{\left\langle p, x-x'\right\rangle} -  \N{x-x'} \N{p-a(v')}}{1 + \kappa(S(v,v')) \boundnormal } \leq d_{\gamma}(v,v')
\leq \frac{\SN{\left\langle p, x-x'\right\rangle} + \N{x - x'}\N{p-a(v')}}{\theta - \kappa(S(v,v'))\boundnormal}.
\end{aligned}
\end{equation}
For the first inequality we have $\SN{\left\langle p, x - x' \right\rangle} \leq \N{x - x'} \N{p - a(v')} + \SN{\left\langle a(v'), x - x' \right\rangle}$, by Cauchy-Schwartz.
The fundamental theorem of calculus and $a(v) \perp w$, $a(v') \perp w'$ then yield
\begin{align*}
\SN{\left\langle a(v'), x - x' \right\rangle}&\leq \SN{\left\langle a(v'), v - v'\right\rangle} + \SN{\left\langle a(v') - a(v), w\right\rangle}
\leq  d_{\gamma}(v,v') +  \kappa\LRP{S(v,v')} \boundnormal d_{\gamma}(v, v').
\end{align*}
where we used Lemma \ref{lem:almost_linearity} in the last step. The bound follows after dividing by $1 + \kappa\LRP{S(v,v')}\boundnormal$.
For the second inequality in \eqref{eq:delta_estimate} using Lemma \ref{lem:almost_linearity}, and again the fact that $w' \perp a(v')$, we get
\begin{align*}
\theta d_{\gamma}(v,v') &< \SN{\left\langle v - v', a(v')\right\rangle} \leq  \SN{\left\langle x - x', a(v')\right\rangle} + \SN{\left\langle w, a(v') \right\rangle} \\
&\leq \SN{\left\langle x - x', p\right\rangle} + \N{x - x'}\N{p - a(v')} + \kappa\LRP{\CS(v,v')}\boundnormal d_{\gamma}(v,v')
\end{align*}
Collecting terms with $d_{\gamma}(v, v')$ and dividing through by $\theta -  \kappa\LRP{\CS(v,v')} \boundnormal$ yields the desired bound.
Denote now for short $d := \lengamma + 2\boundnormal$. Eqn. \eqref{eq:delta_estimate} implies in the context of Proposition \ref{prop:intrinsic_bound_for_approximated_closest_neighbors}
\begin{equation}
\label{eq:metric_equivalency}
\frac{\intMetMod{x}{\bar x_i}{\infty} - d\N{\hat a(\bar x_i) - a(\bar x_i)}}{2} \leq d_{\gamma}(v, V_i)
\leq \frac{\intMetMod{x}{\bar x_i}{\infty} + d\N{\hat a(\bar x_i) - a(\bar x_i)}}{\theta - \kappa B}
\end{equation}
since  $\kappa(S(v,v')) \leq \kappa$ and
\[\N{x-\bar x_i}\leq \N{v-\bar v_i} + \N{w-\bar w_i}\leq \SN{\CI} + 2\boundnormal = d, \text{ and }  1 +  \kappa\LRP{\CS(v,v')}\boundnormal \leq 2.\]

\noindent
We will now use  \eqref{eq:metric_equivalency} to establish Proposition \ref{prop:intrinsic_bound_for_approximated_closest_neighbors}.
Using the left hand side of \eqref{eq:metric_equivalency} we get
\[
\intMetMod{x}{\bar x_{k(x)}}{\infty} \leq \max_{i=1,\ldots,k}\intMetMod{x}{\bar x_{k^*(x)}}{\infty} \leq
\max_{i=1,\ldots,k} 2 d_{\gamma}(x, \bar x_{k^*(x)}) + d \max_{i \in [N]}\N{\hat a(\bar x_i) - a(\bar x_i)},
\]
where $\max_{i=1,\ldots,k} 2 d_{\gamma}(x, \bar x_{k^*(x)})  = 2d_{\gamma}(x, \bar x_{k^*(x)})$ by the definition of $k^*(x)$.
Then, using the right hand side of \eqref{eq:metric_equivalency}, the result follows by
\begin{align*}
d_{\gamma}(x, \bar x_{k(x)}) &\leq
\frac{\intMetMod{x}{\bar x_{k(x)}}{\infty} + d\N{\hat a(\bar x_i) - a(\bar x_i)}}{\theta - \kappa B}  \leq 2\frac{d_{\gamma}(x, \bar x_{k^*(x)}) + d\N{\hat a(\bar x_i) - a(\bar x_i)}}{\theta - \kappa B}.
\end{align*}
\end{proof}

\subsubsection{Proofs for Section 4.2}
\label{subsubsec:nsim_app_proofs_sec_42}
The proof of Proposition
\ref{prop:local_metric_equivalence_general} is more involved
than for Proposition \ref{prop:intrinsic_bound_for_approximated_closest_neighbors}
and requires two auxiliary results that will be developed first.
The first result states that for any $x \in \suppmarg$ with $v := \pi_{\gamma}(x)$ for which there exist
another $v' \in \Im(\gamma)$ that satisfies the condition $a(v')(x-v') = 0$ (\emph{i.e.} $x$ lies
in the normal ray of $\gamma$ at $v'$), we necessarily have a minimum distance $\N{x-v'}$.
The second result uses this observation to ensure
equivalence of $d_{\gamma}(x,\cdot)$ and $\intMetMod{x}{\cdot}{\tempradius}$
under suitable conditions on $\tempradius$.
We also notice that $\intMetMod{x}{\bar x_i}{\infty} \leq 2 d_{\gamma}(x, \bar x_i) + (\lengamma + 2\boundnormal)\N{\hat a(\bar x_i) - a(\bar x_i)}$,
which has been proven in \eqref{eq:metric_equivalency}, remains
valid and will be used also here.

\begin{lemma}
\label{lem:general_curve_ball_auxiliary}
Assume $x\in\bbR^D$ has a unique projection $v:=\pi_{\gamma}(x)$, satisfying $\N{x - v} \leq\boundnormal < \reach$.
For any $v' \neq v \in \Im(\gamma)$
with $\EUSP{a(v')}{(x - v')} = 0$ we have $\N{x - v'} \geq 2 \reach -\boundnormal$.
Furthermore for any $x'$ with $\N{x'-v'} \leq \boundnormal < \reach$ and $\pi_{\gamma}(x') = v'$ we have
$\N{x-x'} \geq 2(\reach - \boundnormal)$.
\end{lemma}
\begin{proof}
First note that by the properties of $\reach$ we know that for all $z\in\bbR^D$, such that $\dist{\Im(\gamma)}{z}< \reach$, there is only one $v_z\in\Im(\gamma)$ such that
$\EUSP{a(v_z)}{(z-v_z)} = 0$ and $\N{z - v_z}<\reach$ {\cite[Sec. 4]{niyogi2008finding}}.
Thus, $\N{x-v'}\geq \reach$.
Moreover, for the line $W(t) = v' + ts$, where $s = (x -v')/\N{X - v'}$, we have $\dist{\Im(\gamma)}{W(t)}=\N{W(t)-v'}=t$, for all $t\in(0,\reach)$ and $\dist{\Im(\gamma)}{W(t)}=\reach$ holds for at least one $t^*\in[\reach,\N{x-v'})$.

\noindent
We now want to show that $\N{W(t^*) -x}\geq \reach -\boundnormal$.
Assume the contrary. Then
\[\N{W(t^*) -v} \leq  \N{W(t^*) -x}  + \N{x -v} <\reach,\]
which contradicts $\dist{\Im(\gamma)}{W(t)}=\reach$.
Since $W(t^*)$ lies on a line between $v'$ and $x$ we have
\begin{align*}
\N{x - v'} &= \N{x - W(t^*)} + \N{W(t^*)-v'}\geq 2 \reach -\boundnormal.
\end{align*}

\noindent
The second statement follows from  $\N{x-x'} \geq \N{x-v'} - \N{v'-x'} \geq 2(\reach - \boundnormal)$.
\end{proof}

\begin{lemma}
\label{lem:local_metric_equivalence_general}
Assume \ref{ass:A2} for $\boundnormal = (1/2-q)\reach$ for some $q > 0$.
Let $x \in \suppmarg$ arbitrary, $\bar x \in \suppmarg \cap \CB_{\N{\cdot}}(x, \reach)$
with tangent approximation $\hat a(\bar x)$. If
\begin{equation}
\label{eq:local_metric_equivalence_general_assumption_lemma}
\intMetMod{x}{\bar x}{\reach} < \left(q-\N{\hat a(\bar x) - a(\bar x)}\right)\reach
\end{equation}
we have $d_{\gamma}(x, \bar x) \leq 4\intMetMod{x}{\bar x}{\reach} + 4\reach \N{\hat a(\bar x) - a(\bar x)}$.
\end{lemma}
\begin{proof}
Let $v := \pi_{\gamma}(x), \bar v := \pi_{\gamma}(\bar x)$, $\omega = \SN{\hat a(\bar x)^\top (x-\bar x)}$
and consider the point $\tilde x := \bar v + Q(\bar x)(x - \bar v)$, where $Q(\bar x) := \Id - a(\bar x)a(\bar x)^\top$.
The point $\tilde x$ satisfies $a(\bar x)^\top(\tilde x - \bar v) = 0$ and, since $a(\bar x) \perp \bar x - \bar v$, it is contained
within a small ball around $x$ bounded by
\begin{align*}
\N{x-\tilde x} = \SN{a(\bar x)^\top(x-\bar x)} \leq \SN{\hat a(\bar x)^\top(x-\bar x)} + \SN{(a(\bar x)-\hat a(\bar x))^\top(x-\bar x)} \leq \omega + \reach \varepsilon_a.
\end{align*}
This also that $\tilde x$ itself is not too far from $\Im(\gamma)$ because using the triangle inequality we get
\[
\dist{\tilde x}{ \Im(\gamma)} \leq \N{\tilde x - x} + \N{x - v} \leq \omega + \reach \varepsilon_a + \boundnormal.
\]
By $\omega + \reach \varepsilon_a + \boundnormal < q\reach + \boundnormal \leq 1/2\reach \leq 1/2\reach$,
it follows that $\tilde x$ has a unique projection  $\tilde v := \pi_{\gamma}(\tilde x)$.
From now, the proof follows two steps. We first show $\pi_{\gamma}(\tilde x) = \bar v$ by contradiction, which is then used
for bounding $d_{\gamma}(x, \bar x)$.

\noindent
1. Assume $\pi_{\gamma}(\tilde x) \neq \bar v$. We have constructed $\tilde x$
with $a(\bar x)^\top(\tilde x - \bar v) = 0$ and $\N{\tilde x - \tilde v} \leq \omega + \reach \varepsilon_a + B$.
Lemma \ref{lem:general_curve_ball_auxiliary} immediately implies the lower bound
\begin{align*}
\N{\tilde x - \bar v} \geq 2\reach - \omega - \reach \varepsilon_a - B = (2 - \varepsilon_a)\reach - \omega - B.
\end{align*}
Using then $\bar x \in \CB_{\N{\cdot}}(x, \reach)$, $\Vert x - \bar x\Vert \geq \Vert \tilde x - \bar v\Vert - \Vert x - \tilde x\Vert  - \Vert \bar v - \bar x\Vert$
from the triangle inequality, and $\boundnormal = (1/2-q)\reach$, we have the inequality
\begin{align*}
\reach &\geq (2 - \varepsilon_a)\reach - \omega - \boundnormal -  \omega - \reach \varepsilon_a - \boundnormal
= (2 - 2\varepsilon_a)\reach - 2\omega - 2\boundnormal = (1 - 2\varepsilon_a)\reach - 2\omega + 2q\reach.
\end{align*}
This implies with $\omega \geq (q-\varepsilon_a) \reach$ a contradiction to Condition \eqref{eq:local_metric_equivalence_general_assumption_lemma}.

\noindent
2. Using first $\max\{\N{x-v}, \N{\tilde x - \tilde v}\} \leq \omega + \reach \varepsilon_a + \boundnormal$ and the Lipschitz-property
of $\pi_{\gamma}$ (see \cite[Theorem 4.8 (8)]{federer1959curvature}), and then   $\omega < (q-\varepsilon_a)\reach$, we get
\begin{align*}
\N{v-\bar v} = \N{\pi_{\gamma}(x) - \pi_{\gamma}(\tilde x)} \leq \frac{\N{x - \tilde x}}{1-\frac{\omega + \reach \varepsilon_a + \boundnormal}{\reach}}
< \frac{\N{x - \tilde x}}{1-\frac{q\reach + (1/2-q)\reach}{\reach}} \leq 2(\omega  + \reach \varepsilon_a).
\end{align*}
Furthermore since $\omega < (q-\varepsilon_a) < (1/4-\varepsilon_a)\reach$ we have $\N{v - \bar v} \leq 2(\omega  + \reach \varepsilon_a) < \reach/2$,
and thus we can apply \cite[Proposition 6.3]{niyogi2008finding} to get
\begin{align*}
d_{\gamma}(v, \bar v)\leq \reach - \reach \sqrt{1-\frac{2\N{v-\bar v}}{\reach}}
\leq \N{v - \bar v} + \frac{2\N{v - \bar v}^2}{\reach}\leq 2\N{v-\bar v} \leq 4(\omega  + \reach \varepsilon_a).
\end{align*}
\end{proof}

\noindent
\begin{proof}[Proof of Proposition \ref{prop:local_metric_equivalence_general}]
Define $\varepsilon_a := \max_{i \in [N]}\N{\hat a(\bar x_i) - a(\bar x_i)}$ and note that $k\delta < (\tempradius - 2\boundnormal)$ implies $\N{x - \bar x_{i^*(x)}} \leq k\delta + 2\boundnormal < \tempradius$, hence
$\bar x_{i^*(x)} \in \CB_{\N{\cdot}}(x, \tempradius)$ for all $i \in [k]$.
This similarly implies $\{\bar x_{i(X)} : i \in [k]\} \subset B_{\N{\cdot}}(x,\tempradius)$, and by
using the left hand side of \eqref{eq:metric_equivalency} we get the bound
\begin{align}
\label{eq:general_curve_ball_auxiliary_2}
\intMetMod{x}{\bar x_{k(x)}}{\tempradius} &\leq \max_{i \in [k]}\intMetMod{x}{\bar x_{k^*(x)}}{\tempradius} \leq 2 d_{\gamma}(x, \bar x_{k^*(x)}) + (\lengamma + 2\boundnormal)\varepsilon_a \\
&\leq 2\delta k + (\lengamma + 2\boundnormal)\varepsilon_a < (q -\varepsilon_a)\reach.
\end{align}
By Lemma \ref{lem:local_metric_equivalence_general} we get $d_{\gamma}(x, \bar x_{k(x)})\leq 4\intMetMod{x}{\bar x_{k(x)}}{\tempradius} + 4\reach \varepsilon_a$ and
the result follows from
\begin{align*}
d_{\gamma}(x, \bar x_{k(x)}) \leq4\intMetMod{x}{\bar x_{k(x)}}{\tempradius} + 4\reach \varepsilon_a \leq 8d_{\gamma}(x, \bar x_{k^*(x)}) + 4(\lengamma + 2\boundnormal)\varepsilon_a + 4\reach \varepsilon_a.
\end{align*}
\end{proof}

\subsection{Referenced results}
\label{subsec:referenced_results}
\begin{theorem}[Matrix Bernstein, 6.1.1. in \cite{Tropp}]
\label{thm:matrix_bernstein}
Consider a finite sequence $ S_k$ of independent, random matrices, with common dimension $d_1\times d_2$ and assume that
\( \bbE[ S_k] = \boldsymbol 0,\) and \(\N{ S_k}\leq L,\, \forall k.\)
Define the random matrix $ S=\sum_{k=1}^N  S_k$, and the matrix variance statistic
\begin{equation}\label{eq:nu} m( S) = \max\LRP{\N{\bbE[ S S^\top]},\N{\bbE[ S^\top S]}}.\end{equation}
Then for all $\epsilon \geq 0$ we have the tail bound
\begin{equation}\label{eq:tail_bound}
\bbP\left(\N{ S} \geq \epsilon \right) \leq \left(d_1 + d_2\right) \exp\left(-\frac{\epsilon^2}{2\left(m( S) + L \epsilon/3\right)}\right).
\end{equation}
\end{theorem}
\begin{remark}
\label{rem:matrix_bernstein_version}
Let us make a short comment regarding Theorem \ref{thm:matrix_bernstein}. Jensen's inequality gives
\[ m( S)  \leq \bbE\max \N{ S S^\top},\N{ S^\top S}=\bbE\N{ S}^2.\]
Hence, it is sufficient to bound $\bbE\N{ S}^2$.
Moreover, \eqref{eq:tail_bound} holds if we replace $m( S)$ with its upper bound $\mu\geq m( S)$.
Rewriting now the right hand side of \eqref{eq:tail_bound} as
\[ \exp\LRP{ \log(d_1+d_2)-\frac{\epsilon^2}{2\LRP{\nu + L \epsilon/3}} }=:\exp(-u),\]
for $u >0$, leads to a quadratic equation for $\epsilon$, the solution of which is given as
\begin{equation}\label{eq:tail_bound_use1} \epsilon = \frac{1}{3}\LRP{\sqrt{ L^2\LRP{u + \log(d_1+d_2)}^2 +18\nu(u +\log(d_1+d_2)) } +L(u  + \log(d_1+d_2))}.\end{equation}
Algebraic manipulation shows that this can be bounded by $\epsilon\leq C\max\LRP{L,\sqrt{\nu}} \LRP{u + \log(d_1+d_2)}$
for some universal constant $C>0$.
Finally, monotonicity of probability gives $\bbP\LRP{\N{ S}\geq \epsilon} \geq\bbP\LRP{\N{ S}\geq \epsilon'}$ for $\epsilon\leq\epsilon'$.
Thus, for every $u>0$
\begin{equation}
\label{eq:tail_bound_final}
\begin{aligned}
\bbP\LRP{\N{ S}\leq C\max\LRP{L,\sqrt{\nu}} \LRP{u + \log(d_1+d_2)} } \leq 1- \exp(-u).
\end{aligned}
\end{equation}
\end{remark}
\end{document}

%% file: Macros.tex
\providecommand*{\Covv}[1]{\operatorname{Cov}\left({#1}\right)}   
\providecommand*{\Covvn}[2]{\operatorname{Cov}_{#2}\left({#1}\right)}   
\providecommand*{\Uni}[1]{\operatorname{Uni}({#1})}
\providecommand*{\Var}[1]{\operatorname{Var}\left({#1}\right)}   

\providecommand*{\len}[1]{\SN{{#1}}}   
\newcommand\independent{\protect\mathpalette{\protect\independenT}{\perp}}
\def\independenT#1#2{\mathrel{\rlap{$#1#2$}\mkern2mu{#1#2}}}

\providecommand{\Dim}{\operatorname{dim}}            
\providecommand{\dim}{\Dim}

\providecommand*{\dist}[2]{\operatorname{dist}({#1};{#2})}   

\providecommand{\rank}{\operatorname{rank}}                        

\renewcommand{\Im}{\operatorname{Im}}             
\providecommand{\argmin}{\operatorname*{argmin}}  
\providecommand{\Id}{\Op{Id}}                     

\providecommand{\CB}{{\cal B}}
\providecommand{\CC}{{\cal C}}
\providecommand{\CD}{{\cal D}}

\providecommand{\CI}{{\cal I}}

\providecommand{\CM}{{\cal M}}

\providecommand{\CO}{{\cal O}}

\providecommand{\CR}{{\cal R}}
\providecommand{\CS}{{\cal S}}

\providecommand{\CU}{{\cal U}}

\providecommand{\CX}{{\cal X}}
\providecommand{\CY}{{\cal Y}}

\providecommand{\bbE}{\mathbb{E}}

\providecommand{\bbN}{\mathbb{N}}

\providecommand{\bbP}{\mathbb{P}}

\providecommand{\bbR}{\mathbb{R}}
\providecommand{\bbS}{\mathbb{S}}

\providecommand{\FI}{\mathfrak{I}}

\newcommand*{\EUSP}[2]{\left<{#1},{#2}\right>} 

\providecommand*{\N}[1]{\left\|{#1}\right\|} 
\newcommand*{\SN}[1]{\left|{#1}\right|}      

\newcommand*{\LRP}[1]{\left(#1\right)}

\newcommand*{\Op}[1]{\mathsf{#1}} 

